\documentclass[a4paper,11pt]{amsart}

\usepackage{amsmath,amsthm,amscd,amssymb,amsfonts, amsbsy}
\usepackage{latexsym}
\usepackage{txfonts}
\usepackage[colorlinks,citecolor=red,pagebackref,hypertexnames=false]{hyperref}
\numberwithin{equation}{section}\usepackage{graphicx}
\usepackage{epstopdf}
\usepackage{geometry}                
\numberwithin{equation}{section}

\theoremstyle{plain}
\newtheorem{theorem}[equation]{Theorem}
\newtheorem{lemma}[equation]{Lemma}
\newtheorem{corollary}[equation]{Corollary}
\newtheorem{proposition}[equation]{Proposition}

\theoremstyle{definition}
\newtheorem{definition}[equation]{Definition}

\theoremstyle{remark}
\newtheorem{remark}[equation]{Remark}

\newcommand{\fiint}{\operatornamewithlimits{\fint\!\!\!\!\fint}}
\newcommand{\dv}{\operatorname{div}}

\newcommand{\supp}{\operatorname{supp}}
\newcommand{\dist}{\operatorname{dist}}
\newcommand{\loc}{\operatorname{loc}}

\newcommand{\tr}{\operatorname{tr}}

\newcommand{\V}{\operatorname{V}}
\newcommand{\la}{\Lambda_{\alpha}}

\newcommand{\CC}{\mathbb{C}}

\newcommand{\RR}{\mathbb{R}}
\newcommand{\NN}{\mathbb{N}}
\newcommand{\reu}{\mathbb{R}^{n+1}_+}
\newcommand{\rel}{\mathbb{R}^{n+1}_-}
\newcommand{\ree}{\mathbb{R}^{n+1}}
\newcommand{\rn}{\mathbb{R}^n}

\newcommand{\N}{{\widetilde{\mathcal{N}}}}

\newcommand{\C}{\mathcal{C}}
\newcommand{\A}{\mathcal{A}}
\newcommand{\SL}{\mathcal{S}}

\newcommand{\Le}{L_\epsilon}

\newcommand{\K}{\mathcal{K}}
\newcommand{\D}{\mathcal{D}}
\newcommand{\T}{\mathcal{T}}
\newcommand{\rr}{\mathcal{R}}

\newcommand{\tT}{\widetilde{T}}
\newcommand{\tN}{\widetilde{N}}
\newcommand{\tB}{\widetilde{B}}
\newcommand{\tR}{\widetilde{R}}
\newcommand{\w}{\widetilde{w}}

\newcommand{\eps}{\varepsilon}

\newcommand{\rec}{R^c_{1/\epsilon}}
\newcommand{\recs}{R^{c,i}_{1/\epsilon}}
\newcommand{\recb}{R^{c,e}_{1/\epsilon}}

\newcommand{\lv}{\lvert}
\newcommand{\rv}{\rvert}
\newcommand{\lp}{\left(}
\newcommand{\rp}{\right)}

\newcommand{\bp}{\noindent {\it Proof}.\,\,}
\newcommand{\ep}{\hfill$\Box$ \vskip 0.08in}

\usepackage{boxedminipage}

\begin{document}	

\title[Layer potentials and boundary problems]{Layer potentials and boundary value problems for elliptic equations with complex $L^{\infty}$ coefficients satisfying the small Carleson measure norm condition}

\author{Steve Hofmann}	
\address{ Steve Hofmann\\
Department of Mathematics\\ University of Missouri, Columbia, MO 65211  USA}
\email{hofmanns@missouri.edu }
\author{Svitlana Mayboroda}

\address{Svitlana Mayboroda
\\
School of Mathematics
\\
University of Minnesota
\\
Minneapolis, MN, 55455  USA} \email{svitlana@math.umn.edu}
\author{Mihalis Mourgoglou}
\address{Mihalis Mourgoglou\\
Department of Mathematics\\ University of Missouri, Columbia, MO 65211  USA}
\curraddr{Univ. Paris-Sud 11, Laboratoire de Math\'ematiques, 
UMR 8628, F-91405 
{\sc ORSAY}; FMJH, F-91405 {\sc ORSAY}}
\email{mihalis.mourgoglou@math.u-psud.fr}

\begin{abstract} We consider divergence form elliptic equations $Lu:=\nabla\cdot(A\nabla u)=0$
in the half space
$\mathbb{R}^{n+1}_+ :=\{(x,t)\in \mathbb{R}^n\times(0,\infty)\}$, whose coefficient matrix $A$
is complex elliptic, bounded and measurable.   In addition, we suppose that
$A$ satisfies some additional regularity
in the direction transverse to the boundary, namely that the discrepancy 
$A(x,t) -A(x,0)$ satisfies a Carleson measure
condition of Fefferman-Kenig-Pipher type, with small Carleson norm.
Under these conditions, we establish a full range of boundedness results for double and single layer potentials in $L^p$, Hardy, Sobolev, BMO and H\"older spaces. Furthermore, we prove 
solvability of the Dirichlet problem for $L$,
with data in $L^p(\rn)$, $BMO(\rn)$, and $C^\alpha(\rn)$, 
 and solvability of the Neumann and Regularity problems, with data in the spaces $L^p(\rn)/H^p(\rn)$ and $L^p_1(\rn)/H^{1,p}(\rn)$ respectively, with the appropriate restrictions on indices,
assuming invertibility of layer potentials in for the $t$-independent operator $L_0:= -\nabla\cdot(A(\cdot,0)\nabla)$.
\end{abstract}


\maketitle

\tableofcontents

\allowdisplaybreaks[1]

\section{Introduction and Statements of Results}\label{s1}
We  consider boundary value problems for divergence form elliptic equations $Lu=0$, where
\begin{equation} L=-\nabla\cdot(A\nabla u)=-\sum_{i,j=1}^{n+1}\frac{\partial}{\partial
x_{i}}\left(A_{i,j} \,\frac{\partial}{\partial x_{j}}\right),
\label{eq1.1}\end{equation}
 and $A=A(x,t)$ is an $(n+1)\times(n+1)$ matrix of complex-valued $L^{\infty}$ coefficients, defined on $\mathbb{R}^{n+1}_+ :=\{(x,t)\in \mathbb{R}^n\times(0,\infty)\}$ and satisfying the uniform ellipticity conditions
\begin{equation}
\lambda|\xi|^{2}\leq\Re e\,\langle A(x,t)\xi,\xi\rangle
\equiv \Re e\sum_{i,j=1}^{n+1}A_{ij}(x,t)\xi_{j}\bar{\xi_{i}}\,\,\,\text{and}\quad
 \Vert A\Vert_{L^{\infty}(\mathbb{R}^{n+1})}\leq\Lambda,
\label{eq1.2}\end{equation}
for some $\lambda>0$, $\Lambda<\infty$ and for all $\xi\in\mathbb{C}^{n+1}$, $x\in\mathbb{R}^{n}$ and $t>0$. The divergence form equation is interpreted in the usual weak sense, i.e., we say that $Lu=0$ in a domain $\Omega$ if $u\in W^{1,2}_{loc}(\Omega)$ and 
$$\int A \nabla u \cdot \overline{\nabla \Psi} = 0$$ 
for all complex valued $\Psi \in C_0^\infty(\Omega)$.  

The goal of this paper is two-fold. First, we develop a comprehensive theory of layer potentials for elliptic operator with bounded measurable complex coefficients satisfying the Small Carleson Measure Condition. Secondly, we establish the well-posedness for the corresponding boundary value problems  assuming the well-posedness of the associated boundary value problems for operators with $t$-independent coefficients, $A=A(x,0)$, or, to be more precise, the invertibility of the boundary layer potentials. 

The aforementioned Small Carleson Measure Condition can be defined as follows. Let $A^1(x,t)=A(x,t)$ be a complex $(n+1) \times (n+1)$ coefficient matrix, satisfying the ellipticity condition \eqref{eq1.2}, and let $A^0(x,t)=A^0(x):=A(x,0)$.   Set
\begin{equation}\label{eq1.3}
\epsilon(x,t):=\sup\left\{\left|A^1(y,s)-A^0(y,s)\right|: \, (y,s)\in W(x,t)\right\} 
\end{equation}
where $W(x,t) \equiv \Delta(x,t) \times (t/2, 3t/2)$ and $\Delta(x,t) \equiv \left\{y \in \rn: \left| x-y \right| < t\right\}$, and assume that $\epsilon=\epsilon(x,t)$ satisfies
a  Carleson measure condition
\begin{equation}\sup\limits_{\text{cube}\;Q\subset\rn} \displaystyle{\left( \frac{1}{|Q|} \iint\limits_{R_Q} 
\epsilon^2(x,t)\frac{dxdt}{t}\right)^\frac{1}{2}} < {\varepsilon}_0,
\label{eq1.4} \end{equation} 
where $R_Q:=Q\times(0, \ell(Q))$ is a Carleson box and $\ell(Q)$ is the side-length of the surface cube $Q$.
We say that $L$ satisfies the Small Carleson Measure Condition if $\eps_0$ above is sufficiently small (depending on parameters which will be specified with the corresponding Theorems).
At this point we just mention that the condition \eqref{eq1.4} arises very naturally in this context, as the forthcoming discussion of history of the problem and known counterexamples will reveal.

Let us start by introducing some notation. Recall that the scale of $L^p$ spaces is naturally extended to $p\leq 1$ by Hardy classes $H^p(\rn)$. Similarly, the scale of homogeneous Sobolev spaces $\dot L^p_1(\rn)$, $1<p<\infty$, extends to $H^{1,p}(\rn)$, the Sobolev-Hardy spaces of tempered distributions whose first order distributional derivatives are in $H^p$. As customary, we shall write $H^p(\rn)$ (resp., $H^{1,p}(\rn)$) for all $0<p<\infty$, with the understanding that $H^p=L^p$ (resp., $H^{1,p}=\dot L_1^p$) when $1<p<\infty$. Furthermore, let us denote 
\begin{equation}
\la(\rn)=\begin{cases} 
BMO(\rn),& \text{if}\,\,\, \alpha=0,\\
{C^\alpha}(\rn),& \text{if}\,\,\, 0<\alpha<1,\end{cases}
\label{eq1.5}\end{equation}
where ${C^\alpha}(\rn)$ are the homogeneous H\"older spaces and $BMO$ is the space of functions of bounded mean oscillation. Rigorous definitions and main properties of all these function classes can be found in Section~\ref{s2} below. 

Going further, given $x_0\in\mathbb{R}^{n}$, let $\Gamma^{\beta}(x_0)=\{(x,t)\in\mathbb{R}_{+}^{n+1}:|x_0-x|<\beta t\}$ be the cone with vertex $x_0$ and aperture $\beta>0$. Then the non-tangential maximal function for a function $u:\mathbb{R}_{+}^{n+1}\!\to\mathbb{C},$ is defined by $$N_*^{\beta} u(x_0)  \equiv\sup_{(x,t)\in\Gamma^{\beta}(x_0)}|u(x,t)|.$$
We shall also use its modified version based on the $L^2$-averages, 
$$\N^{\beta} u(x_0)  \equiv\sup_{(x,t)\in \Gamma^{\beta}(x_0)}\left(\fiint_{W(x,t)} |u(y,s)|^{2}dyds\right)^{\frac{1}{2}},$$ where, as before,  $W(x,t) \equiv \Delta(x,t) \times (t/2, 3t/2)$ and $\Delta(x,t) \equiv \left\{y \in \rn: \left| x-y \right| < t\right\}$. The superscript $\beta$ will be omitted when equal to 1 or not important in the context.

Let us return now to elliptic operators defined in \eqref{eq1.1}--\eqref{eq1.2}. 
Throughout the paper $L$ will be an elliptic  divergence form elliptic operator
with  bounded, measurable, complex-valued  coefficients. We shall assume, in addition, that the solutions to $Lu=0$ in $\reu$ are locally H\"older continuous in the following sense.  
For any  $u\in W^{1,2}_{\loc}(\reu)$ such that  $Lu=0$ in $\reu$ and any $n+1$-dimensional ball $B_{2R}(X)\subset \reu$, $X\in \reu$, $R>0$ we have 
\begin{equation}
|u(Y)-u(Z)|\leq
C \left(\frac{|Y-Z|}{R}\right)^{{\alpha_0}}\left(\,\fiint\limits_{B_{2R}(X)}|u|^{2}\right)^{\frac{1}{2}},\quad\mbox{for all}\quad Y,Z \in B_R(X),
\label{eq1.6}\end{equation}
\noindent for some constants $\alpha_0>0$ and $C>0$.
 In particular, one can show that Moser's local boundedness estimate holds (\cite{Mo}):
\begin{equation}
|u(Y)|\leq
C \left(\,\fiint\limits_{B_{2R}(X)}|u|^{p}\right)^{\frac{1}{p}},\quad\mbox{for all}\quad Y \in B_R(X), \quad 0<p<\infty.
\label{eq1.7}\end{equation}

We shall refer to property \eqref{eq1.6} by saying that the solutions (or, slightly abusing the terminology, the operator) satisfy the De Giorgi-Nash-Moser (or DG/N/M)
bounds.  Respectively, the constants $C$ and $\alpha_0$ in \eqref{eq1.6}, \eqref{eq1.7} will be called the De Giorgi-Nash-Moser constants of $L$. Finally, following \cite{AAAHK}, \cite{HMiMo}, we shall normally refer to the following collection of quantities: the dimension, the ellipticity, and the De Giorgi-Nash-Moser constants of $L$, $L^*$  as the {\it ``standard constants"}.

We note that the De Giorgi-Nash-Moser estimates are not necessarily satisfied for all elliptic PDEs with complex coefficients \cite{Fr, MNP, HMMc}. However, \eqref{eq1.6}, \eqref{eq1.7} always hold when the coefficients of the underlying equation are real \cite{DeG,N,Mo}, 
 and the constants depend quantitatively only upon ellipticity and dimension.
 We also note that \eqref{eq1.6}, \eqref{eq1.7}  hold always if $n+1=2$ and, if the coefficients of the operator are $t$-independent, 
when the ambient dimension $n+1=3$
(see \cite[Section 11]{AAAHK}). Moreover, \eqref{eq1.6} (which implies \eqref{eq1.7}) is stable under small complex perturbations
of the coefficients in the $L^\infty$ norm (see, e.g., \cite{Gi}, Chapter VI, or  \cite{A}).  Resting on this result, we demonstrate in Section~\ref{s2} that, in addition, whenever operator $L$ which satisfies \eqref{eq1.4} is such that $L_0$ with coefficients $A^0=A(x,0)$ verifies DeG/N/M and $\eps_0$ in \eqref{eq1.4} is small enough depending on the dimension, the ellipticity, and the De Giorgi-Nash-Moser constants of $L_0$, one can conclude that the De Giorgi-Nash-Moser bounds hold for $L$. Thus, in particular, \eqref{eq1.6}, \eqref{eq1.7} hold automatically for operators satisfying the Small Carleson Measure condition on coefficients in dimension $n+1=3$, and for Small Carleson Measure perturbations of real elliptic coefficients in all dimensions.

Throughout the paper we assume that the operators treated satisfy the De Giorgi-Nash estimates \eqref{eq1.6}. 

Recall now the definitions of classical layer potentials. For $L$ satisfying \eqref{eq1.1} and \eqref{eq1.2}, there exist the fundamental solutions  $\Gamma,\,\Gamma^*$  in $\mathbb{R}^{n+1}$ associated with $L$ and $L^*$ respectively, so that 
$$L_{x,t} \,\Gamma (x,t,y,s) = \delta_{(y,s)} \;\;\;\text{and}\;\;\; L^*_{y,s}\, \Gamma^*(y,s,x,t) \equiv L^*_{y,s} \,\overline{\Gamma (x,t,y,s)} = \delta_{(x,t)},$$
where $\delta_{(x,t)}$ denotes the Dirac mass at the point $(x,t)$, and $L^*$ is the hermitian adjoint of $L$. One can refer to \cite{HK} for their construction. We define the single layer potential and the double layer potential operators associated with $L$, respectively, by 
\begin{equation*}
\begin{split}{\SL}^L_t f(x) & \equiv\int\limits_{\mathbb{R}^{n}}\Gamma(x,t,y,0)\,f(y)\,dy, \,\,\, t\in \mathbb{R},\\
\mathcal{D}^L_{t}f(x) & 
\equiv\int\limits_{\mathbb{R}^{n}}\overline{\partial_{\nu^*} \Gamma^*
(y,0,x,t)}\,f(y)\,dy,\,\,\, t \neq 0,\end{split}
\end{equation*}
where $\partial_{\nu^*}$ is the adjoint exterior conormal derivative 
\begin{align*}
\partial_{\nu^*} \Gamma^* (y,0,x,t)
&=-\sum^{n+1}_{j=1}A_{n+1,j}^{\ast}(y,0)\frac{\partial \Gamma^*}{\partial
y_{j}}(y,0,x,t)\\
&=-e_{n+1}\cdot A^{\ast}(y,0)
\nabla_{y,s}\Gamma^*(y,s,x,t) \mid_{s=0},
\end{align*}
and $A^{\ast}$ is the hermitian adjoint of $A$. Furthermore, let
\begin{align}\label{eq1.10}
\left(\SL_t D_j\right)f(x)&:= \int_{\rn}\frac{\partial}{\partial y_j}\Gamma(x,t,y,0)\,f(y)\,dy\,,\quad 
1\leq j\leq n\\[4pt]\label{eq1.11}
\left(\SL_t D_{n+1}\right)f(x)&:= \int_{\rn}\frac{\partial}{\partial s}\Gamma(x,t,y,s)\big|_{s=0}\,f(y)\,dy\,,
\end{align}
and we set 
\begin{equation}\label{eq1.12}\left(\SL_t\nabla\right):=\Big(\left(\SL_t D_1Ä\right),\left(\SL_t D_2\right),...,
\left(\SL_t D_{n+1}\right)\Big)\,,\,\,\,{\rm and} \,\, \left(\SL_t\nabla\right)\cdot \vec{f}:=\sum_{j=1}^{n+1}
\left(\SL_t D_j\right)f_j\,,
\end{equation}
where $\vec{f}$ takes values in $\mathbb{C}^{n+1}$.
Note that for $t$-independent operators, we have by translation invariance in $t$ that
$\left(\SL_t D_{n+1}\right)=\,-\,\partial_t\SL_t$.

In the presence of non-tangential maximal function or Lipschitz space estimates, one can define the operators $\K^L$, $\widetilde{\mathcal{K}}^L$ and ${\SL}^L_t\!\mid_{t=0}$ on the boundary $\partial \RR^{n+1}_{\pm} =\rn$ such that 
\begin{equation}\label{eq1.12.1}\mathcal{D}^L_{t}f \xrightarrow{t \to 0} \left(\mp \frac{1}{2}I+\mathcal{K}^L\right)f, \quad \partial_\nu{\SL}^L g=\left(\pm\frac{1}{2}I+\widetilde{\mathcal{K}}^L\right)g, \quad \SL^L_t h \xrightarrow{t \to 0} \SL^L_t\!\mid_{t=0} h,\,
\end{equation}
in a suitable sense.  For operators with $t$-independent coefficients, the construction of 
$\K^L$, $\widetilde{\mathcal{K}}^L$ as above, and a description of the sense in which they exist,
may be found in \cite{HMiMo} (see also \cite[Lemma 4.28]{AAAHK} for the discussion in $L^2$). For elliptic operators treated in this paper, we shall present the pertinent results, as well as the proper definitions of convergence to the boundary and normal derivative, in Section~\ref{s7}.

The first main result of this work as follows.

\begin{theorem}\label{t1.13}
Let $A^1(x,t)=A(x,t)$ be a complex $(n+1) \times (n+1)$ coefficient matrix, satisfying the ellipticity condition \eqref{eq1.2}, and let $A^0(x,t)=A^0(x):=A(x,0)$.  Denote $L=L_1=-\nabla\cdot A(x,t) \nabla$ and $L_0=-\nabla\cdot(A(\cdot,0)\nabla)$.  Assume further that the operator $L_0$ and its adjoint $L_0^*$  
satisfy the De Giorgi-Nash property \eqref{eq1.6} for some $\alpha_0>0$.  
Then there exists $\varepsilon_0$ sufficiently small, depending on $n,\lambda,\Lambda$, and the De Giorgi/Nash bounds of $L_0$, $L_0^*$, such that whenever coefficients of $L$ satisfy \eqref{eq1.4} we have 
\begin{eqnarray}
\label{eq1.14} \|\N(\nabla \SL^L f)\|_{L^p(\rn)} &\lesssim  \|f\|_{H^p(\rn)}, \quad &p\in (p_0, 2+\eps),\\[4pt]
\label{eq1.15} \|N_*(\D^{L^*} f)\|_{L^{p'}(\rn)} &\lesssim  \|f\|_{L^{p'}(\rn)}, \quad &p\in (1, 2+\eps),
\\[4pt]
\label{eq1.15.1} \|\D^{L^*} f\|_{\Lambda^\beta(\overline{\reu})} &\lesssim  \|f\|_{\Lambda^\beta(\rn)}, \quad &\beta\in (0,\alpha_0),
\\[4pt]
\label{eq1.16} \sup\limits_{\text{cube}\;Q\subset\rn} \left(\displaystyle{\frac{1}{|Q|^{1 + 2\beta /n}} \iint_{R_Q} |\nabla \D^{L^*} f |^2 tdxdt}\right)^{1/2} &\lesssim  \|f\|_{\Lambda^\beta(\rn)}, \quad &\beta\in [0,\alpha_0),
\end{eqnarray}
\noindent the following square function bounds hold
\begin{eqnarray}
\int_{\rn}\left(\iint_{|x-y|<t} 
\left|t\,\nabla\left(\SL^{L^*}_{t}\nabla\right)f(y)\right|^{2}\,\frac{dydt}{t^{n+1}}\right)^{p'/2} dx
& \lesssim\,
 \|f\|^{p'}_{L^{p'}(\rn)}\,,\quad &1<p<2+\eps\,,\label{eq1.17}\\[4pt]
\sup\limits_{\text{cube}\;Q\subset\rn} \left(\displaystyle{\frac{1}{|Q|^{1 + 2\beta /n}}} \iint_{R_Q} \left|t\,\nabla\left(\SL^L_{t}\nabla_\|\right)f(x)\right|^{2}\,\frac{dxdt}{t}\right)^{1/2}
&\lesssim\,
 \|f\|_{\Lambda^\beta(\rn)}\,, \quad &\beta\in [0,\alpha_0),
\label{eq1.19}
\end{eqnarray}
\noindent and the following estimates on the boundary are valid
\begin{eqnarray}
\label{eq1.20} \|\nabla_\| \SL^L f\|_{H^p(\rn)} &\lesssim  \|f\|_{H^p(\rn)}, \quad &p\in (p_0, 2+\eps),\\[4pt]
\label{eq1.21} \|\widetilde \K^L f\|_{H^p(\rn)} &\lesssim  \|f\|_{H^p(\rn)}, \quad &p\in (p_0, 2+\eps),\\[4pt]
\label{eq1.22} \|\K^{L^*}f\|_{L^{p'}(\rn)} &\lesssim  \|f\|_{L^{p'}(\rn)}, \quad &p\in (1, 2+\eps),\\[4pt]
\label{eq1.22.1} \|\K^{L^*}f\|_{\la(\rn)} &\lesssim  \|f\|_{\la(\rn)}, \quad &\alpha\in [0,\alpha_0).\end{eqnarray}

\noindent The constants $p_0<1$ and $\eps>0$ above, and the implicit constants in \eqref{eq1.14}--\eqref{eq1.22} depend on $n,\lambda,\Lambda$, and the De Giorgi/Nash constants of $L_0$, $L_0^*$ only. Analogous statements hold for the adjoint operator $L^*$.
\end{theorem}

When $L$ is $t$-independent a special case of \eqref{eq1.17}, the square function estimate  
\begin{equation}\label{eq1.23}
\iint_{\ree}  
\left|t\,\nabla\partial_t\SL^L_{t}f(y)\right|^{2}\,\frac{dydt}{|t|}\lesssim
 \|f\|^2_{L^2(\rn)},
\end{equation}

\noindent was proved for complex perturbations of real, symmetric coefficient matrices in \cite{AAAHK}, and in general 
in \cite{R}  (for an alternative proof, 
see  \cite{GH}). It underpins much of the (generalized) Calder\'on-Zygmund theory leading to \eqref{eq1.14}--\eqref{eq1.16} in that context. With \eqref{eq1.23} given, and {\it still restricted to $t$-independent coefficients}, \eqref{eq1.14}--\eqref{eq1.16}, and \eqref{eq1.20}--\eqref{eq1.22} were proved in \cite{HMiMo}.  The full estimates on the square function, \eqref{eq1.17}--\eqref{eq1.19}, are unique to this paper\footnote{P. Auscher has pointed out to the
first named author that, with  \eqref{eq1.23} in place, the case $p=2$ of \eqref{eq1.17}, in the $t$-independent
setting, follows
readily from the functional calculus results of \cite{AAMc} (see also \cite{AAMc2}), which in turn are proved
via a refinement of the technology of the solution of the Kato problem;  in the present paper, we 
use the technology of the Kato problem directly, to give a self-contained proof.}, even in the $t$-independent case, and in fact, \eqref{eq1.17} was used in \cite{HMiMo} to establish, e.g., \eqref{eq1.15}. 

The major goal of this paper is to develop a perturbation approach allowing one to efficiently employ the condition \eqref{eq1.4} to establish the full range of layer potential estimates for operators whose coefficients {\it depend} on the transversal direction to the boundary $t$, as above.

Let us consider now the classical boundary value problems. We say that the Dirichlet problem for $L^*$ is solvable in $L^{p'}(\rn)$, $1<p<2+\eps$,  (and write $(D)_{p'}$) holds) if for any $f\in L^{p'}(\rn)$ there exists a solution $u$ of 
\begin{equation}
\begin{cases} 
L^*u=0,\;\text{ in }\mathbb{R}_{+}^{n+1}\\ 
\mbox{$u (\cdot,t)\to f$ as $t\to 0$},
\end{cases}
\label{eq1.24}\end{equation}
such that $$\|N_*(u)\|_{L^{p'}(\rn)}\leq C.$$ The statement $u (\cdot,t)\to f$ as $t\to 0$ is interpreted in the sense of convergence (strongly) in $L^{p'}$ on compacta of $\rn$.

Respectively, we say that the Dirichlet problem is solvable in
$\la(\rn)$, $0\leq \alpha<\alpha_0$, and we write that $(D)_{\la}$  holds if for any $f \in \la(\rn)$ there exists a solution $u$ of \eqref{eq1.24}
satisfying
\begin{equation}\sup\limits_{\text{cube}\;Q\subset\rn} \left(\displaystyle{\frac{1}{|Q|^{1 + 2\alpha /n}} \iint\limits_{R_Q} |\nabla u(x,t)|^2 tdxdt}\right)^{1/2} \leq C.\label{eq1.25} \end{equation}
The statement $u (\cdot,t)\to f$ as $t\to 0$ is interpreted in the sense of convergence in the weak* topology of $\la(\rn)$, $0\leq\alpha<\alpha_0$, and as a pointwise limit when $0<\alpha\leq \alpha_0$.

We say that the Neumann problem for $L$ is solvable in $H^p(\rn)$, $p_0<p<2+\eps$, 
 and we write that $(N)_{p}$ holds, if for any $g \in H^p(\rn)$  there exists a solution $u$ of 
 \begin{equation}
\begin{cases} 
Lu=0\;\;\text{ in }\mathbb{R}_{+}^{n+1},\\ 
\partial_{\nu_A} u = g\;\text{ on }\rn,
\end{cases}
\label{eq1.26}\end{equation}
satisfying
\begin{equation}\|\N(\nabla u)\|_{L^p(\rn)} \leq C.\label{eq1.27} \end{equation}
Here the conormal derivative on the boundary is understood in the weak sense and also is a limit of $\partial_{\nu_A} u(\cdot, t)$ as $t\to 0$ in the sense of distributions (see Section~\ref{s7} for details).

Finally, we say that the Regularity problem is solvable in 
$H^p(\rn)$ and we write that $(R)_{p}$ holds, 
if for any $f \in H^{1,p}(\rn)$
 there exists a solution $u$ of \eqref{eq1.24} satisfying
\begin{equation}\|\N(\nabla u)\|_{L^p(\rn)} \leq C.\label{eq1.28} \end{equation}
In this case, $u (\cdot,t)\to f$ as $t\to 0$ n.t. a.e., and $\fint_{t/2}^{2t}\nabla_{\|}u(\cdot,\tau) \,d\tau \to \nabla_{\|}f$ as $t \to 0$, weakly in $L^p$ when $p>1$.

One of the central methods of solution of the boundary value problems  is the method of layer potentials. To be specific, observe that the solutions to Dirichlet, Neumann, and Regularity problems in $\ree_{\pm}$ can be obtained via representations
\begin{eqnarray}
& u=\mathcal{D}^L_{t}\left(\mp\frac{1}{2}I+\mathcal{K}^L\right)^{-1}f, &\quad \mbox{for} \quad (D)_{p'}, (D)_{\Lambda_\alpha}, \label{eq1.28.1}\\[4pt]
& u=\SL^L_t\left(\pm\frac{1}{2}I+\widetilde{\mathcal{K}}^L\right)^{-1}g, &\quad \mbox{for} \quad  (N)_{p},\label{eq1.28.2} \\[4pt]
& u=\SL^L_t\left(\SL^L_t\!\mid_{t=0}\right)^{-1}f, &\quad \mbox{for} \quad  (R)_{p},\label{eq1.28.3}
\end{eqnarray}
respectively, provided that the corresponding layer potentials are bounded, that is, \eqref{eq1.14}--\eqref{eq1.16} and \eqref{eq1.20}--\eqref{eq1.22.1} are satisfied, and that the boundary layer potentials are invertible in the underlying function spaces. 

Given $1<p'<\infty$ (resp. $0<\alpha<1$), 
we say that the Dirichlet problem $(D)_{p'}$ 
(resp., $(D)_{\Lambda_\alpha}$) is solvable by the method of layer potentials if \eqref{eq1.15} (resp. \eqref{eq1.16}) is satisfied and the operator 
\begin{equation}\label{eq1.29} -\frac{1}{2}I+\mathcal{K}^L: L^{p'}\to L^{p'}, \quad \Big({\mbox{resp.,}} -\frac{1}{2}I+\mathcal{K}^L: \Lambda_\alpha\to\Lambda_\alpha\Big), \quad\mbox{ is bounded and invertible.}\end{equation}

Similarly, given $p_0<p<\infty$ 
we say that the Neumann problem $(N)_p$ 
  is solvable by the method of layer potentials if \eqref{eq1.14}  is satisfied and the operator 
\begin{equation}\label{eq1.30} \frac{1}{2}I+\widetilde{\mathcal{K}}^L: H^p\to H^p, \quad\mbox{ is bounded and invertible.}\end{equation}

Finally, given $p_0<p<\infty$  
we say that the Regularity problem $(R)_p$ 
 is solvable by the method of layer potentials if \eqref{eq1.14}  is satisfied and the operator 
\begin{equation}\label{eq1.31} \SL^L_t\!\mid_{t=0}: H^p\to H^p_1, \quad\mbox{ is bounded and invertible.}\end{equation} 

Analogous definitions apply to $\rel$.

We remark that, in principle, one could use different layer potential representations than those outlined above. For instance, the representation for the solution of the Regularity problem can serve the Dirichlet problem, and vice versa, provided that the underlying boundedness and invertibility results for layer potentials are established. This has been successfully used, e.g., in \cite{HKMP2}, and can be employed in the present setting too. To keep the statements in the introduction reasonably brief, for now we shall restrict ourselves to the representations above. 

The second main result of this work is the following.

\begin{theorem}\label{t1.32}
Let $A^1(x,t)=A(x,t)$ be a complex $(n+1) \times (n+1)$ coefficient matrix, satisfying the ellipticity condition \eqref{eq1.2}, with entries in $L^\infty(\mathbb{R}^{n+1})$, and let 
$A^0(x,t)=A^0(x):=A(x,0)$.  Set $L_0:=-\nabla\cdot(A(\cdot,0)\nabla)$,
and assume that $L_0$ and its adjoint $L_0^*$  
satisfy the De Giorgi-Nash property \eqref{eq1.16}, and that the coefficients of $L$ satisfy \eqref{eq1.4}. Let $0<\alpha<\alpha_0$ and $p_0<p<2+\eps$, $p_0=\frac{n}{n+\alpha_0}$, with $\eps>0$ and $0<\alpha_0<1$ depending on the standard constants of $L_0$ retain the same significance as in Theorem~\ref{t1.13}.  Then invertibility of boundary layer potentials in \eqref{eq1.29}--\eqref{eq1.31} for the operator $L_0$, for a given $p\in (p_0, 2+\eps)$ (or $0\leq \alpha<\alpha_0$), implies invertibility of the corresponding boundary layer potentials for the operator $L$ in the same function spaces, provided that the constant $\eps_0$ in \eqref{eq1.4} is sufficiently small depending on the standard constants of $L_0$ and implicit constants in \eqref{eq1.29}--\eqref{eq1.31} only. 

Respectively, with the same restrictions on $\alpha$ and $p$, if the boundary value problems  $(D)_{p'}$ 
(resp., $(D)_{\Lambda_\alpha}$), $(N)_p$, and $(R)_p$, are solvable for $L_0$ by the method of layer potentials, then so are the corresponding boundary problems for $L$.

Moreover the corresponding solutions are unique in the conditions of Proposition~\ref{p8.3.1}.
\end{theorem}

We remark that, to prove solvability of a given boundary value problem for the operator $L_1$,
we shall use our hypotheses concerning
De Giorgi/Nash estimates for 
both $L_0$ and $L_0^*$.

Although the theorem is stated in the upper-half space, the same arguments work identically in the lower-half space. Moreover, by a standard pull-back transformation, Theorem~\ref{t1.32} can be proved in unbounded Lipschitz graph domains. Our results may be further extended to the setting of star-like Lipschitz domains by the use of a partition of unity argument, as in \cite{MMT}, to reduce to the case of a Lipschitz graph. We omit the details.

With a slight abuse of language, if $\epsilon(x,t)$ is as in \eqref{eq1.4} with appropriately small $\eps_0$ we shall say that $A^1-A^0$ satisfies a Carleson measure condition with sufficiently small constant or that $A^1$ is a small perturbation of $A^0$ and we write $A_0$, $A_1$ satisfy (SCMC). In the case that the mode of measuring the discrepancy $A^1-A^0$ is the $L^\infty$-norm, i.e. $\|A^1-A^0\|_{L^\infty} \leq \varepsilon_0$, we will say that $A^1$ is a small $L^\infty$-perturbation of $A^0$. 
\medskip

One should notice that \eqref{eq1.4} is the Carleson measure condition for the discrepancy of the coefficient matrices introduced  by Fefferman, Kenig and Pipher \cite{FKP}\footnote{We mention that in \cite{FKP}, it is not required that $A_0$ be $t$-independent.  In the setting of complex coefficients,
it is still an open problem to treat the case that $A_0$ is $t$-dependent.}, where it is treated only in the case of 
matrices with real $L^\infty$ coefficients (since their method relies on positivity and estimates of harmonic measure), with data in $L^p$ for some $p>1$. The goal is to extend the results in \cite{FKP} and \cite{KP} as far as possible in the complex setting. A purely $L^2$ version of the results in \cite{FKP} and in \cite{KP}
has recently been obtained for complex equations, 
and more generally, for elliptic systems,
by Auscher and Axelsson in \cite{AA}, assuming 
that $A^0 = A^0(x)$ is $t$-independent, 
and that one has sufficient smallness of the Carleson norm
controlling the discrepancy.  
Our results may therefore be viewed as $L^p$ and
``endpoint" analogues of
those of \cite{AA}, and \cite{AA2}, 
using the techniques developed in \cite{AAAHK}.    
We note that such endpoint bounds cannot be obtained in general,
without the De Giorgi/Nash hypothesis, by virtue of the examples of \cite{MNP} and \cite{Fr}. For more details of the history of work in this area we refer the reader to \cite{AAAHK}.

We conclude this introduction with two further comments.  First, it should be noted that the basic strategy of our perturbation argument owes much to that of \cite{KP2}.  Indeed, our representation for the discrepancy between null solutions $u_0$ and $u_1$, of $L_0$ and $L_1$ respectively, is 
analogous to that of \cite{KP2}, and
our tent space estimates 
in Section \ref{s5}, 
which are really the key to our work, 
are a variant of those proved in  \cite{KP2}.    There is one important difference however:  in \cite{KP2}, the authors considered only the case of real coefficients, and in that case it was natural to work with
$L_D^{-1}$, the inverse of $L$ with Dirichlet boundary condition, whose kernel is the Green function.
In that setting, positivity, and properties of the Green function and of harmonic measure could be exploited.  In our setting of complex coefficients, where these tools are not available, it seems natural to replace $L_D^{-1}$ with $L^{-1}$, the global inverse of $L$, whose kernel is the global fundamental solution,
and to work with layer potential estimates in place of harmonic measure.  In this setting, in order to prove our main tent space estimates, we are led to prove certain square function bounds in Section~\ref{s3}
which generalize the Kato problem \cite{AHLMcT}.

Our second comment is as follows.  Since our solvability results are of perturbation type, in which the ``unperturbed"
operator $L_0$ is $t$-independent, it is of course of interest to know that there is a rich class of operators $L_0$ to which our Theorem \ref{t1.32}
applies.    A substantial class of such operators is provided by the results of
\cite{AAAHK} and \cite{HMiMo}. We note, furthermore, that for any $t$-independent operator $L_0$ with real, possibly non-symmetric, coefficients, there exists $1<p<2+\eps$ such that the $(D)_{p'}$ and $(R)_p$ are solvable, accompanied by the invertibility results for  the single layer potential at the boundary (see \cite{HKMP}, \cite{HKMP2}).

\vskip 0.08 in
\noindent {\it{Acknowledgements.}} S. Hofmann was supported by NSF grant DMS-1101244.
Svitlana Mayboroda was partially supported by the Alfred P. Sloan Fellowship, the NSF CAREER Award DMS 1056004, and the NSF Materials Research Science and Engineering Center Seed Grant DMR 0819885. Mihalis Mourgoglou was supported by Fondation Math\'ematiques Jacques Hadamard.

\section{Preliminaries} \label{s2}

In this section we review some definitions and theorems concerning the function spaces we deal with along with the results on layer potentials we will use in the sequel. 

\subsection{Notation}\label{s2.1}
As is common practice we use the letter $C$ to denote a constant depending on the 
``standard constants", i.e., the
ellipticity parameters, dimension, and the De Giorgi/Nash constants,
that may vary from line to line, while whenever it is deemed necessary, we indicate the dependence by adding a subscript or a parenthesis with the parameters of dependence specified. 
For our convenience we often write $A \lesssim B$ or $A \gtrsim B$ instead of saying that there exists a positive constant $C$ such that $A \leq C B$ or $A \geq C B$. When $A \lesssim B$ and $A \gtrsim B$ hold simultaneously we simply write $A \approx B$ or $A\sim B$ and we say that $A$ is comparable to $B$ or $A$ and $B$ are comparable.  In these cases, the implicit constants are 
permitted to depend upon the ``standard constants" mentioned above. 

Throughout the paper we shall denote a point in the upper (lower) half-space by either $X\in \mathbb{R}^{n+1}_\pm$ or $(x, \pm t) \in \mathbb{R}^{n+1}_\pm$, where $x\in \rn$, $\pm t=x_{n+1}$ and $t \in (0,\infty)$. By $\frac{\partial}{x_j}=\partial_{x_j}=\partial_j$ we denote the partial derivative with respect to the $j$-th variable in $\mathbb{R}^d$ when $j \in \{1,2,\dots,d\}$, and by $\nabla= (\partial_1,\partial_2,\dots,\partial_d)$ the gradient operator. In the special case of $\reu$ for the $n$-dimensional gradient or divergence we have the notational convention $\nabla_x$ and $\dv_x$ or $\nabla_\|$ and $\dv_\|$, while $\partial_t$ will be the $(n+1)$-partial derivative. Similarly, if $A$ is an $(n+1)\times(n+1)$ matrix we write $A_\|$ for the $n\times n$ sub-matrix with entries $(A_\|)_{i,j}\equiv A_{i,j}$, when $1 \leq i,j \leq n$. Finally, following \cite{AA}, we shall denote $\nabla_A:=(\nabla_\|, \partial_{\nu_A})^\perp$, where $\partial_{\nu_A}$ is, as before, the conormal derivative to be interpreted in the appropriate weak sense below.

We normally denote by $\Delta$ and $Q$ balls and cubes in $\rn$, respectively, and by $B$ balls in $\ree$. Whitney cubes in $\ree$ or $\rel$ are denoted by $W$. 

If $f \in {\bf S}'$ and $\varphi \in {\bf S}$, we let $\langle f, \varphi \rangle$ denote the Hermitian duality pairing between
${\bf S}$ and ${\bf S}'$, i.e.,  $\langle f,\varphi\rangle:= (f, \overline{\varphi})$, where $(\cdot,\cdot)$ denotes the standard
${\bf S},{\bf S}'$ duality pairing.  We shall use $\langle \cdot,\cdot\rangle$ to denote this Hermitian 
duality pairing in both $\rn$ and in $\mathbb{R}^{n+1}$, but the usage should be clear in context. 
\subsection{Function spaces}\label{s2.2}
\subsubsection{Hardy and Hardy-Sobolev Spaces}\label{s2.2.1}Let us first give the definitions of $H^p$ atoms and molecules.

\begin{definition}\label{d2.1} For $\textstyle{\frac{n}{n+1}}<p\leq 1$  and $r$ such that $p<r$ and $1 \leq r< \infty$, a complex-valued function $a(x)$ is defined to be a $(p,r)$-atom in $\rn$ if it is supported in a cube $Q \subset \rn$ and satisfies:
\begin{equation}\label{eq2.2}
\int\limits_{\rn} a(x)\,dx=0, \;\;\;\text{and}\;\;\; \| a \|_{L^r(\rn)} \leq \left|Q\right|^{\frac{1}{r}-\frac{1}{p}}.
\end{equation}
\end{definition}

\begin{definition}\label{d2.3}
We say that $m \in L^1(\rn) \cap L^2(\rn)$ is an $H^p$-molecule adapted to a cube $Q \subset \rn,$ if
\begin{itemize}
\item[\textit{(i)}] We have the following cancellation and size conditions: $$\displaystyle \int\limits_{\rn} m(x)\,dx=0,\;\;\;\;\;\quad
\displaystyle \int\limits_{16Q} \left| m(x)\right|^2\,dx \leq \ell(Q)^{n\left(1-\frac{2}{p}\right)}.$$
\item[\textit{(ii)}] There exists a constant $\varepsilon>0$ such that  for every $k \geq 4$ 
$$\displaystyle \int\limits_{2^{k+1} Q\setminus 2^k Q} \left| m(x)\right|^2\,dx \leq 2^{-\varepsilon k} \left(2^k\ell(Q)\right)^{n\left(1-\frac{2}{p}\right)}.$$
\end{itemize}
\end{definition}

\begin{lemma}[Atomic decomposition]\label{l2.4}
Let $0<p\leq 1$  and $r$ such that $p<r$ and $1 \leq r < \infty$. Then $f \in H^p(\rn)$ if and only if $f \in \mathcal{S}'(\rn)$ and there exists a sequence of $(p,r)$-atoms $\{a_j\}_{j=1}^\infty$ and a sequence of complex numbers $\{\lambda_j\}_{j=1}^\infty$ with $$\sum_j |\lambda_j|^p \leq C \|f\|^p_{H^p(\rn)},$$ for a positive constant $C$, such that, in the sense of tempered distributions,
\begin{equation}f=\sum_j \lambda_j a_j.\label{eq2.5}\end{equation}
In particular,
$$\|f\|_{H^p(\rn)}  \approx \inf \left\{\left(\sum_j |\lambda_j |^p\right)^{1/p} : \text{all decompositions of $f$ as in \eqref{eq2.5} }\right\}.$$
\end{lemma}
\begin{proof}
For a proof see \cite[p.283]{GCRDF} or \cite[p.106]{St}.
\end{proof}
The following lemma allows us to obtain $H^p \to H^p$ boundedness of a sublinear operator, which turns out to be a very useful criterion in the case of an integral operator whose kernel has pointwise or $L^2$ bounds.

\begin{lemma}\label{l2.6}
\noindent (i) Let $T$ be an $L^2(\rn) \to L^2(\rn)$ bounded sublinear operator. If $T$ maps any $(p,2)$-atom to an $H^p$-molecule adapted to the cube $Q$ which contains the support of the atom, then $T$ is $H^p(\rn) \to H^p(\rn)$ bounded.

\noindent(ii) A linear operator $T$ extends to a bounded linear operator from Hardy spaces $H^p(\rn)$ with $p\in(0,1]$ to some quasi-Banach space $\mathcal{B}$ if and only if $T$ maps all $(p,2)$-atoms into uniformly bounded elements of $\mathcal{B}$.
\end{lemma}

\begin{proof}
For a proof of \textit{(i)} see for example \cite{TW}, \cite{CW} and \cite{HMMc}, while for \textit{(ii)} see \cite{YZ1} and \cite{YZ2}.
\end{proof}

\begin{definition}\label{l2.7}
A distribution $f$ is said to be in the Hardy-Sobolev space $H^{1,p}(\rn)$, for $p\in(\frac{n}{n+1},1]$, if $\nabla f \in H^p(\rn)$.
\end{definition}

Let us state without proof a lemma which provides us with an alternative characterization of $H^{1,p}(\rn)$ which is very useful in the proof of Theorem \ref{t1.32} (see \cite{KoS}).

\begin{lemma}[\cite{KoS}]\label{l2.8}
Let $n\geq 1$ and $p\in(\frac{n}{n+1},1]$. Then a distribution $f \in \SL'(\rn)$ belongs to $H^{1,p}(\rn)$ if and only if is locally integrable and there is a function $g \in L^p$ such that 
\begin{equation}
|f(x)-f(y)| \leq |x-y| (g(x)+g(y)), \,\,\,\,\, x,y \in \rn \setminus E,
\label{eq2.9}\end{equation}
where $E$ is a set of measure zero. Moreover, one has the equivalence of (quasi)norms
\begin{equation}
\|f\|_{H^{1,p}(\rn)} \approx inf\|g\|_{L^p(\rn)},
\label{eq2.10}\end{equation}
where the infimum is taken over all admissible functions $g$ in \eqref{eq2.9}, and one identifies functions differing only by a constant.
\end{lemma}

The following lemma is the extension of the usual Hardy-Sobolev embedding theorem in the case of Hardy-Sobolev spaces (see \cite[p.136]{St}).
\begin{lemma}\label{l2.11}
Assume that $\frac{n}{n + 1} < p \leq 1$ and let $p^*:=\frac{pn}{n-p}$	be the Sobolev conjugate exponent, so that $p^* \geq 1$. Then $H^{1,p}(\rn) \hookrightarrow L^{p^*}(\rn)$.
\end{lemma}

\subsubsection{Lipschitz and $BMO$ Spaces}\label{s2.2.2}  
The quotient spaces $BMO(\rn)$ and $\dot C^\alpha(\rn)$ of functions modulo constants, equipped with the norms
$$\| f \|_{BMO(\rn)}= \sup_{\text{cube}\,Q\subseteq \rn} {\fint\limits_Q {\left| f(x)-f_Q \right|}}\,dx$$ and
$$\| f \|_{\dot C^\alpha(\rn)}:= \sup_{\substack{{(x,y) \in \rn\times \rn}\\{x \neq y}}}\frac{\lv f(x)-f(y) \rv }{{\lv x-y \rv}^\alpha},\quad 0<\alpha<1,$$
become Banach spaces. Here $f_E=\fint\limits_E f$ stands for the mean value $|E|^{-1} \int\limits_E f$, and $|E|$ here represents the Lebesgue measure of the set $E$.
By John-Nirenberg Theorem and  \cite{Me} the space $\la(\rn)$ defined by \eqref{eq1.5} is a Banach space with norm
\begin{equation}
\| f \|_{\la(\rn)}= \sup_{\text{cube}\,Q\subseteq \rn} \frac{1}{{\lv Q \rv}^{\alpha/n}} \lp \frac{1}{\lv Q \rv}\int\limits_Q {\left| f(x)-f_Q \right| ^2}\rp ^{1/2}.
\label{eq2.12}\end{equation}
Moreover, $\la(\rn)$ can be viewed as the dual space of $H^p(\rn)$ for $\alpha=n(1/p-1)$ which will be denoted by $(H^p(\rn))^*$. A $BMO(\rn)$ function $f$ is said to belong to $VMO(\rn)$ if 
$$ \fint\limits_Q {\left| f(x)-f_Q \right|} \to 0, \;\;\text{as}\;\; |Q| \to 0.$$
It can be proved that $VMO(\rn)$ is the closure in $BMO(\rn)$ of the continuous functions that vanish at infinity and $H^1(\rn)=(VMO(\rn))^*$ (for details see \cite{Sa}).

\subsubsection{Tent Spaces}\label{s2.2.3}
Given a measurable function $f:\reu\to\RR$, let us consider
\begin{eqnarray*} 
\mathcal{A}_q(f)(x) &=& \left(\iint\limits_{\Gamma(x)} \left|f(z,t)\right|^q \frac{dzdt}{t^{n+1}} \right)^{1/q}, \quad {\mathfrak C}_{q}(f)(x) = \sup_{B \ni x} \left(\frac{1}{|B|} \iint\limits_{\widehat{B}} |f(y,s)|^q \frac{dxdt}{t} \right)^{1/q}, 
\\
N_* f(x)  &=&\sup_{(z,t)\in \Gamma(x)}|f(z,t)|, \qquad \qquad
{\C}_{\alpha}(f)(x) = \sup_{B \ni x} \left(\frac{1}{|B|^{1 + 2\alpha /n}} \iint\limits_{\widehat{B}} |f(y,s)|^2 \frac{dxdt}{t} \right)^{1/2}, 
\end{eqnarray*} 
where $\widehat{B}:=\{(x,t): \text{dist}(x,B^c)\geq t\}.$  Furthermore,  denote
\begin{equation}\label{eq2.12.1} f_W(x,t):= \left(\fiint\limits_{W(x,t)}|f(y,s)|^{2}dyds\right)^{\frac{1}{2}}, \quad (x,t)\in\reu,\end{equation}
$W$ standing for Whitney cubes, and let the operators marked with a {\it tilde} stand for modifications of $\A_q$, $N_*$ and ${\mathfrak C}_{q}$ with $f_W$ in place of $f$, that is, 
\begin{eqnarray*} 
{\widetilde{\mathcal{A}}}_q(f)(x) &=& \left(\iint\limits_{\Gamma(x)} \left|f_W(z,t)\right|^q \frac{dzdt}{t^{n+1}} \right)^{1/q},\quad 
{\widetilde{\mathfrak C}}_{q}(f)(x) = \sup_{B \ni x} \left(\frac{1}{|B|} \iint\limits_{\widehat{B}} |f_W(y,s)|^q \frac{dxdt}{t} \right)^{1/q},
\\[4pt]
\N f(x)  &=&\sup_{(z,t)\in \Gamma(x)}|f_W(z,t)|.\end{eqnarray*}

For a Lebesgue measurable set $E$, we let ${\bf M}(E)$ denote the
collection of measurable functions on $E$.  For $0<p,q<\infty$ we define the following tent spaces:
\begin{eqnarray*} 
T^p_q &=& \left\{f\in {\bf M}(\mathbb{R}^{n+1}_+): {\A}_q(f) \in L^p(\rn)\right\},\quad \widetilde T^p_q = \left\{f\in {\bf M}(\mathbb{R}^{n+1}_+): {\widetilde \A}_q(f) \in L^p(\rn)\right\},\\ 
T_{\infty}^p &=& \left\{f\in {\bf M}(\mathbb{R}^{n+1}_+):N_*(f) \in  L^p(\rn)\right\},\quad 
\widetilde{T}_{\infty}^p = \left\{f\in {\bf M}(\mathbb{R}^{n+1}_+):\N(f) \in  L^p(\rn)\right\},\\ 
T_{2, {\alpha}}^{\infty} &=&\left \{f\in {\bf M}(\mathbb{R}^{n+1}_+): {\C}_{\alpha}(f) \in L^{\infty}(\rn)\right \},\quad 
{\bf T}_{2, {\alpha}}^{\infty} = \left\{f\in {\bf M}(\mathbb{R}^{n+1}_+): {\C}_{\alpha}({\bf s}(f)) 
\in L^{\infty}(\rn)\right \},
\end{eqnarray*}
where in the last definition, for $f\in{\bf M}(\mathbb{R}^{n+1}_+)$, we set
$${\bf s}(f)(x,t):= \sup_{(y,s)\in W(x,t)} \left|f(y,s)\right|$$
(the notation ``sup" is interpreted as the essential supremum). The spaces $T^p_q(\reu)$, $0<p,q\leq \infty$ were first introduced by Coifman, Meyer and Stein in \cite{CMS}. The spaces $\widetilde{T}_{q}^p(\reu)$, $\widetilde{T}_{\infty}^p(\reu)$ and ${\bf T}_{q}^{\infty}(\reu)$ started appearing in the literature more recently, naturally arising for elliptic PDEs with non-smooth coefficients.

We mention that  \cite{CMS}, Theorem~3, Section~6, {\it loc. cit.} implies that 
\begin{equation}\label{eq2.12.1}
\|\A_q(f)\|_{L^p}\approx \|{\mathfrak C}_q f\|_{L^p}, \quad q<p<\infty, 
\end{equation}
and hence, the corresponding norms in $T^p_q$ and $\widetilde T^p_q$ can be defined with ${\mathfrak C}_q$ in place of $\A_q$ and $\widetilde {\mathfrak C}_q$ in place of $\widetilde \A_q$, respectively.

\begin{remark}\label{r2.13}In the case $\alpha=0$ we will simply write $\C, T_2^{\infty}$, and ${\bf T}_2^{\infty}$ instead of ${\C}_0, T_{2, {0}}^{\infty}$ and ${\bf T}_{2, {0}}^{\infty}.$
\end{remark}

\begin{lemma}\label{l2.14}
If\,\, $0< p\leq 1$ and $\alpha=n(\frac{1}{p}-1),$ the pairing 
$$<f,g> \longrightarrow \!\iint\limits_{\reu} f(x,t)g(x,t) \frac{dxdt}{t}$$ 
realizes $T_{2, {\alpha}}^{\infty}$ as equivalent to the Banach space dual of $T^p_2$. Moreover, for $p \in (1,\infty)$ the same pairing realizes $T_2^{p'}$ as equivalent with the dual of $T_2^p$, where $1/p+1/p'=1$.
\end{lemma}

\begin{proof} 
For $p \in (0,1]$, in one direction one can adjust the arguments in \cite[p.313]{CMS}, or \cite[p.162]{St}, while for the other one can follow {\it mutatis mutandi} the argument in \cite[p.32]{HMMc}. The proof for the case $p \in (1,\infty)$ can be found in \cite[p.316]{CMS}.
\end{proof}

\begin{definition}\label{d2.15} Let $0<p\leq 1$.
A $T^p_2$-atom is defined to be a function $a \in T^p_2$ supported in $\widehat{B},$ for some ball $B,$ for which 
\begin{equation}
\left(\iint\limits_{\widehat{B}} |a (x,t)|^2\frac{dxdt}{t}\right)^{1/2}\lesssim \left|B\right|^{\frac{1}{2} - \frac{1}{p}},
\label{eq2.16}\end{equation}
and
\begin{equation}
\left\| a \right\|_{T^q_2} \lesssim \left|B\right|^{\frac{1}{q}-\frac{1}{p}},\,\,\,\, \text{for every}\,\,\, 1<q<\infty.
\label{eq2.17}\end{equation}
\end{definition}

\begin{remark}\label{r2.18} In the definition above we can replace the ball $B$ with a cube $Q$ and the tent region $\widehat{B}$ with a Carleson box $R_Q.$\end{remark}
Every function in $T^p_2$ has an atomic decomposition (see \cite{CMS},  \cite{HMMc}):

\begin{lemma}\label{l2.19} 

If $0<p\leq 1,$ for every $f \in T_2^p$ there exists a sequence of atoms $\{a_j\}_{j\geq 1} \subset T^p_2$ and a sequence $\{\lambda\}_{j\geq 1} \subset \CC$ such that 
$f=\sum_j {\lambda}_j a_j$ in $T^p_2$, and 
\begin{equation*} 
\left(\sum_j|{\lambda}_j|^p \right)^{1/p} \lesssim \|f\|_{T^p_2}:= \|{\A}_2(f)\|_{L^p}.
\end{equation*}
 
\end{lemma}

\begin{lemma}\label{l2.20}
If\,\, $0< p<\infty$ then ${\bf T}^\infty_2 \cdot \widetilde{T}^p_\infty \hookrightarrow T^p_2$ and ${\bf T}^\infty_2 \cdot {T}^p_2 \hookrightarrow \widetilde T^p_1.$
\end{lemma}

This factorization-type lemma is valid in a much bigger generality, but we are restricting here to the particular cases of use in the present manuscript. 

\begin{proof} The result is a direct consequence of the fact that  ${T}^\infty_2 \cdot {T}^p_\infty = T^p_2$ and ${ T}^\infty_2 \cdot {T}^p_2 =  T^p_1$ due to \cite{CV} and 
the fact that for every $f,g:\reu\to \rn$, $(x,t)\in\reu$, we have $(fg)_W(x,t)\leq g_W(x,t) \,\sup_{W(x,t)} f$.
\end{proof}

\subsection{Estimates for solutions and layer potentials}\label{s2.3}
Recall that we say that $L=-\nabla \cdot(A \nabla\cdot)$ satisfies the standard assumptions when $A$ is an $(n+1)\times(n+1)$ matrix of complex-valued $L^{\infty}$ coefficients, defined on $\mathbb{R}^{n+1}$ that satisfies \eqref{eq1.2} and for every solution the De Giorgi-Nash estimate \eqref{eq1.6} holds. 
Let us start with

\begin{proposition}[Caccioppoli inequality]\label{p2.23}
Let $\Omega$ be a bounded open set in $\mathbb{R}^{n+1}$, $L$ satisfy the standard assumptions, u be a solution for $L$ in $\Omega$. Let $B \subset \Omega$ a fixed ball of radius $R$ such that $2B\subset \Omega$. Then
\begin{equation}
\fiint\limits_B \left| \nabla u(Y) \right|^2 \,dY \leq \frac{C}{R^2} \fiint\limits_{2B} \left| u(Y) \right|^2 \,dY.
\label{eq2.24}\end{equation}
\end{proposition}

We recall a real variable result proved in \cite[Lemma 2.2]{HMiMo}:
\begin{lemma}\label{l2.31} If $w\in L^2_{loc}(\reu)$ is such that  $\N(w)\in L^p(\rn)$ for some $0<p\leq 2n/(n+1)$ then
$w \in L^{p(n+1)/n}(\mathbb{R}^{n+1}_+)$
and \begin{equation}\label{eq2.32}
\|w\|_{L^{p(n+1)/n}(\mathbb{R}^{n+1}_+)} \leq C(p,n)\, \|\N(w)\|_{L^p(\rn)}\,.
\end{equation}
If, in addition, $w$ is a solution of $Lw=0$ in $\reu$ for an elliptic operator $L$ whose solution satisfies the De Giorgi-Nash-Moser bounds \eqref{eq1.6}--\eqref{eq1.7}, then \eqref{eq2.32} is valid for any $0<p<\infty$. If $w=\nabla u$, where $u$ is a null-solution of $L$, then \eqref{eq2.32} holds for $p< \frac{n(2+\eps)}{n+1}$ for some $\eps>0.$
\end{lemma}

Let us demonstrate at this point that, given our hypothesis that $L_0$ and $L_0^*$ enjoy the De Giorgi/Nash
property,   so do $L_1$ and $L_1^*$, by virtue of the ``SCMC" entailed by \eqref{eq1.4}
with $\eps_0$ small.
Indeed, this observation follows immediately from the following pair of results.
The first is due to Auscher  \cite{A} and says that the De Giorgi-Nash bounds are stable under $L^\infty$-perturbations of $t$-independent elliptic operators.   The second says that SCMC implies 
small discrepancy of the coefficients also in the $L^\infty$ sense and thus, stability of
De Giorgi/Nash bounds.

\begin{lemma}[\cite{A}]\label{l2.25}
If $L^0$ is an elliptic operator with  $t$-independent complex-valued $L^{\infty}$ coefficients satisfying the De Giorgi-Nash bounds, then any elliptic operator $L^1$ with coefficients given by a sufficiently small $L^\infty$-perturbation of the coefficients of $L^0$ satisfies the De Giorgi-Nash bounds.
\end{lemma}

\begin{proposition}\label{p2.26}
Let $A^1\equiv A(x,t)$ be an $(n+1)\times (n+1)$ complex elliptic matrix with bounded measurable coefficients such that $A^0\equiv A(x,0)$ that satisfies the De Giorgi-Nash bounds. If $A^1-A^0$ satisfies the Carleson measure condition \eqref{eq1.4} with sufficiently small constant, then $A^1$ satisfies the De Giorgi-Nash bounds as well.
\end{proposition}

\begin{proof}
Fix $(x,t) \in \reu,$ and set 
$$ R(x,t) = \Delta (x,t) \times (0,3t/2), \quad W(x,t) =\Delta(x,t/8) \times (7t/8, 9t/8),\quad  W_1(x,t) = \Delta(x,t/8) \times (7t/8, 9t/8).$$
It is easy to see from the definitions that if $(y,s) \in W_1(x,t),$ then $(x,t) \in W(y,s).$ Consequently,
$$ \left| A(x,t) - A(x,0) \right| \leq \epsilon(y,s), \quad \forall (y,s) \in W_1(x,t),$$
so that 
\begin{align*}
\left| A^1(x,t) - A^0(x,t) \right| &\leq \frac{1}{|W_1(x,t)|} \iint\limits_{W_1(x,t)} \left|\epsilon(y,s)\right|^2\,dyds\\
&\lesssim t^{-n} \iint\limits_{R(x,t)} \left|\epsilon(y,s) \right|^2\frac{dyds}{s}\lesssim \varepsilon_0,
\end{align*}
where the implicit constants are purely dimensional. Thus
\begin{equation}\label{eq2.27}
\|A^1-A^0\|_{\infty} \leq C_n \varepsilon_0\,,
\end{equation}
and by Lemma \ref{l2.25}, \eqref{eq1.6} holds for $L_1$ as well.
\end{proof}

We now record some estimates which hold in the special case when the coefficients are $t$-independent.

\begin{lemma}[\cite{AAAHK}, Proposition 2.1]\label{l2.33}
Let $L$ and $L^*$ satisfy the standard assumptions and be $t$-independent. Then there is a uniform constant $\epsilon>0$ depending only on $n$ and ellipticity, and for every $p \in [2,2+\epsilon)$, a uniform constant $C_p$ such that, for every cube $Q \subset \rn$, and $t \in \Bbb{R}$, if $Lu=0$ in the box $I_Q:=4Q \times (t-\ell(Q),t+\ell(Q))$, then we have the following estimates
\begin{align}\label{eq2.34}
\left(\frac{1}{|Q|}\int_Q |\nabla u(x,t)|^p dx\right)^{1/p} &\leq C_p \left(\frac{1}{|Q^\ast|}\iint_{Q^\ast} |\nabla u(x,\tau)|^p dxd\tau\right)^{1/p}, \\\label{eq2.35}
\left(\frac{1}{|Q|}\int_Q |\nabla u(x,t)|^p dx\right)^{1/p} &\leq C_p \left(\frac{1}{\ell(Q)^2}\frac{1}{|Q^{\ast\ast}|}\iint_{Q^{\ast\ast}} | u(x,\tau)|^p dxd\tau\right)^{1/p}, 
\end{align}
where $Q^\ast:=2Q \times (t-\ell(Q)/4,t+\ell(Q)/4)$ is an $n+1$ dimensional rectangle with diameter comparable to that of $Q$, and $Q^{\ast\ast}:=3Q \times (t-\ell(Q)/2,t+\ell(Q)/2)$ is a fattened version of $Q^\ast$.
\end{lemma}

In the introduction we defined the fundamental solutions associated with an elliptic operator. In the case when the coefficients of the underlying matrix are $t$-independent, 
\begin{equation}\label{eq2.28}
\Gamma(x,t,y,s)=\Gamma(x,t-s,y,0).
\end{equation}
In light of \eqref{eq1.6} and \eqref{eq1.7}, for every integer $m\geq 0$, we have the estimates
\begin{equation}
\left| (\partial_t)^{m} \Gamma(x,t,y,0) \right| \leq \frac{C}{(|t|+|x-y|)^{n+m-1}},
\label{eq2.29}\end{equation}
and if $\alpha_0$ is the De Giorgi-Nash exponent in \eqref{eq1.6}, 
\begin{multline}
\left| (\partial_t)^{m} \Gamma(x+h,t,y,0)-(\partial_t)^{m}\Gamma(x,t,y,0)\right|\\[4pt]
+\,\,\left| (\partial_t)^{m} \Gamma(x,t,y+h,0)-(\partial_t)^{m}\Gamma(x,t,y,0)\right|\, \leq\, 
C \frac{\left| h \right|^{\alpha_0}}{(|t|+|x-y|)^{n+m-1+\alpha_0}},\label{eq2.30}\\[4pt]
 \text{whenever} \,\,\, 2|h|\leq \max(|x-y|, t).
\end{multline}
See \cite{AAAHK} for a detailed discussion of these and related results.

\begin{lemma}\label{l2.36}
 Let $L$ and $L^*$ satisfy the standard assumptions and be $t$-independent. 
 Let $Q$ be a cube in $\rn$, and let $x$ be any fixed point in $Q$. Then for $k\geq 1$ and $m\geq -1$, 
\begin{equation}
\int\limits_{2^{k+1}Q \setminus 2^kQ} \left| \left(2^k \ell(Q) \right)^m\left(\partial_t\right)^{m+1} 
\nabla_{y,s} \Gamma(x,t,y,s)\big|_{s=0} \right|^2\,dy\leq C_m \left(2^k \ell(Q) \right)^{-n-2}\,,
\quad \forall t\in\mathbb{R}\,,
\label{eq2.37}\end{equation}
and 
\begin{equation}
\int\limits_{2Q} \left|t^m\left(\partial_t\right)^{m+1} \nabla_{y,s} 
\Gamma(x,t,y,s)\big|_{s=0} \right|^2\,dy\leq C_m \left( \ell(Q) \right)^{-n-2}\,,
\quad t\approx\ell(Q)\,.
\label{eq2.38}\end{equation}
If $\alpha_0>0$ is the H\"{o}lder exponent in \eqref{eq1.6}, $k\geq 4$ and $m\geq 0$, then for any $y, y'\in Q$ we have
\begin{equation}
\int\limits_{2^{k+1}Q \setminus 2^kQ} \left|\nabla_{x,t}\partial^{m}_t \left( \Gamma(x,t,y,0)-\Gamma(x,t,y',0)\right) \right|^2\,dx \leq C_m 2^{-2\alpha_0 k}\left(2^k \ell(Q) \right)^{-n-2m}.
\label{eq2.39}\end{equation}
Furthermore, fix $(x_0,t_0) \in \mathbb{R}^{n+1}$ and suppose that $|x_0 - x| < 2\rho, |t_0-t|<2\rho$
and that $k\geq 2$.  Then, 
\begin{equation}\label{eq2.40}
\int_{2^k\rho\leq |x_0-y|<2^{k+1}\rho}\left|\nabla_{y,s}\partial^m_s\Big(\Gamma(x,t,y,s) - 
\Gamma(x_0,t_0,y,s)\Big)\big|_{s=0}\right|^2dy\leq \,C_{m} \,2^{-2k\alpha_0}(2^k\rho)^{-n-2m}.
\end{equation}
\end{lemma}

\begin{proof} See \cite[Lemma 2.8]{AAAHK}, \cite[Lemma 2.13]{AAAHK}
and \cite[(4.15)]{AAAHK}.  
\end{proof}



\begin{lemma}[\cite{AAAHK}, Lemma 2.9 and Lemma 2.10]\label{l2.43}
Suppose that $L,L^*$ satisfy the standard assumptions and 
be $t$-independent. 
Let ${\bf f}: \mathbb{R}^n \to \mathbb{C}^{n+1}.$ Then for every cube $Q$,  for all integers $k\geq
1$ and $m \geq -1$, and for all $t\in\mathbb{R}$, we have 
\smallskip
\begin{equation}\label{eq2.44}
\|\partial^{m+1}_t (S_t\nabla )\cdot({\bf f}
1_{2^{k+1}Q\setminus 2 ^kQ})\|^2_{L^2(Q)} 
\leq C_m 2^{-nk} (2 ^k\ell (Q))^{-2m-2} \|{\bf f}\|^2_{L^2(2^{k+1}Q\backslash 2^k
Q)}.\,
\end{equation}
Moreover, for each $m\geq 0$,
\begin{equation}\label{eq2.45}
\| t^{m+1}\partial^{m+1}_t \left(S_t\nabla \right)
\cdot {\bf f}\|_{L^2(\mathbb{R}^n)}\leq 
C_m\,\|
{\bf f}\|_2
\end{equation}
\end{lemma}
We note that by translation invariance in $t$, taking ${\bf f}=(0,...,0,f)$ 
to be a purely ``vertical" vector,
we obtain from \eqref{eq2.45} that
\begin{equation}\label{eq2.46}
\| t^{m+1}\partial^{m+2}_t S_t f\|_{L^2(\mathbb{R}^n)}\leq 
C_m\,\|f\|_2\,.
\end{equation}
For the sake of notational convenience, we observe that \eqref{eq2.44}
can be reformulated as
\begin{equation}\label{eq2.47}
\Vert\theta_{t}({\bf f}1_{2^{k+1}Q\backslash 2^{k}Q})\Vert_{L^{2}(Q)}^{2}\leq
C_{m}2^{-nk}\left(\frac{t}{2^{k}\ell(Q)}\right)^{2m+2}\Vert
{\bf f}\Vert^2_{L^{2}(2^{k+1}Q\backslash 2^{k}Q)}\end{equation}
 where $\theta_{t}=t^{m+1}\partial_{t}^{m+1}(S_{t}\nabla).$ We now consider generic 
 operators $\theta_{t}$ which
satisfy~\eqref{eq2.47} for some integer $m\geq 1$. 

\begin{lemma}[\cite{AAAHK}, Lemma 3.5]
\label{l2.48} (i) Suppose that $\{\theta_t\}_{t\in \mathbb{R}}$ is a 
family of operators satisfying  \eqref{eq2.47}, for some $m\geq 1$, and for 
all $t\approx\ell (Q)$.  Suppose also that $\sup_{t}
\|\theta_t\|_{2\to2}\leq C$, and that $\theta_t1=0$ for all $t\in \mathbb{R}$ (our hypotheses allow $\theta_t1$ to be defined as an
element of $L^2_{\loc }$). Then for $h\in \dot{L}^2_1 (\mathbb{R}^n)$,
\begin{equation}\label{eq2.49}\int_{\mathbb{R}^n}|\theta_th|^2\leq Ct^2\int_{\mathbb{R}^n}|\nabla_x h|^2.\end{equation}
(ii) If, in addition, $\|\theta_t \nabla_x\|_{2 \to 2} \leq C/t,$ then also
\begin{equation}\label{eq2.50} \iint_{\mathbb{R}^{n+1}_+}|\theta_t f(x)|^2 \frac{dxdt}{t} \leq C \|f\|^2_2.\end{equation}\end{lemma}

As mentioned earlier, for the $t$-independent operators the majority of the results on the boundedness of the layer potentials as postulated in Theorem~\ref{t1.13} has been established in \cite{HMiMo}. (See the introduction for a more detailed discussion). Here, we recall the boundary convergence results, which will, in particular, clarify the definition of the operators $\nabla_\| \SL_t^L\Big|_{t=0}$, $\widetilde{\K}$, and $\K$ and specify the meaning of \eqref{eq1.12.1} in the $t$-independent case. 

\begin{lemma}[\cite{HMiMo} (see also \cite{AAAHK}, Lemma 4.18 for $p=2$)]\label{l2.59} Suppose that $L,L^*$ satisfy the standard assumptions and are $t$-independent and retain the significance of constants $p_0<1$ and $\eps>0$ from Theorem~\ref{t1.13}.  

There exist  operators $\K^L, \widetilde{\K}^{L^*}, \nabla_\| \SL_t^L\Big|_{t=0}$ satisfying \eqref{eq1.20}--\eqref{eq1.22.1} with the following properties
\begin{enumerate}
\item[(i)] $\partial_{\nu_A}^\pm \nabla \SL^Lf (\cdot, t)  \to \left(\pm \frac{1}{2}I +\widetilde{\K}\right)f,$  as $t \to 0,$ weakly in $ L^p$, for all $f\in L^p$, when $p>1$ and in the sence of distributions when $p\leq 1$, $f\in H^p$.  
\item[(ii)]$\mathcal{D}_{\pm t} f \to \left(\mp\frac{1}{2}I + \K \right)f,$  as $t \to 0,$ weakly in $L^{p'}$ when $p>1$, $f\in L^{p'}$, and in the weak* topology of $\Lambda^\beta$, $0\leq \beta<\alpha_0$, for $f\in \Lambda^\beta$.
\item[(iii)]  $\nabla_\| {\SL}_{\pm t} f 
\to \nabla_\| \SL_t^L\Big|_{t=0}f,$  as $t \to 0,$ weakly in $ L^p$, for all $f\in L^p$, when $p>1$ and in the sence of distributions when $p\leq 1$, $f\in H^p$. 
\end{enumerate} 

Moreover, there exists an operator $\T_L: L^p\to L^p$, $1<p<2+\eps$, such that 
$ \partial_t \SL_t^{\pm}f\to \left(\mp \frac{1}{2A_{n+1,n+1}}+\T_L\right) f$, as $t\to 0$, weakly in $L^p.$
\end{lemma}

\section{The first main estimate:  square function bounds \`a la the Kato problem}\label{s3}

In this section and the next, we prove a pair of ``main estimates" 
that are really the deep facts underlying all of the results in this paper.
The first is an $L^2$ square function bound, 
which may be viewed as an extension
of the solution of the Kato problem \cite{AHLMcT}.  We also deduce $L^p$ and endpoint versions,
as a corollary of this $L^2$ estimate.  We recall that 
the ``standard assumptions" are listed in subsection \ref{s2.3}.

\begin{lemma}
\label{l3.1} Suppose that $L$ and 
$L^*$ satisfy the standard assumptions and are $t$-independent. 
If $\vec{f} \in L^2(\rn, \mathbb{C}^{n+1})$, 
then
\begin{equation}\label{eq3.2}
\left(\int_{-\infty}^\infty\int_{\mathbb{R}^{n}}|t\nabla (\SL_t \nabla) 
\cdot \vec{f}(x,t)|^{2}\frac{dxdt}{|t|}\right)^{\frac{1}{2}} \leq C \| \vec{f}\|_{L^2(\rn)}.
\end{equation}
\end{lemma}
As a corollary of the previous lemma, we have
\begin{corollary}\label{c3.3}  Suppose that
$L$ and its adjoint $L^*$ satisfy the standard assumptions and have $t$-independent coefficients.  Then
\begin{equation}
\sup_Q\frac1{|Q|}\int_0^{\ell(Q)}\!\!\!\int_Q \left|t\,\nabla\left(\SL_{t}\nabla\right)f(x)\right|^{2}\,\frac{dxdt}{t}
 \lesssim\,
 \|f\|^2_{L^\infty(\rn)}\,,\label{eq3.4}
\end{equation}
\begin{equation}
\int_{\rn}\left(\iint_{|x-y|<t} 
\left|t\,\nabla\left(\SL_{t}\nabla\right)f(y)\right|^{2}\,\frac{dydt}{t^{n+1}}\right)^{p/2} dx
 \lesssim\,
 \|f\|^p_{L^p(\rn)}\,,\qquad 2\leq p<\infty\,,\label{eq3.5}
\end{equation}
and
\begin{equation}
\int_{\rn}\left(\iint_{|x-y|<t} 
\left|t\,\nabla\partial_t\SL_{t}f(y)\right|^{2}\,\frac{dydt}{t^{n+1}}\right)^{p/2} dx
 \lesssim\,
 \|f\|^p_{L^p(\rn)}\,,\qquad 1<p<\infty\,,\label{eq3.6}
\end{equation}
Analogous estimates hold in the the lower half space $\ree_-$.
 \end{corollary}

We give the proof of Lemma \ref{l3.1} now, and 
the proof of Corollary \ref{c3.3} at the end of this section.

\begin{proof}[Proof of Lemma \ref{l3.1}]
We shall prove that
\begin{equation}
\iint_{\mathbb{R}_+^{n+1}}\left|t\,\nabla\left(\SL_{t}\nabla\right)f(x)\right|^{2}\,\frac{dxdt}{t}
 \lesssim\,\Vert
f\Vert_{L^2(\rn)}^{2}\,,\label{eq3.7}
\end{equation}
provided that $L$ and its adjoint $L^*$ satisfy the standard assumptions and have $t$-independent coefficients.
The analogous bound in the lower half-space holds by the same argument.

As we pointed out in the introduction, the estimate \eqref{eq3.6} for $p=2$, that is,  \eqref{eq1.23}, was proved in \cite{R}, \cite{GH}. It will be our starting point. Let us note that  \eqref{eq1.23}
implies, in particular, the Carleson measure estimate
\begin{equation}\label{eq3.8}
\sup_Q\frac1{|Q|}\int_0^{\ell(Q)}\!\!\!\int_Q 
\left|t\,(\partial_t)^2\SL_{t}f(x)\right|^{2}\,\frac{dxdt}{t} \lesssim \|f\|_{L^\infty(\rn)}\,,
\end{equation}
by a classical argument of \cite{FS}.  

We begin with some preliminary reductions.  
Following the proof of
\cite[Lemma 5.2]{AAAHK}, we may integrate by parts in $t$, and then use 
Caccioppoli's inequality in Whitney boxes to reduce matters to proving
\begin{equation*}
\iint_{\mathbb{R}_+^{n+1}}\left|t\,\partial_t\left(\SL_{t}\nabla\right)f(x)\right|^{2}\,\frac{dxdt}{t}
 \lesssim\,\Vert
f\Vert_{L^2(\rn)}^{2}\,.
\end{equation*}
We refer the reader to \cite{AAAHK} for the details.  Futhermore, by translation invariance in $t$,
and \eqref{eq1.23}, we may replace $\nabla$ by
 $\nabla_\|:=\nabla_x$, the $n$-dimensional ``horizontal" gradient.  Then,
by the adapted Hodge decomposition for the $n$-dimensional operator
$L_\|:=-\sum_{i,j=1}^n D_i A_{ij} D_j$, we may reduce matters to proving that
\begin{equation}
{\bf K}:= \iint_{\mathbb{R}_+^{n+1}}
\left|t\,\partial_t\left(\SL_{t}\nabla_\|\cdot A_\|\nabla_\| F\right)(x)\right|^{2}\,\frac{dxdt}{t}
 \lesssim\,\Vert
\nabla_\| F\Vert_{L^2(\rn)}^{2}\,,\label{eq3.9}
\end{equation}
where
$A_\|:=(A_{ij})_{1\leq i,j\leq n}$ is the $n\times n$ ``upper left block" of $A$;  thus
$L_\|=-\nabla_\|\cdot A_\|\nabla_\|$.  
Observe that 
\begin{multline*}{\bf K}=\iint_{\mathbb{R}_+^{n+1}}
\left|\partial_t\left(\SL_{t}\nabla_\|\cdot A_\|\nabla_\| F\right)(x)\right|^{2}\,t\,dxdt\\[4pt]
=-\frac12\iint_{\mathbb{R}_+^{n+1}}\partial_t
\left(\partial_t\left(\SL_{t}\nabla_\|\cdot A_\|\nabla_\| F\right)(x)\,
\overline{\partial_t\left(\SL_{t}\nabla_\|\cdot A_\|\nabla_\| F\right)(x)}\right)\,t^2\,dxdt\\[4pt]
\leq C \iint_{\mathbb{R}_+^{n+1}}
\left|(\partial_t)^2\left(\SL_{t}\nabla_\|\cdot A_\|\nabla_\| F\right)(x)\right|^{2}\,t^3\,dxdt\,+\,\frac12{\bf K}\,.
\end{multline*}
Hiding the small term on the left hand side of the inequality\footnote{To do this rigorously, i.e., to ensure that ${\bf K}$ is finite, we would truncate the $t$-integral, resulting in controllable errors when we integrate by parts;  we omit the routine details.}, we see that it is now enough to establish the bound \eqref{eq3.9}, but with ${\bf K}$ replaced by
$$\widetilde{\bf K}:= \iint_{\mathbb{R}_+^{n+1}}
\left|t^2\,(\partial_t)^2\left(\SL_{t}\nabla_\|\cdot A_\|\nabla_\| F\right)(x)\right|^{2}\,\frac{dxdt}{t}\,.$$
We now write
\begin{multline}\label{eq3.10}
t^2\,(\partial_t)^2\SL_{t}\nabla_\|\cdot A_\|\nabla_\| F=\\[4pt]
\left\{t^2\,(\partial_t)^2\SL_{t}\nabla_\|\cdot A_\|
\,-\,t^2\,\left((\partial_t)^2\SL_{t}\nabla_\|\cdot A_\|\right)P_t\right\}\nabla_\| F\,+\,
t^2\,(\partial_t)^2\left(\SL_{t}\nabla_\|\cdot A_\|\right)P_t\nabla_\| F\\[4pt]=:R_t(\nabla_\|F)
\,+\,t^2\,\left((\partial_t)^2\SL_{t}\nabla_\|\cdot A_\|\right)P_t\nabla_\| F\,,
\end{multline}
where $P_t$ is a
``nice" approximate identity given by convolution with a smooth, compactly supported,
non-negative kernel with integral 1.  In order to avoid possible confusion, we note that
the $t$-derivatives are applied only to $\SL_t$, but not to $P_t$.
The last term is the main term;  we note that by Carleson's lemma, its contribution 
to $\widetilde{\bf K}$ will be bounded, once we establish the Carleson measure estimate
\begin{equation}\label{eq3.11}
\sup_Q\frac1{|Q|}\int_0^{\ell(Q)}\!\!\!\int_Q
|t^2\,(\partial_t)^2\Big(\SL_t\nabla_\|\cdot A_\|\Big)(x)|^2 \frac{dxdt}{t}\,\leq\,C\,.
\end{equation}
We defer momentarily the proof of this bound, and consider first the ``error" term 
$R_t$, which we rewrite as
$$R_t = \left\{t^2\,(\partial_t)^2\SL_{t}\nabla_\|\cdot A_\| P_t
\,-\,t^2\,\left((\partial_t)^2\SL_{t}\nabla_\|\cdot A_\|\right)P_t\right\}
\,+\,t^2\,(\partial_t)^2\SL_{t}\nabla_\|\cdot A_\| (I-P_t)=: R_t' +R_t''\,.$$
Note that $R_t'1=0$.  Thus,
by the case $m=1$ of Lemma \ref{l2.43}, the operator $R_t'$ satisfies (all of) the hypotheses of 
Lemma \ref{l2.48}, whence it contributes a bounded square function to $\widetilde{\bf K}$.

\smallskip

Next, we consider $R_t''$.
Since the kernel of $\SL_t$ is $\Gamma(x,t,y,0)$, where 
$\overline{\Gamma(x,t,y,s)}=\Gamma^*(y,s,x,t)$ 
is an adjoint solution 
in $(y,s)$ away from the pole at 
$(x,t)$, and is jointly translation invariant  in the arguments $(t,s)$, we have that 
\begin{equation}\label{eq3.12}
R_t''(\nabla_\| F) = t^2\,\Big(\left((\partial_t)^3\SL_t\nabla\right)\cdot\left(\vec{a}\,(I-P_t)F\right)\Big)
\,-\,t^2\,\partial_t^3\SL_t\left(\vec{b}\cdot \nabla_\|(I-P_t)F\right)=:I_1+I_2\,,
\end{equation}
where $\vec{a}:= (A_{1,n+1},...,A_{n+1,n+1})$, and $\vec{b}:= (A_{n+1,1},...,A_{n+1,n})$.
By \eqref{eq2.45}, $t^3(\partial_t)^3 (\SL_t\nabla)$ is bounded on $L^2(\rn)$, uniformly in $t$.  
Moreover, $$\iint_{\mathbb{R}^{n+1}_+} |t^{-1}(I-P_t)F(x)|^2 \frac{dxdt}{t}\lesssim 
\|\nabla_\| F\|_2^2\,,$$
by a standard argument using Plancherel's theorem.  Thus, the contribution of $I_1$ is bounded.
To handle $I_2$, we further decompose it as
\begin{equation*}I_2 = -\,t^2\,\partial_t^3\SL_t\left(\vec{b}\cdot \nabla_\|F\right)
\,+\,\,t^2\,\left\{\partial_t^3\SL_t\vec{b} P_t -
\left(\partial_t^3\SL_t\vec{b}\right) P_t\right\}\cdot\nabla_\|F\,+\,t^2\,
\left(\partial_t^3\SL_t\vec{b}\right)  P_t(\nabla_\|F)\,.
\end{equation*}
The first of these terms may be handled by \eqref{eq1.23}, after using Caccioppoli's inequality on Whitney boxes to reduce the order of differentiation by one;  moreover, by Carleson's lemma
and \eqref{eq3.8}, the same strategy may be applied to the last term.  The middle term 
is of the form $\theta_t (\nabla_\|F)$, where $\theta_t$
satisfies the hypotheses of Lemma \ref{l2.48}, and therefore its contribution is also bounded.

It remains to establish the Carleson measure estimate \eqref{eq3.11}.
To this end, we invoke a key fact in the proof of the Kato conjecture. By \cite{AHLMcT}, 
there exists, for each $Q$, a mapping
$F_{Q}=\mathbb{R}^{n}\rightarrow \mathbb{C}^{n}$ such that \begin{equation}
\begin{split}\label{eq3.13}\text{(i)} & \quad\int_{\mathbb{R}^{n}}|\nabla_{\|}F_{Q}|^{2}\leq C|Q|\\
\text{(ii)} & \quad\int_{\mathbb{R}^{n}}|L_{\|}F_{Q}|^{2}\leq C\frac{|Q|}{\ell(Q)^{2}}\\
\text{(iii)} & \quad\sup_{Q}\int_{0}^{\ell(Q)}\fint_{Q}|\vec{\zeta}(x,t)|^{2}\frac{dxdt}{t}\\
 & \quad\quad\leq
C\sup_{Q}\int_{0}^{\ell(Q)}\fint_{Q}|\vec{\zeta}(x,t)E_{t}\nabla_{\|}F_{Q}(x)|^{2}\frac{dxdt}{t},\end{split}
\end{equation}
 for every function $\vec{\zeta}:\mathbb{R}_{+}^{n+1}\rightarrow\mathbb{C}^{n}$, where $E_{t}$ denotes the dyadic
averaging operator, i.e. if $Q(x,t)$ is the minimal dyadic cube (with respect to the grid induced by $Q$) containing $x$,
with side length at least $t$, then \begin{equation*} E_{t}g(x):=\fint_{Q(x,t)}g.\end{equation*}
 Here $\nabla_{\|}F_{Q}$ is the Jacobian matrix $(D_{i}(F_{Q})_{j})_{1\leq i,j\leq n}$, and the product
$$\vec{\zeta}E_{t}\nabla_{\|}F_{Q}=\sum_{i=1}^{n}\zeta_{i}E_{t}D_{i}F_{Q}$$ is a vector. Given the existence of a family of
mappings $F_{Q}$ with these properties, we see,
as in \cite[Chapter 3]{AT}, that by (iii), applied with
$\vec{\zeta}(x,t)=T_{t}A_{\|}$, where $T_t:= t^2(\partial_t)^2(\SL_t\nabla_\|)$, 
it is enough to show that \begin{equation*}
\int_{0}^{\ell(Q)}\!\!\!\int_{Q}|\left(T_{t}A_{\|}\right)(x)\,\left(E_{t}\nabla_{\|}F_{Q}\right)(x)|^{2}\frac{dxdt}{t}
\leq C|Q|.
\end{equation*}
 But as in \cite{AT}, we may exploit the idea of \cite{CM} to write
\begin{equation*}
\begin{split}(T_{t}A_{\|})E_{t}\nabla_{\|}F_{Q} & =\left\{(T_{t}A_{\|})E_{t}-T_{t}A_{\|}\right\}\nabla_{\|}F_{Q}+T_{t}A_{\|}\nabla_{\|}F_{Q}\\
 & =(T_{t}A_{\|})(E_{t}-P_{t})\nabla_{\|}F_{Q}+\left\{(T_{t}A_{\|})P_{t}-T_{t}A_{\|}\right\}\nabla_{\|}F_{Q}+T_{t}A_{\|}\nabla_{\|}F_{Q}\\
 & =:R_{t}^{(1)}\nabla_{\|}F_{Q}+R_{t}^{(2)}\nabla_{\|}F_{Q}+T_{t}A_{\|}\nabla_{\|}F_{Q},\end{split}
\end{equation*}
 where as above, $P_{t}$ is a nice approximate identify.  By the case $m=1$ of Lemma
\ref{l2.36}, and Cauchy-Schwarz, we have
that $T_tA_\|\in L^\infty$, so the contribution of
the term $R_t^{(1)}$ may be handled by a standard orthogonality argument, given
\eqref{eq3.13}(i).   The operator $R_t^{(2)}$ is exactly the same as the operator
$R_t$ in \eqref{eq3.10} (up to a minus sign), 
and we have already shown that the latter obeys a square function bound.
Finally, by definition of $T_t$, we have that
$$T_{t}A_{\|}\nabla_{\|}F_{Q}= t^2(\partial_t)^2\SL_t L_\|F_Q\,.$$
By the case $m=0$ of \eqref{eq2.46}, and \eqref{eq3.13}(ii), we then obtain
$$\int_{0}^{\ell(Q)}\!\!\!\int_{Q}|\left(T_{t}A_{\|}\nabla_{\|}F_{Q}\right)(x)|^{2}\frac{dxdt}{t}
\leq\int_{0}^{\ell(Q)}\!\!\!\int_{\rn}|t\left((\partial_t)^2\SL_{t}L_\|F_{Q}\right)(x)|^{2}\,t\,dxdt
\lesssim \frac{|Q|}{\ell(Q)^2}\int_0^{\ell(Q)}tdt \approx |Q|\,,$$
as desired. 
\end{proof}

We conclude this section with the proof of corollary of \ref{c3.3}.

 \begin{proof}[Proof of Corollary \ref{c3.3}] Estimate \eqref{eq3.5} follows immediately from
Lemma \ref{l3.1} and the Carleson measure estimate \eqref{eq3.4}.
Indeed, the case $p=2$ of  \eqref{eq3.5} is equivalent to the upper half-space version
of \eqref{eq3.2}, by the elementary fact that vertical and conical square 
functions are equivalent in $L^2$.  The case $2<p<\infty$ then follows by tent space interpolation
\cite{CMS} between the case $p=2$, and the case $p=\infty$, which is \eqref{eq3.4}.

Thus, we have reduced matters to proving \eqref{eq3.4} and \eqref{eq3.6}.  
We treat \eqref{eq3.4} first.
 The proof follows a classical argument of \cite{FS}.
 We split $f=\sum_{k=0}^\infty f_k$, where $f_0:= f1_{4Q}$, and 
 $f_k:= f1_{2^{k+2}Q\setminus 2^{k+1} Q}$,
$k\geq 1$.  By \eqref{eq3.7}, we have
\begin{equation*}
\int_0^{\ell(Q)}\!\!\!\int_Q \left|t\,\nabla\left(\SL_{t}\nabla\right)f_0(x)\right|^{2}\,\frac{dxdt}{t}
 \lesssim\, \int_{4Q} |f|^2\,\lesssim \,
|Q|\, \|f\|_{L^\infty(\rn)}\,.
\end{equation*}
Now suppose that $k\geq 1$.  Since $u_k:=(\SL_t\nabla) f_k$ solves $Lu_k=0$ in
$\ree\setminus(4Q\times(-4\ell(Q),4\ell(Q))$, by Caccioppoli's inequality, we have that
$$\int_0^{\ell(Q)}\!\!\!\int_Q \left|t\,\nabla\left(\SL_{t}\nabla\right)f_k(x)\right|^{2}\,\frac{dxdt}{t}
 \lesssim\,\fint_{-\ell(Q)}^{2\ell(Q)}\!\!\int_{2Q} \left|\left(\SL_{t}\nabla\right)f_k(x)-c_Q\right|^{2}\,dxdt\,,$$
 where the constant $c_Q$ is at our disposal.  We now choose $c_Q:= 
 \left(\SL_{t}\nabla\right)f_k\big|_{t=0}(x_Q)$, where $x_Q$ denotes the center of $Q$.
We observe that, by Cauchy-Schwarz,
\begin{multline*}
\left|\left(\SL_{t}\nabla\right)f_k(x)-c_Q\right|^{2}\\[4pt]
\leq\, \int_{2^{k+2}Q\setminus2^{k+1}Q}\left|\nabla_{y,s}\big(\Gamma(x,t,y,s) - 
\Gamma(x_Q,0,y,s)\big)|_{s=0}\right|^2dy\,\int_{2^{k+2}Q\setminus2^{k+1}Q}|f|^2\\[4pt]
\lesssim \, 2^{-2k\alpha_0}\fint_{2^{k+2}Q\setminus2^{k+1}Q}|f|^2\,\lesssim\, 
2^{-2k\alpha_0} \|f\|^2_\infty\,,
\end{multline*}
where we have used \eqref{eq2.40} in the next to last inequality.  Summing in $k$, we obtain \eqref{eq3.4}.

Next, we prove \eqref{eq3.6}.  By translation invariance in $t$, the case $2\leq p<\infty$
is already included in the conclusion of Lemma \ref{l3.1}.  
The case $1<p<2$ then follows by interpolation with the Hardy space bound
\begin{equation*}
\|t\nabla\partial_t \SL_t f\|_{T^1_2} \lesssim \|f\|_{H^1(\rn)}\,.
\end{equation*}
To prove the latter, by a standard argument,
it suffices to assume that $f= a$, an $H^1$ atom adapted to a cube $Q$,
and to establish the ``molecular bounds":
\begin{equation}\label{eq3.14}
\int_{V_k}\iint_{|x-y|<t}|t\nabla\partial_t\SL_t a(y)|^2 \frac{dy dt}{t^{n+1}}dx \lesssim 2^{-\eps k}
\left(2^{-k}\ell(Q)\right)^{-n}\,,\qquad k=0,1,2,...,
\end{equation}
where,  given a cube $Q\subset \rn$, we set $V_0=V_0(Q):= 2Q$, and
$V_k=V_k(Q):= 2^{k+1}Q\setminus 2^kQ\,, k\geq 1$.
To prove \eqref{eq3.14}, we may suppose that $k\geq 4$, since the desired bound for small $k$ 
follows trivially from the global $L^2$ bound.  To this end, we neeed only observe that
\begin{multline*}
\int_{V_k}\iint_{|x-y|<t}|\nabla\partial_t\SL_t a(y)|^2 \frac{dy dt}{t^{n-1}}dx\\[4pt]
\leq\, \int_{V^*_k}\int_0^{2^{k-1}\ell(Q)}|\nabla\partial_t\SL_t a(y)|^2 t\, dtdy
\,+\,\sum_{j\geq k}\int_{V_k}\int_{2^{j-1}\ell(Q)}^{2^j\ell(Q)}
\int_{|x-y|<t}|\nabla\partial_t\SL_t a(y)|^2 \frac{dy dt}{t^{n-1}}dx\\[4pt]
\lesssim \,
\int_{V^{**}_k}\fint_{2^{-k}\ell(Q)}^{2^{k}\ell(Q)}|\partial_t\SL_t a(y)|^2 \, dtdy
\,+\,\sum_{j\geq k}\int_{V_k}\fint_{2^{j-2}\ell(Q)}^{2^{j+1}\ell(Q)}
\int_{|x-y|<8t}|\partial_t\SL_t a(y)|^2 \frac{dy dt}{t^{n}}dx \\[4pt]\lesssim \,
2^{-2\alpha_0 k}
\left(2^{-k}\ell(Q)\right)^{-n}\,,
\end{multline*}
where in the preceding estimates we have used the following ingredients:
Fubini's Theorem in the first inequality, to obtain the first integral, where $V_k^*$
is a ``fattened" version of $V_k$;   Caccioppoli in the second inequality, 
where $V_k^{**}$ is a further fattened version of $V_k$;  and, in the third inequality,
that $|\partial_t \SL_t a(y)| \lesssim \ell(Q)^{\alpha_0} (t+|y-x_Q|)^{-n-\alpha_0}$,
by  the standard properties of $H^1$ atoms, and the ``Calder\'on-Zygmund estimate"
\eqref{eq2.30} with $m=1$, where $x_Q$ denotes the center of $Q$.
 \end{proof}

\section{The second main estimate:  tent space bounds for $\nabla L^{-1}\nabla $}
\label{s4}

In this section we will prove our second set of
``main estimates", which are tent space bounds for the operator 
$\nabla L^{-1}\dv$, for $t$-independent $L$.  The first result treats the case $p=2$.
\begin{lemma}
\label{l4.1} Suppose that $L$ and 
$L^*$ satisfy the standard assumptions and  are $t$-independent.  Then $\nabla L^{-1}\dv: T^2_2\to \widetilde{T}^2_\infty$,
and \begin{equation}\label{eq4.2}
\|\nabla L^{-1}\dv\Phi\|_{\widetilde{T}_\infty^2}\,\leq \,C\, \|\Phi\|_{T_2^2}\,,
\end{equation}
where $C$ depends only upon dimension, ellipticity, the De Giorgi/Nash constants,
and the constant in \eqref{eq1.23}.
\end{lemma}
We defer momentarily the proof of the lemma.

Next, we state appropriate versions in the case $p\leq 1$.
As above, and in the sequel, given a cube $Q\subset \rn$, we set 
$$V_k=V_k(Q):= 2^{k+1}Q\setminus 2^kQ\,,\qquad k\geq 1.$$

\begin{proposition} \label{p4.3}
Suppose that $L$ and 
$L^*$ satisfy the standard assumptions and are $t$-independent.  Suppose also that $n/(n+\alpha_0)<p\leq 1$,
and let $a$ be a (vector-valued)
$T^p_2$ atom, taking values in $\CC^{n+1}$, adapted to a cube $Q$ (i.e., to the
Carleson box $R_Q:= Q\times (0,\ell(Q))$).  Then $\nabla L^{-1}\dv a$ satisfies
the ``molecular size" estimates
\begin{align}\label{eq4.4}
\| \N( \nabla L^{-1} \nabla \cdot a)\|_{L^p(64Q)}& \leq\, C\,\ell(Q)^{n\big(\frac12-\frac1p\big)}\\[4pt]
\| \N( \nabla L^{-1} \nabla \cdot a)\|_{L^p(V_k(Q))}& \leq\, C\,2^{-\epsilon k}
\left(2^k\ell(Q)\right)^{n\big(\frac12-\frac1p\big)}\,,\quad k=6,7,...,\label{eq4.5}
\end{align}
where $\epsilon>0$ and $C$ depend only upon dimension, ellipticity, $p$, 
the De Giorgi/Nash constants, 
and the constant in \eqref{eq1.23}.
\end{proposition}

As a corollary of Lemma \ref{l4.1} and Proposition \ref{p4.3}, we have the following.
\begin{proposition} \label{p4.6}
Suppose that $L$ and 
$L^*$ satisfy the standard assumptions and are $t$-independent.    Then there is an $\eps>0$ 
such that for $n/(n+\alpha_0)<p< 2+\eps$,
we have $\nabla L^{-1}\dv: T^p_2 \to \widetilde{T}^p_\infty$,
i.e., for every $\Phi\in T^p_2(\reu,\CC^{n+1})$, 
\begin{equation}\label{eq4.7}
\|\nabla L^{-1}\dv\Phi\|_{\widetilde{T}_\infty^p}\,\leq \,C\, \|\Phi\|_{T_2^p}\,,
\end{equation}
where $\eps$ depends only upon dimension, ellipticity, the De Giorgi/Nash constants, 
and the constant in \eqref{eq1.23}, and $C$ depends upon these parameters and $p$. 
\end{proposition}

\begin{proof}[Proof of Proposition \ref{p4.6} 
(assuming Proposition \ref{p4.3}
and Lemma \ref{l4.1})]
The proof is completely standard except for the range $2<p<2+\eps$.  
Let us suppose first that $n/(n+\alpha_0)<p\leq2$.  The case $p=2$ is Lemma \ref{l4.1}.  
Given the lemma, it therefore suffices, by tent space interpolation, to prove the case
$n/(n+\alpha_0)<p\leq 1$.  To this end,
we claim that it suffices to replace $\Phi$ (in \eqref{eq4.7})
by a $T^p_2$ atom $a(x,t)$, i.e. to prove that there
exists a uniform constant $C$ such that for every $T^p_2$ atom,
\begin{equation} \| \N( \nabla L^{-1} \nabla \cdot a)\|_{L^p(\rn)} \leq C.
\label{eq4.8}
\end{equation}
Indeed, $T^p_2\cap T^2_2$ is dense in $T^p_2$; moreover, for every $\Phi \in T^p_2\cap T^2_2$,
there is an
atomic decomposition $\Phi = \sum \lambda_k a_k$, with
$\sum|\lambda_k|^p \approx \|\Phi\|_{T^p_2}^p$, which converges in both
$T^p_2$ and in $T^2_2$ (see \cite[Theorem 1 and Proposition 5]{CMS}, and also 
\cite[Proposition 3.25]{HMMc}.)  Since 
$\nabla L^{-1}\dv: T^2_2 \to \widetilde{T}^2_\infty$ (by Lemma \ref{l4.1}), 
we therefore obtain that, for $\Phi =\sum\lambda_ka_k \in T^p_2\cap T^2_2$,
\begin{equation} \N\left( \nabla L^{-1} \nabla \dv \Phi \right)\leq 
\sum|\lambda_k|\left( \N( \nabla L^{-1} \nabla \dv  a_k)\right)\,,
\label{eq4.9}
\end{equation}
whence the claim follows.

Hence, we fix a cube $Q\subset \rn$, and an atom $a$, supported
in the Carleson box $R_Q= Q\times (0, \ell(Q))$.
We decompose the $p$-th power of the norm 
in \eqref{eq4.8} as follows:
\begin{equation*}
\int\limits_{\rn}\, =\, \int\limits_{64Q} \,+\, \sum\limits_{k\geq 6} \int\limits_{\V_k},
\end{equation*}
where $\V_k= 2^{k+1} Q \setminus 2^k Q.$
The desired bound now follows immediately from H\"older's inequality 
and  \eqref{eq4.4}-\eqref{eq4.5}. 
This concludes the proof of Proposition \ref{p4.6}
in the range in the range $n/(n+\alpha_0)<p\leq2$.  The proof in the range
$2<p<2+\eps$ is more delicate, and we defer it until the end of this section. \end{proof}

Next, we give the 
\begin{proof}[Proof of Lemma \ref{l4.1}]

Let $\Phi\in T^2_2(\reu)$, taking values in $\CC^{n+1}$.  We may assume that 
$\Phi\in C_0^\infty(\reu)$, since this class of functions is dense in $T_2^2$.
Our goal is to show that $\N(\nabla L^{-1}\nabla\Phi)$ belongs to $L^2$.  For $(x,t)\in \reu$, set 
$B_{x,t}:= B((x,t),t/2) \subset \mathbb{R}^{n+1}$, and let
$\Delta_{x,t}:= \Delta(x,t/2)=\{y\in\rn:|x-y|<t/2\}$.  We set $w:= L^{-1}\dv\Phi$,
and define
$$\left(\widetilde{N}(\nabla w)(x)\right)^2 :=
\sup_{t>0} \fiint_{B_{x,t}}|\nabla w|^2 dy ds\,. $$
We remark that $\tN$ differs slightly from $\N$, but they are equivalent in $L^p$ norm,
by the well known fact that one may vary the aperture of the cones defining $\N$ (see \cite{FS}).
Therefore we work with $\tN$ throughout this section.  We claim that it suffices
to show that there is an exponent
$q<2$, depending only upon dimension and ellipticity, for which we have the pointwise bound
\begin{equation}\label{eq4.10}
\widetilde{N}(\nabla w)(x)\lesssim \left(M\big(\A_2(\Phi)^q\big)(x)\right)^{1/q} + 
M\Big((\nabla w)( \cdot,0)\Big)(x)\,,
\end{equation}
where $M$ denotes the Hardy-Littlewood maximal operator.
The desired $L^2$ bound for the first of these terms follows immediately. To handle the 
$L^2$ norm of the second, let $h\in L^2(\rn, \CC^{n+1})$, so that
\begin{multline}\label{eq4.11}
\left|\langle\nabla w(\cdot,0), h\rangle\right|=\left|\iint_{\reu}\Phi(x,t)\cdot 
\overline{t\nabla\left(\SL_t^{L^*}\nabla \right)h(x)}
\frac{dxdt}{t}\right|\\[4pt]\lesssim \|\Phi\|_{T^2_2} \,
\left(\iint_{\reu}|t\,\nabla\left(\SL_t^{L^*}\nabla \right)h|^2\frac{dxdt}{t}\right)^{1/2}
\lesssim \|\Phi\|_{T^2_2} \,\|h\|_2\,,
\end{multline}
where in the last step we have used \eqref{eq3.7} for $L^*$.  Taking a supremum over all such $h$ with norm 1, we obtain that $\|M\big(\nabla w( \cdot,0)\big)\|_2\lesssim \|\Phi\|_{T^2_2}$, thus establishing the claim.    Let us digress momentarily and note, for the sake of future reference, that the same duality argument, but now with $h\in L^{p'}(\rn)$, shows that, by \eqref{eq3.6} and translation invariance in $t$,
we have
\begin{equation}\label{eq4.12}
\|(\partial_t w)(\cdot,0)\|_{L^p(\rn)}\, \leq \,C_p\, \|\Phi\|_{T^p_2}\,,\qquad 1<p<\infty\,.
\end{equation}

We proceed now to the proof of \eqref{eq4.10}.
Fix $(x_0,t_0) \in \reu$, set $B_k:=B_k(x_0,t_0):=B((x_0,0),2^{k+2}t_0)\subset\ree$, $k=0,1,2...$, and let
$\Delta_k:=\Delta_k(x_0,t_0):= B_k(x_0,t_0)\cap (\rn\times\{0\})$.  We split
$$w=\sum_{k=0}^\infty w_k:= \sum_{k=0}^\infty L^{-1}\dv\Phi_k\,,$$
where
$$\Phi_0:= \Phi 1_{B_0}\,,\qquad \Phi_k:=\Phi1_{B_k\setminus B_{k-1}}\,,\,\, k\geq 1\,.$$
We further subdivide $\Phi_0=\Phi_0'+\Phi_0''$, where $\Phi_0':=\Phi_0 1_{\{t\geq t_0/4\}}$,
which induces a corresponding splitting $w_0=w_0'+w_0''$.
Since $\nabla L^{-1}\dv: L^2(\ree)\to L^2(\ree)$, we have that
$$\fiint_{B_{x_0,t_0}}|\nabla w'_0|^2 \lesssim \frac1{t_0^{n+1}}\iint_{B_0\cap\{t\geq t_0/4\}}
|\Phi|^2\lesssim
\iint_{B_0\cap\{t\geq t_0/4\}}|\Phi(x,t)|^2 \frac{dxdt}{t^{n+1}}\lesssim \A_2(\Phi)^2(x_0)\,,$$
uniformly in $t_0$, provided that we have defined $\A_2$ with respect to cones of sufficiently large aperture (as we may do:  see \cite[Prop. 4, p. 309]{CMS}).  Thus, the contribution of $w'_0$ gives the desired estimate.  To handle $w_0''$, we first note that $Lw_0''=0$ in 
$B^*_{x_0,t_0}:= B((x_0,t_0),3t_0/4)$.  Consequently, $\nabla w_0''$ satisfies the Reverse H\"older estimate
$$\fiint_{B_{x_0,t_0}}|\nabla w''_0|^2 \leq C_{q,n,\lambda,\Lambda} 
\left(\fiint_{\widetilde{B}_{x_0,t_0}}|\nabla w''_0|^q\right)^{2/q}\,,\qquad q<2\,,$$
where $\tB_{x_0,t_0}:=B((x_0,t_0),5t_0/8)$.   We recall that for some $\eps>0$, depending
only upon dimension and ellipticity, 
$\nabla L^{-1}\dv: L^q(\ree)\to L^q(\ree)$, for all $q\in (2-\eps,2)$.  We now fix such a $q$,
so that, by the Reverse H\"older estimate, we have
\begin{multline}\label{eq4.13}
\fiint_{B_{x_0,t_0}}|\nabla w''_0|^2 \,\lesssim\,
\left(\frac1{t_0^{n+1}}\iint|\Phi''_0|^q\right)^{2/q}\,\lesssim\,
\left(\fint_{\Delta_0}\!\fint_0^{4t_0}|\Phi(y,t)|^q\left(\fint_{|x-y|<t}dx\right)\,dydt\right)^{2/q}\\[4pt]
\,\leq\, \left(\fint_{|x_0-x|<8t_0}\left(\fint_0^{4t_0}\fint_{|x-y|<t}|\Phi(y,t)|^2\,dydt \right)^{q/2}dx\right)^{2/q}\,
\lesssim \left(M\left(\A_2(\Phi)^q\right)(x_0)\right)^{2/q}.
\end{multline} 
Since $q<2$, the contribution of $w_0''$ also satisfies the desired bound.

Next, we consider $\w:= \sum_{k=1}^\infty w_k$.  Since $L\w=0$ in $B_0$,
following \cite{KP}, by Caccioppoli's inequality and the Moser type bound
\eqref{eq1.7} (recall that we assume De Giorgi/Nash/Moser estimates,
by hypothesis), we have that, 
\begin{multline}\label{eq4.14}
\fiint_{B_{x_0,t_0}}|\nabla\w|^2 \lesssim \left(\frac{1}{t_0}\fiint_{\tB_{x_0,t_0}} 
|\w-C_{x_0,t_0}|\right)^2\\[4pt]
\lesssim \left(\frac{1}{t_0}\fiint_{\tB_{x_0,t_0}}  |\w(y,s)-\w(y,0)|dyds\right)^2\,
+\,\left(\frac{1}{t_0}\fiint_{\tB_{x_0,t_0}}  |\w(y,0)-C_{x_0,t_0}|dyds\right)^2\\[4pt]
\lesssim \left(\fint_{|x_0-y|<5t_0/8}\,\,\,\sup_{s<13t_0/8} |\partial_s\w(y,s)| \,dy\right)^2
\,+\,\left(\fint_{|x_0-y|<5t_0/8}|\nabla_\|\w(y,0)| dy\right)^2 =:I+II\,,
\end{multline}
where as above, $\tB_{x_0,t_0}:=B((x_0,t_0),5t_0/8)$; in the last line, we  
have used Poincar\'e's inequality to
obtain term $II$ (with $C_{x_0,t_0} := \fint_{|x_0-y|<5t_0/8} \w$). 

We treat term $II$ first.  Observe that
\begin{equation}\label{eq4.15}
II \lesssim \left(M\Big(\nabla w(\cdot,0)\Big)(x_0)\right)^2 
+ \left(\fint_{|x_0-y|<5t_0/8}|\nabla w_0(y,0)| dy\right)^2=:II' + II''\,,
\end{equation}
where we have crudely dominated
$|\nabla_\| w|, |\nabla_\|w_0|$ by the norms of
their respective full gradients.  
The term $II'$ is of the form that we seek in \eqref{eq4.10}.
We handle $II''$ by a similar duality argument,
but now with $h \in L^\infty(\{|x_0-y|<5t_0/8\})$.  We then have that, by \cite[Theorem 1, p. 313]{CMS},
\begin{multline*}
t_0^{-n}\left|\langle\nabla w_0, h\rangle\right|=t_0^{-n}\left|\iint_{\reu}\Phi_0\cdot 
\overline{t\nabla\left(\SL_t^{L^*}\nabla \right)h}\,
\frac{dxdt}{t}\right| \\[4pt]\lesssim \,t_0^{-n} \int_{\rn}\A_2(\Phi_0)(x)\,
\C\left(t\,\nabla\left(\SL_t^{L^*}\nabla \right)h\right)(x)\, dx
\\[4pt]\lesssim\,\fint_{|x_0-x|<8t_0} \A_2(\Phi)(x)\, dx
\,\, \|h\|_\infty\,\lesssim \, M\left( \A_2(\Phi)\right)(x_0)\, \|h\|_\infty\,,
\end{multline*}
where in the next-to-last step we have used Corollary \ref{c3.3} for $L^*$, 
and the fact that the compact support of
$\Phi_0$ contrains in turn the support of $\A_2(\Phi_0)$.
Taking a supremum over all such $h$ with $L^\infty$ norm 1, we have that $II''\lesssim
 M\left( \A_2(\Phi)\right)(x_0)^2$,
so that the desired bound holds for this term also.

It remains to treat term $I$ in \eqref{eq4.14}.  We write
\begin{equation}\label{eq4.16}
\partial_s \w(y,s)\,=\, \Big(\partial_s \w(y,s) -\left(\partial_s \w\right)(y,0) \Big)
\,+\, \left(\partial_s \w\right)(y,0)\,=: \,\mathcal{R}(y,s) \,+ \, \left(\partial_s \w\right)(y,0)\,.
\end{equation}
Notice that $\left(\partial_s \w\right)(y,0)$ contributes to $I$ an expression that satisfies the same bound
as $II$ in \eqref{eq4.15}, which we have already handled.
Recalling that $\w=\sum_{k=1}^\infty w_k$,
we have a corresponding splitting $\rr=\sum_{k=1}^\infty \rr_k$, and we treat each term in the latter sum
separately.  We have, for some $2<p<2+\eps$, with $\eps$ depending only on dimension
and ellipticity, and for $q:=p/(p-1)$, that
\begin{multline*}
|\rr_k(y,s)|=\left|\iint\nabla_{x,t}\Big(\partial_s\Gamma(y,s,x,t)-(\partial_s\Gamma)(y,0,x,t)\Big)
\cdot \Phi_k(x,t) \,dx dt\right|\\[4pt]
\leq \left(\iint_{B_k\setminus B_{k-1}}\left|\nabla_{x,t}\Big(\partial_s\Gamma(y,s,x,t)-
(\partial_s\Gamma)(y,0,x,t)\Big)\right|^p\,dxdt\right)^{1/p}\,
 \left(\iint\left|\Phi_k(x,t)\right|^{q}\,dxdt\right)^{1/q}\\[4pt]
 \lesssim\, 2^{-k} \left(\fiint_{B_k}\left|\Phi(x,t)\right|^q\,dxdt\right)^{1/q}
\end{multline*}
by the estimate of N. Meyers\footnote{We note that the estimate of Meyers continues to hold for complex coefficients:  it is simply a consequence of the self-improvement of the reverse H\"older inequalities for the gradient that one obtains from the inequalities of Caccioppoli and Sobolev.}, Caccioppoli's inequality, and \eqref{eq2.29}, since $|y-x_0|<t_0$ and $s< 2t_0$.  But now, following the computations in \eqref{eq4.13} above,
we find that for some $q<2$,
\begin{equation}\label{eq4.17}
|\rr(y,s)|\leq\sum_{k=1}^\infty|\rr_k(y,s)|\,\lesssim \, 
\sum_{k=1}^\infty2^{-k}\left(M\left(\A_2(\Phi)^q\right)(x_0)\right)^{1/q}\lesssim
\left(M\left(\A_2(\Phi)^q\right)(x_0)\right)^{1/q}\,,
\end{equation}
uniformly in $(y,s)$ as above, 
which yields the desired bound for $I$.
\end{proof}

With the lemma in hand, we present the

\begin{proof}[Proof of Proposition \ref{p4.3}]  We fix a cube $Q\subset \rn$, and an atom adapted
to $Q$ (more precisely, to its Carleson box $R_Q$).
We first observe that \eqref{eq4.4} is an immediate consequence of
Lemma \ref{l4.1}, and the definition of a $T^p_2$ atom \eqref{eq2.16}:
\begin{equation}
\int\limits_{64Q} \left|\N\left(\nabla L^{-1}\nabla \cdot a\right) \right|^2 
\,\lesssim\, \|a\|^2_{T_2^2}\,\lesssim |Q|^{1-2/p}\,.
\label{eq4.18}\end{equation}

We turn now to the ``far-away" estimate \eqref{eq4.5}.  
To lighten the exposition, as before, we shall work with a variant of the non-tangential
maximal function $\N$.
For $F\in L^2_{loc}(\reu)$, we set
$$\left(\widetilde{N}(F)(x)\right)^2 :=
\sup_{\delta>0} \fiint_{B_{\delta}}|F(y,s)|^2 dy ds\,. $$
Here, $B_\delta:=B_{x,\delta} := B((x,\delta),\delta/2).$  
We fix $k\geq 6$, let $x\in V_k:=2^{k+1}Q \setminus 2^kQ$, and for 
$B_\delta=B_{x,\delta}$ as above, we write
\begin{equation}\label{eq4.19}
\fiint\limits_{{B}_{\delta}}\left|\nabla L^{-1}\nabla\cdot a\right|^2 dyds\,=\,
\fiint\limits_{{B}_{\delta}} \Big| \nabla_{y,s} \iint\limits_{R_Q}\nabla_{z,t} \Gamma(y,s,z,t) \cdot
a(z,t) \,dzdt\Big|^2dyds\,.
\end{equation}
We note that, since $x\in V_k$, we have that $w:=L^{-1}\nabla\cdot a$ is a solution of
$Lw=0$ in $B^*_\delta:= B((x,\delta),3\delta/4)$.
For the sake of notational convenience, we set $r:=\ell(Q)$, and 
we split
$$\tN(F)\,\lesssim \,\sup_{\delta \geq 2^{k-3} r}\left(\fiint_{B_\delta}|F|^2 \right)^{1/2}
+\,\sup_{\delta \leq 2^{k-3} r}\left(\fiint_{B_\delta}|F|^2 \right)^{1/2}=:\,\tN_1(F) +\tN_2(F)\,.$$
Consider first $\tN_1$, i.e., the case that
$\delta \geq 2^{k-3} r$.  In this case, by Caccioppoli's inequality in the $y,s$ integral,
followed by Cauchy-Schwarz,  Caccioppoli, and the atomic estimate in the $z,t$ integral, 
we have that \eqref{eq4.19} is bounded by a constant times
\begin{multline*}
(2^k r)^{-2} \fiint\limits_{\tB_{\delta}} \Big| \iint\limits_{R_Q} \nabla_{z,t} \Gamma(y,s,z,t)\cdot 
a(z,t)\,dzdt\Big|^2dyds \qquad {\rm (where} \,\tB_\delta:= B((x,\delta),5\delta/8))\\[4pt]
\lesssim\,(2^k r)^{-2}r^{-2}
 \fiint\limits_{\tB_{\delta}}\left(\iint\limits_{\tR_Q} 
\left |\Gamma(y,s,z,t)-\Gamma(y,s,z_Q,0)\right|^2 \iint_{R_Q}|a|^2\right)dyds\\[4pt]
\lesssim\, 2^{-2\alpha_0k}(2^k r)^{-2n}r^{n}\iint_{R_Q}|a|^2\frac{dzdt}{t}
 \lesssim\,2^{-2\alpha_0k}(2^k r)^{-2n}r^nr^{n(1-2/p)}\,,
\end{multline*}
where in the last pair of estimates, we have used \eqref{eq2.30} and
\eqref{eq2.16}, and the fact that $|\tR_Q|\approx r^{n+1}$.   
Since this bound holds uniformly for $\delta\geq2^{k-3}r$, 
we find that
\begin{equation}\label{eq4.20}
\int_{V_k}\left(\tN_1(\nabla L^{-1}\nabla\cdot a)\right)^2\, \lesssim \,
2^{-k\big(2\alpha_0+n\big)}r^{n(1-2/p)}=2^{-2\epsilon k}\left(2^k r\right)^{n(1-2/p)}\,,
\end{equation}
with $\epsilon:= \alpha_0+n-n/p >0$, since $p>n/(n+\alpha_0)$.  

Next, we consider $\tN_2$, i.e., the case that $\delta\leq 2^{k-3} r$.
Recall that $w:= L^{-1} \nabla\cdot a$.  For $x\in V_k$, we have that $Lw=0$ in $B_\delta
=B_{x,\delta}$.
Following \cite{KP}, by Caccioppoli's inequality and the Moser type bound
\eqref{eq1.7}, we have that, 
\begin{multline}\label{eq4.21}
\fiint_{B_\delta}|\nabla w|^2 \lesssim \left(\frac{1}{\delta}\fiint_{\tB_{\delta}} 
|w-C_{x,\delta}|\right)^2\\[4pt]
\lesssim \left(\frac{1}{\delta}\fiint_{\tB_{\delta}}  |w(y,s)-w(y,0)|\,dyds\right)^2\,
+\,\left(\frac{1}{\delta}\fiint_{\tB_{\delta}}  |w(y,0)-C_{x,\delta}|\,dyds\right)^2\\[4pt]
\lesssim \left(\fint_{|x-y|<5\delta/8}\,\,\,\sup_{s<13\delta/8} |\partial_s w(y,s)| \,dy\right)^2
\,+\,\left(\fint_{|x-y|<5\delta/8}|\nabla_\|w(y,0)| \,dy\right)^2 =:I+II\,,
\end{multline}
where as above, $\tB_{\delta}:=B((x,\delta),5\delta/8)$; in the last line, we 
have used Poincar\'e's inequality to
obtain term $II$ (with $C_{x,\delta} := \fint_{|x-y|<5\delta/8} w$). 

We first consider the contribution of term $I$.  For $s< \delta\leq 2^{k-3} r$, 
and for $|x-y|<\delta\leq 2^{k-3}r$, with $x\in V_k$, by Caccioppoli's inequality
we have that
\begin{multline*}
|\partial_s w(y,s)|^2 \,\leq\, \Big| \iint\limits_{R_Q}\partial_s \nabla_{z,t} \Gamma(y,s,z,t)\cdot 
a(z,t)\,dzdt\Big|^2 \\[4pt]
\lesssim\,r^{-2}
\iint\limits_{\tR_Q} 
\left |\partial_s\Big(\Gamma(y,s,z,t)-\Gamma(y,s,z_Q,0)\Big)\right|^2dz dt\iint_{R_Q}|a|^2\,dzdt\,\\[4pt]
\lesssim\,2^{-2\alpha_0k}(2^kr)^{-2n}\, 
r^n\iint_{R_Q}|a|^2\frac{dzdt}{t}\,\lesssim \lesssim\,2^{-2\alpha_0k}(2^kr)^{-2n}r^{n}r^{n(1-2/p)}\,,
\end{multline*}
where we have used
\eqref{eq2.30} and
\eqref{eq2.16}.  Since this inequality holds uniformly for the range of $y,s$ under consideration,
we obtain the same bound for term $I$, uniformly in $\delta\leq 2^{k-3}r$.
Integrating over $V_k$, we obtain precisely the
same estimate as in \eqref{eq4.20}, as desired.

Finally, we consider term $II$, which for $|x-y|<\delta\leq 2^{k-3}r$, with $x\in V_k$, satisfies
$$II^{1/2}\leq M\left({\bf 1}_{2^{k+2}Q\setminus 2^{k-2}Q}|\nabla_\|w|\right)(x)\,.$$
Therefore, taking a supremum over $\delta\leq 2^{k-3} r$, and 
integrating over $V_k$, we find that the contribution of $II$ is controlled by
\begin{multline*}
\int_{V_k}M\left({\bf 1}_{2^{k+2}Q\setminus 2^{k-2}Q}|\nabla_\|w|\right)^2
\lesssim\, \int_{2^{k+2}Q\setminus 2^{k-2}Q}|\nabla_\|w(y,0)|^2dy\\[4pt]
=\,\int_{2^{k+2}Q\setminus 2^{k-2}Q}\Big| \iint\limits_{R_Q}\nabla_y \nabla_{z,t} \Gamma(y,0,z,t)\cdot 
a(z,t)\,dzdt\Big|^2dy\\[4pt]
\lesssim\,\left(\int_{2^{k+2}Q\setminus 2^{k-2}Q}r^{-1}
\iint\limits_{\tR_Q} 
\left |\nabla_y\Big(\Gamma(y,0,z,t)-\Gamma(y,0,z_Q,0)\Big)\right|^2dz dt\,dy\right)\,
\iint_{R_Q}|a|^2\,\frac{dzdt}{t}\,\\[4pt]
\lesssim \,
2^{-2\alpha_0k}(2^kr)^{-n}r^{n}r^{n(1-2/p)}=
2^{-2\epsilon k}\left(2^kr\right)^{n(1-2/p)}\,,
\end{multline*}
as in \eqref{eq4.20},
where in the last three lines we have used Cauchy-Schwarz and Caccioppoli, \eqref{eq2.40},
and \eqref{eq2.16}. \end{proof}

We now state a corollary of the estimates that we have proved so far.
\begin{corollary}\label{c4.22} 
Suppose, as in Proposition \ref{p4.3} and Proposition \ref{p4.6}, that $L$ and 
$L^*$ satisfy the standard assumptions, are $t$-independent and that $L^*$
satisfies the $L^2$ square function 
bound \eqref{eq1.23}.   Suppose first that $n/(n+\alpha_0)<p\leq1$.  
Let $a$ be a $\CC^{n+1}$-valued $T^p_2$ atom adapted to $Q$.   Then
\begin{align}\label{eq4.23}
\|  \left(\nabla L^{-1} \nabla \cdot a\right)(\cdot,0)\|_{L^p(64Q)}& \leq\, C\,\ell(Q)^{n\big(\frac12-\frac1p\big)}\\[4pt]
\|  \left(\nabla L^{-1} \nabla \cdot a\right)(\cdot,0)\|_{L^p(V_k(Q))}& \leq\, C\,2^{-\epsilon k}
\left(2^k\ell(Q)\right)^{n\big(\frac12-\frac1p\big)}\,,\quad k=6,7,....\label{eq4.24}
\end{align}
Moreover, for every $\Phi\in T^p_2(\reu,\CC^{n+1})$, and $p$ as indicated, we have
\begin{equation}\label{eq4.25}
\int_{\rn} |\left(\nabla L^{-1}\dv \Phi \right)(y,0)|^p\,dy\, \leq\, C\, \|\Phi\|_{T^p_2}^p\,,\quad 
\quad n/(n+\alpha_0)<p\leq 2\,,
\end{equation}
\begin{equation}\label{eq4.26}
\big\|\left(\nabla_\| L^{-1}\dv \Phi \right)(\cdot,0)\big\|_{H^p(\rn)} \leq\, C\, \|\Phi\|_{T^p_2}\,,
\quad n/(n+\alpha_0)<p\leq1\,,
\end{equation}
\begin{equation}\label{eq4.27}
\big\|\left(\vec{N}\cdot A\nabla L^{-1}\dv \Phi\right) (\cdot,0)\big\|_{H^p(\rn)} \leq\, C\, \|\Phi\|_{T^p_2}\,,
\quad n/(n+\alpha_0)<p\leq1\,,
\end{equation}
The constant $C$ in each of these estimates has the same dependence as the constant in \eqref{eq4.7}.
\end{corollary}
\begin{proof}[Sketch of proof]  It is enough to work with $a$ and 
$\Phi$ compactly supported in $\reu$, as the 
class of such functions is dense in $T^p_2$ (and the compact support property is preserved
in the \cite{CMS} atomic decomposition).  For such $a$ and $\Phi$, 
the ``boundary traces" of $\nabla L^{-1}\dv a,\,\nabla L^{-1}\dv \Phi$ make sense (e.g.,
via Lemma \ref{l2.33}), and inherit 
from $\N(\nabla L^{-1}\dv a),\,\N(\nabla L^{-1}\dv \Phi)$ the
claimed bounds  in \eqref{eq4.23}-\eqref{eq4.25}.  To be more precise, Lemma~\ref{l2.33} trivially implies its $L^p$ analogue, for $p\leq 2$, by H\"older's inequality on one side and reverse H\"older estimates on the gradient of the solution on the other. This, in turn, directly entails bounds on the boundary trace in terms of the corresponding non-tangential maximal function in $L^p$, $p\leq 2$, and, respectively, \eqref{eq4.23}-\eqref{eq4.25}. To prove \eqref{eq4.26}-\eqref{eq4.27},
it is enough to decompose $\Phi$ into $T^p_2$ atoms, and to observe that by 
\eqref{eq4.23}-\eqref{eq4.24}, for a $T^p_2$ atom $a$,
the functions $\vec{m}:=\left(\nabla_\| L^{-1}\dv a \right)(\cdot,0)$,
and $m_{\vec{N}}:= \left(\vec{N}\cdot A\nabla L^{-1}\dv a\right) (\cdot,0)$, are
$H^p(\rn)$ molecules (up to a uniform multiplicative constant), since each has mean value zero
(as may be seen by integrating by parts).
\end{proof}

It remains to complete the proof of Proposition  \ref{p4.6}.
We shall require the following variant of Gehring's lemma 
as proved in
Iwaniec \cite{I}.

\begin{lemma}\label{l4.28} Suppose that $0\leq g, h\in L^s(\rn)$,
with $1<s<\infty$, and that for all cubes $Q\subseteq \rn$,
\begin{equation}\label{eq4.29}
\left(\fint_Qg^s\right)^{1/s}
\leq C_0\fint_{3Q} g
+ C_1\left(\fint_{3Q}h^s\right)^{1/s},
\end{equation}
Then, there exists $p=p(n,p,C_0,C_1)>s$ such that
\begin{equation}\label{eq4.30}
\int_{\rn}g^p\leq C\int_{\rn}h^p.
\end{equation} 
\end{lemma}
\begin{proof}[Proof of Proposition \ref{p4.6} in the case $2<p<2+\eps$]
As above, we set $w:=L^{-1}\dv\Phi$.   By \eqref{eq4.10}, we may reduce matters to
proving an $L^p(\rn)$ bound for
$(\nabla w)(\cdot,0)$.
We claim that the conditions of the
lemma are verified, with 
\begin{equation}\label{eq4.31}
s=2\,, \,\,g:=|\nabla w(\cdot,0)|\,, \,{\rm and}\,\, 
h:= \C(\Phi) + \left(M\left(\A_2(\Phi)^q\right)\right)^{1/q}+ |(\partial_t w)(x,0)|\,,
\end{equation}
for some $q<2$.  Given the claim, we observe that
$\|h\|_p\lesssim \|\Phi\|_{T^p_2}$, by \eqref{eq4.12} and  \cite[Theorem 3]{CMS},
since $2<p<\infty$.   The conclusion of Proposition
\ref{p4.6} then follows in the present case.  

We now proceed to establish the claim.  We fix a cube $Q\subset\rn$,
and we split $\Phi=
\Phi_0 + \sum_{k=1}^\infty\Phi_k$, where $\Phi_0 := \Phi 1_{R_{4Q}}$, and where
$\Phi_k:= \Phi 1_{R_{2^{k+2}Q\setminus  2^{k+1}Q}}$, $k\geq 1$.   Let $w_0+\sum w_k$
be the corresponding splitting of $w$, and set $\widetilde{w}:= \sum_{k\geq 1} w_k$.  Of course, this is essentially equivalent to our splitting of $w$ in the proof of Lemma \ref{l4.1}.
Our goal is to prove the bound
\begin{equation}\label{eq4.32}
\left(\fint_Q |\nabla w(\cdot, 0)|^2 \right)^{1/2} \lesssim \fint_{3Q}|\nabla w(\cdot,0)| 
 + \left(\fint_Q |\C(\Phi)|^2 \right)^{1/2} +\fint_{3Q}\, \sup_{|t|\leq\ell(Q)} |\partial_t \widetilde{w}(x,t)| dx\,.
\end{equation}
Let us take the latter estimate for granted momentarily, and show that
it implies  \eqref{eq4.29}, for the particular $g$ and $h$ defined in \eqref{eq4.31}.    
We observe that we need only 
treat the last term, as the others are of the desired form.
We note that this last
term is essentially the same as (the square root of) term $I$ in \eqref{eq4.14}, so that,
exactly as in \eqref{eq4.16}-\eqref{eq4.17}, we have that
\begin{multline*}
\sup_{|t|\leq\ell(Q)}|\partial_t \widetilde{w}(x,t)|\,\lesssim\,
\inf_{z\in Q}\left(M\left(\A_2(\Phi)^q\right)(z)\right)^{1/q} + |(\partial_t \widetilde{w})(x,0)|\\[4pt]
\lesssim\,\inf_{z\in Q}\left(M\left(\A_2(\Phi)^q\right)(z)\right)^{1/q}+
|(\partial_t w)(x,0)|+|(\partial_t w_0)(x,0)|
\end{multline*}
(here, we are using the non-centered maximal function).  The first and second of these terms contribute the desired bound as per the definition of $h$;   the contribution of the third is controlled by the stronger
estimate
\begin{multline}\label{eq4.33}
\left(\fint_{3Q} |(\nabla w_0)(x,0)|^2\right)^{1/2}\\[4pt]
=\,\left(\fint_{3Q} |(\nabla L^{-1}\dv \Phi_0)(x,0)|^2\right)^{1/2}\lesssim\,
\left(\frac1{|Q|}\iint_{R_{4Q}} |\Phi(x,t)|^2\frac{dxdt}{t}\right)^{1/2} \lesssim \inf_{z\in Q} \C(\Phi)\,,
\end{multline}
by \eqref{eq4.25} and the definition of $\Phi_0$ (and where we are using a 
non-centered version of $\C$).  Gathering our estimates, we see that
\eqref{eq4.32} implies \eqref{eq4.29} with $s,g,h$ as in \eqref{eq4.31}.  

Therefore, it remains only to
establish \eqref{eq4.32}.
To this end, we observe that, by \eqref{eq4.33},
$$\left(\fint_Q|\big(\nabla w\big)(x,0)|^2 dx\right)^{1/2} \lesssim
\left(\fint_Q|\big(\nabla \widetilde{w}\big)(x,0)|^2 dx\right)^{1/2} +\left(\fint_Q |\C(\Phi)|^2\right)^{1/2}\,.$$
Since the last term is an appropriate bound, we need only consider the $\widetilde{w}$ term.
By construction $L\widetilde{w}=0$ in $\Omega_{4Q}:=4Q\times (-4\ell(Q),4\ell(Q))$.  Thus, by
\eqref{eq2.35} (with $p=2$), and \eqref{eq1.7}, and then 
following an argument of \cite{KP}, we have
\begin{multline*}
\left(\fint_Q|\big(\nabla \widetilde{w}\big)(x,0)|^2 dx\right)^{1/2}
\lesssim \frac1{\ell(Q)}\fint_{3Q}\fint_{-\ell(Q)}^{\ell(Q)}| \widetilde{w}(x,t)-c_Q| dtdx\\[4pt]
\lesssim \frac1{\ell(Q)}\fint_{3Q}\fint_{-\ell(Q)}^{\ell(Q)}| \widetilde{w}(x,t)-\widetilde{w}(x,0)|\, dtdx\,
+\,\frac1{\ell(Q)}\fint_{3Q}| \widetilde{w}(x,0)-c_Q| \,dx\\[4pt]
\lesssim \fint_{3Q}\,\sup_{|t|\leq \ell(Q)}|\partial_t \widetilde{w}(x,t)|\,dx + \fint_{3Q} |\nabla_\| w(x,0)|\,dx
+\fint_{3Q} |\nabla_\| w_0(x,0)|\,dx\,,
\end{multline*}
where in the last step, having chosen the constant $c_Q$ appropriately, we have used Poincar\'e's
inequality, and then split $\widetilde{w}=w-w_0$.  Crudely dominating the tangential gradient by the full gradient, and once again applying \eqref{eq4.33}, we obtain  \eqref{eq4.32}.
\end{proof}

We conclude this section with two remarks.
\begin{remark}\label{r4.34}
Having completed the proof of Proposition \ref{p4.6},
we note that \eqref{eq4.25} may now be extended to the range
$n/(n+\alpha_0)<p<2+\eps$.
\end{remark}  
\begin{remark}\label{r4.35}
The upper bound $p<2+\eps$ in Proposition \ref{p4.6}
is optimal in the following sense.  Observe that by Remark \ref{r4.34} and duality
(and interchanging the roles of $L$ and $L^*$), we have the square function bound
\begin{equation}\label{eq4.36}
\int_{\rn}\left(\iint_{|x-y|<t}|\nabla\left( \SL_t\nabla\right) f(y)|^2\frac{dydt}{t^{n-1}}\right)^{p/2}\leq C_p
\|f\|_{L^p(\rn)}\,,\qquad 2-\eps_1<p<\infty\,, 
\end{equation}
where $2-\eps_1$ is dual to $2+\eps$.   Let us now specialize to the case that the 
$t$-independent coefficients of $L$ are real symmetric.  Then with 
$u(x,t): = (\SL_t \nabla)f(x)$, with $f\in L^p$, 
by \cite{DJK}, we have that 
$$\int |u(x,0)|^p \leq \int N_*(u)^p  \lesssim \int_{\rn} S(u)^p\lesssim \int_{\rn}|f|^p
\,,\qquad 2-\eps_1<p<\infty\,,$$  where 
$S(u):= \A(t\nabla u)$ 
denotes the conical square function, and the last step is \eqref{eq4.36}.  But this says that
$(\SL_t \nabla)\big|_{t=0}:L^p\to L^p$, by definition of $u$.  
Now further specialize to the ``block" case (i.e., the case that
$A_{n+1,j}=0=A_{j,n+1}, 1\leq j\leq n$), so that
 $S_t = \frac12 e^{-t\sqrt{L_\|}}L_\|^{-1/2}$, 
 where $L_\|:=-\sum_{i,j=1}^n \partial_{x_i}A_{i,j}\partial_{x_j}$ is the $n$-dimensional operator formed from the ``upper left" $n\times n$ block of the coefficient matrix $A$.
Thus, $(\SL_t \nabla)\big|_{t=0} = L_\|^{-1/2}\nabla$,
 which is the adjoint of the $L_\|$-adapted Riesz transform.  
 By the examples of Kenig (which may be found in
 \cite{AT}), given $\eps_1>0$, there is an $L_\|$ with real symmetric  coefficients,
 for which the adjoint Riesz transform is not bounded on $L^p$ with
 $p=2-\eps_1$.
\end{remark}

\section{The third main estimate:  tent space bounds for $\nabla L^{-1} \textstyle{\frac 1t} $}
\label{s5}

Let us now turn to the third main estimate, which will be the core of our approach to the Dirichlet problem.

\begin{proposition}\label{p5.1} Suppose that $L$ and $L^*$ satisfy the standard assumptions and are  $t$-independent. Then 
\begin{equation}\label{eq5.2}
 \nabla L^{-1} \textstyle{\frac 1t}: \widetilde T^p_1\to \widetilde{T}^p_\infty, \quad 1<p< 2+\eps,
\end{equation}
\noindent for some $\eps>0$, depending on the standard constants only. Here the operator  $ L^{-1} \textstyle{\frac 1t}$ is to be interpreted via
\begin{equation}\label{eq5.3} \left(L^{-1}\textstyle{\frac 1t} \Psi\right) (y,s):=\iint_{\reu} \Gamma(y,s; x,t) \,\Psi(x,t)\,\frac{dxdt}{t}, \quad (y,s)\in \reu. \end{equation}
\end{proposition}

\bp The argument follows the general lines of the proof of Lemma~\ref{l4.1}. Let $\Phi\in \widetilde T^p_1(\reu)$.  By density, it is again sufficient to work with $\Phi\in C_0^\infty(\reu)$ and we let $w:=L^{-1} \textstyle{\frac 1t}\Phi$. 

Next, fix $(x_0,t_0)\in \reu$ and split $\Phi=\Phi_0'+\Phi_0''+\sum_{k=1}^\infty \Phi_k$ according to dyadic annuli of $B((x_0,0),4t_0)$ (this is the same notation as in the proof of Lemma~\ref{l4.1}, in particular, $\Phi_0'$ is supported in a Whitney cube and  $\Phi_0''$ is supported in a Carleson region below it), and, respectively, $w=w_0'+w_0''+\w=w_0'+w_0''+\sum_{k=1}^\infty w_k$. 

First of all, by Sobolev embedding $\nabla L^{-1}: L^{\frac{2(n+1)}{n+3}}(\ree)\to L^2(\ree)$ and hence, 
$$\fiint_{B_{x_0,t_0}}|\nabla w'_0|^2\lesssim \frac1{t_0^{n+1}}\left(\iint_{B_0\cap \{t\geq t_0/4\}}
|\Phi_0'/t|^{\frac{2(n+1)}{n+3}}\right)^{\frac{n+3}{(n+1)}}. $$
Now we use the fact that $\Phi_0'$ is supported in a Whitney region around $(x_0,t_0)$, so that in particular $t\approx t_0$ within the region of integration, and H\"older inequality to bound the expression above by 
$$
\frac1{t_0^{n}}\iint_{B_0\cap \{t\geq t_0/4\}} |\Phi_0'(x,t)|^2 \frac{dxdt}{t}\lesssim {\widetilde{\mathfrak C}}_1(\Phi)(x_0)^2,$$
uniformly in $t_0$. (See Section~\ref{s2.2.3} for notation). 

Turning to $w_0'',$ we observe that $L^{-1}(\Phi''_0/t)$ is a solution in $B((x_0,t_0), 3t_0/4)$, and hence,  by Caccioppoli inequality, 
\begin{multline}\fiint_{B_{x_0,t_0}}|\nabla w''_0|^2\lesssim \frac1{t_0^{n+3}}\iint_{\widetilde B_{x_0,t_0}}
|L^{-1}(\Phi_0''/t)|^{2}\\[4pt]
\lesssim \frac1{t_0^{n+3}}\iint_{\widetilde B_{x_0,t_0}}
\left|\iint_{{B_0\cap \{t\leq t_0/4\}}} \Gamma (x,t; y,s)\Phi_0''(y,s)\frac{dyds}{s}\right|^{2} \\[4pt]
\lesssim \frac1{t_0^{n+3}}\iint_{\widetilde B_{x_0,t_0}}
\left|\iint \left(\chi_{B_0\cap \{t\leq t_0/4\}}\Gamma (x,t; \cdot, \cdot)\right)_W(y,s)\Phi''_{0,W}(y,s)\frac{dyds}{s}\right|^{2}\\[4pt]
\lesssim 
\left|\frac{1}{t_0^n}\iint \Phi''_{0,W}(y,s)\frac{dyds}{s}\right|^{2}\lesssim {\widetilde{\mathfrak C}}_1(\Phi)(x_0)^2, \end{multline}
\noindent (again, see Lemma~\ref{l4.1} and Section~\ref{s2.2.3} for notation). Here, in the fourth inequality we exploited the separation between $\widetilde B_{x_0,t_0}$ and the support of $\left(\chi_{B_0\cap \{t\leq t_0/4\}}\Gamma (x,t; \cdot, \cdot)\right)_W$,  as well as pointwise estimates on $\Gamma$. It is tacitly assumed throughout the argument that averaging denoted by the subscript $W$ is performed on  Whitney cubes $W(x,t)$ of diameter $ct$, with $c$ small enough to not obscure such a separation.  

It remains to address $\w$. Similarly to \eqref{eq4.14}, we have
\begin{multline*}
\fiint_{B_{x_0,t_0}}|\nabla\w|^2 
\lesssim \left(\frac{1}{t_0}\fiint_{\tB_{x_0,t_0}}  |\w(y,s)-\w(y,0)|dyds\right)^2\,
+\,\left(\frac{1}{t_0}\fiint_{\tB_{x_0,t_0}}  |\w(y,0)-C_{x_0,t_0}|dyds\right)^2\\[4pt]
\lesssim \left(\fint_{|x_0-y|<5t_0/8}\,\,\,\fint_0^{13t_0/8} |\partial_\tau\w(y,\tau)|\, d\tau \,dy\right)^2
\,+\,\left(\fint_{|x_0-y|<5t_0/8}|\nabla_\|\w(y,0)| dy\right)^2\\[4pt]
\lesssim \left(\fint_{|x_0-y|<5t_0/8}\,\,\,\fint_0^{13t_0/8} |\partial_\tau\w(y,\tau)-\partial_\tau\w(y,0)|\, d\tau \,dy\right)^2\\[4pt]
\,+\,\left(\fint_{|x_0-y|<5t_0/8}|\nabla \w(y,0)| dy\right)^2 =:I^2+II^2\,,
\end{multline*}
\noindent with a suitably chosen $C_{x_0,t_0}$ to ensure Poincar\'e inequality.  Note a slightly different separation into $I$ and $II$ compared to \eqref{eq4.14}. 
Just as in  \eqref{eq4.15}, 
\begin{equation*}
II \lesssim M\Big(\nabla w(\cdot,0)\Big)(x_0)
+ \fint_{|x_0-y|<5t_0/8}|\nabla w_0(y,0)| \,dy=:II' + II''\,.
\end{equation*}

Let us start with $II'$. Let $h\in L^{p'}(\rn, \CC^{n+1})$, so that
\begin{multline*}
\left|\langle\nabla w(\cdot,0), h\rangle\right|=\left|\langle\nabla L^{-1} (\Phi/t)(\cdot,0), h\rangle\right|=\left|\iint_{\reu}\Phi(x,t)\cdot 
\overline{\left(\SL_t^{L^*}\nabla \right)h(x)}\,
\frac{dxdt}{t}\right|\\[4pt]\lesssim \|\Phi\|_{\widetilde T^p_1} \,
\|(\SL_t^{L^*}\nabla)h\|_{\widetilde T^{p'}_{\infty}}
\lesssim \|\Phi\|_{\widetilde T^p_1} \,
\|h\|_{L^{p'}}
\end{multline*}
where in the last step we have used \eqref{eq1.15} for $L^*$ and its companion for $\SL_t \nabla_\|$ in place of $\D_t$ (proved in the $t$-independent case in \cite{HMiMo}).  One can note that \eqref{eq1.15} provides an estimate in $T^{p'}_{\infty}$, not $\widetilde T^{p'}_{\infty}$, but since $(\SL_t^{L^*}\nabla)h$ is a solution in $\reu$, the corresponding norms are equivalent by Moser bounds. Taking a supremum over all $h\in L^{p'}(\rn, \CC^{n+1})$ with norm 1, we obtain that $\|M\big(\nabla w( \cdot,0)\big)\|_{L^{p}}\lesssim \|\Phi\|_{\widetilde T^p_1}$.

To handle $II''$, take now $h \in L^\infty(\{|x_0-y|<5t_0/8\})$.  Then
\begin{equation*}
t_0^{-n}\left|\langle\nabla w_0, h\rangle\right|=t_0^{-n}\left|\iint_{\reu}\Phi_0(x,t)\cdot 
\overline{\left(\SL_t^{L^*}\nabla \right)h(x)}\,
\frac{dxdt}{t}\right| \lesssim t_0^{-n}\|\Phi_0\|_{\tT_1^r} \|(\SL_t^{L^*}\nabla)h\|_{\tT^{r'}_{\infty}}\end{equation*}
\noindent where $r$ is taken so that $1<r<p$. Note that $\|\Phi_0\|_{\tT_1^r}=\|\A_1(\Phi_{0,W})\|_{L^r}\approx \|\A_1(\Phi_{0,W})\|_{L^r(B(x_0,10t_0)}$, due to the support of $\Phi_0$. Hence, employing bounds on $(\SL_t^{L^*}\nabla)h$ once again, we have
$$ II'' \lesssim \left(M((\A_1(\Phi_{W}))^r)(x_0)\right)^{1/r}, \quad 1<r<p,$$
which will end up in the same contribution as $II'$.

Finally, let us turn to $I$. We have for $|x_0-y|<5t_0/8$, $0<\tau<13t_0/8$, 
\begin{multline*}
|\partial_\tau \w(y,\tau)-\partial_\tau \w(y,0)| \lesssim \sum_{k=1}^\infty  |\partial_\tau w_k(y,\tau)-\partial_\tau w_k(y,0)|\\[4pt]
 \lesssim \sum_{k=1}^\infty
\left|\iint\Big(\partial_\tau\Gamma(y,\tau,x,t)-(\partial_\tau\Gamma)(y,0,x,t)\Big)
\cdot \Phi_k(x,t) \,\frac{dx dt}{t}\right|\\[4pt]
\lesssim \sum_{k=1}^\infty \sup_{(x,t)\in B_k\setminus B_{k-1}}\Big(\partial_\tau\Gamma(y,\tau,\cdot,\cdot)-
(\partial_\tau\Gamma)(y,0,\cdot,\cdot)\Big)_W(x,t)\,
\iint_{B_k\setminus B_{k-1}}\left|\Phi_{k,W}(x,t)\right|\,\frac{dxdt}{t}\\[4pt]
\lesssim \sum_{k=1}^\infty 2^{-k\alpha_0} \frac{1}{(2^kt_0)^n}
\iint_{B_k\setminus B_{k-1}}\left|\Phi_{k,W}(x,t)\right|\,\frac{dxdt}{t}\lesssim {\widetilde{\mathfrak C}}_1(\Phi)(x_0).\end{multline*}
This finishes the proof. 
\ep

\section{The proof of Theorem~\ref{t1.13}}\label{s6}

In the present section we prove Theorem~\ref{t1.13} and establish some accompanying results extending our Main Estimates from Sections~\ref{s3}--\ref{s5} to the case of coefficients depending on the transversal direction $t$. We shall often write $L_1$ in place of $L$, to underline the difference with $L_0$. 

Let us start stating the following auxiliary Lemma. 
\begin{lemma}\label{l6.0.1} Assume that operator $L=L_1=-\nabla\cdot A \nabla$ satisfies the conditions on Theorem~\ref{t1.13}. Then for every $\Phi\in C_0^\infty (\reu)$ the quantity $\|\nabla L^{-1} \Phi\|_{\tT^p_\infty}$, $p\in (p_0, 2+\eps)$, is finite. 
\end{lemma} 

We postpone the proof of the Lemma to the end of this Section. For now, let us proceed with the perturbation results.

\begin{lemma}\label{l6.26}
Let $L=L_1$ be an elliptic operator satisfying assumptions of Theorem~\ref{t1.13}. Then
\begin{equation}
\nabla L_1^{-1} \dv : T^p_2 \to \widetilde{T}^p_\infty, \quad n/(n+\alpha_0)<p< 2+\eps.
\label{eq6.27}\end{equation}
Moreover, for every $\Phi\in T^p_2(\reu,\CC^{n+1})$, and $p$ as indicated, we have
\begin{equation}\label{eq4.25.1}
\int_{\rn} \Big|\left(\nabla L_1^{-1}\dv \Phi \right)(y,0)\Big|^p\,dy\, \leq\, C\, \|\Phi\|_{T^p_2}^p\,,\quad 
\quad n/(n+\alpha_0)<p< 2+\eps\,,
\end{equation}
\begin{equation}\label{eq4.26.1}
\big\|\left(\nabla_\| L_1^{-1}\dv \Phi \right)(\cdot,0)\big\|_{H^p(\rn)} \leq\, C\, \|\Phi\|_{T^p_2}\,,
\quad n/(n+\alpha_0)<p\leq1\,,
\end{equation}
\begin{equation}\label{eq4.27.1}
\big\|\left(\vec{N}\cdot A_1\nabla L_1^{-1}\dv \Phi\right) (\cdot,0)\big\|_{H^p(\rn)} \leq\, C\, \|\Phi\|_{T^p_2}\,,
\quad n/(n+\alpha_0)<p\leq1\,.
\end{equation}
\end{lemma}

\bp By density, it is sufficient to estimate $\nabla L_1^{-1} \dv \Phi$ for  $\Phi\in C_0^\infty(\reu)$. Let us recall that for such $\Phi$ the quantity $\|\nabla L_1^{-1} \dv \Phi\|_{\widetilde{T}^p_\infty}$ is finite by Lemma~\ref{l6.0.1}. Furthermore, 
\begin{equation}\label{eq6.27.1}
L_0^{-1}-L_1^{-1}= L_0^{-1}L_1 L_1^{-1}-L_0^{-1}L_0 L_1^{-1}
=L_0^{-1} (L_1-L_0) L_1^{-1}=L_0^{-1} \dv(A^0-A^1)\nabla L_1^{-1}.
\end{equation}
Hence, by Proposition~\ref{p4.6} and Lemma~\ref{l2.20},
\begin{align} \|\nabla L_1^{-1} \nabla \cdot \Phi\|_{\widetilde{T}^p_\infty} &\leq \|\N(\nabla (L_1^{-1}-L_0^{-1}) \nabla \cdot \Phi)\|_{L^p} +\|\N(\nabla L_0^{-1} \nabla \cdot \Phi)\|_{L^p}\label{eq6.29}\\
& \leq \|\N(\nabla L_0^{-1} \nabla \cdot (A^1-A^0) \nabla L_1^{-1} \nabla \cdot \Phi\|_{L^p}+ \|\nabla L_0^{-1} \nabla\|_{T^p_2 \to \widetilde{T}^p_\infty}\|\Phi\|_{T^p_2}\notag\\
&\leq \|\nabla L_0^{-1} \nabla\|_{T^p_2 \to \widetilde{T}^p_\infty} \|(A^1-A^0) \nabla L_1^{-1} \nabla \cdot \Phi\|_{T^p_2} +\|\nabla L_0^{-1} \nabla\|_{T^p_2 \to \widetilde{T}^p_\infty}\|\Phi\|_{T^p_2}\notag\\
&\leq \varepsilon_0 \|\nabla L_0^{-1} \nabla\|_{T^p_2 \to \widetilde{T}^p_\infty} \| \nabla L_1^{-1} \nabla \cdot \Phi\|_{\widetilde{T}^p_\infty}+\|\nabla L_0^{-1} \nabla\|_{T^p_2 \to \widetilde{T}^p_\infty} \|\Phi\|_{T^p_2}.\notag
\end{align}
Now, assuming that $\varepsilon_0$ is sufficiently small compared to $1/\|\nabla L_0^{-1} \nabla\|_{T^p_2 \to \widetilde{T}^p_\infty}$, the estimate \eqref{eq6.29} implies
$$ \|\nabla L_1^{-1} \nabla \cdot \Phi\|_{\widetilde{T}^p_\infty} \lesssim \|\Phi\|_{T^p_2},$$
as desired.

A very similar argument, now using \eqref{eq6.27}, in combination with the results of Corollary~\ref{c4.22} and Remark~\ref{r4.34} allows us to make sense of boundary traces of $\nabla L_1^{-1} \nabla \cdot \Phi$ for  $\Phi \in T^p_2$ and to establish an analogue of  \eqref{eq4.25} (the latter in the extended range $n/(n+\alpha_0)<p< 2+\eps$) for the operator $L_1$: 
\begin{align*} \|\nabla L_1^{-1} \nabla \cdot \Phi(\cdot, 0)\|_{L^p(\rn)} &\leq \|\nabla (L_1^{-1}-L_0^{-1}) \nabla \cdot \Phi)(\cdot, 0)\|_{L^p(\rn)}  +\|\nabla L_0^{-1} \nabla \cdot \Phi(\cdot, 0)\|_{L^p(\rn)} \\
& \leq \|\nabla L_0^{-1} \nabla \cdot (A^1-A^0) \nabla L_1^{-1} \nabla \cdot \Phi(\cdot, 0)\|_{L^p(\rn)} + C\|\Phi\|_{T^p_2}\notag\\
&\lesssim \varepsilon_0 \| \nabla L_1^{-1} \nabla \cdot \Phi\|_{\widetilde{T}^p_\infty}+\|\Phi\|_{T^p_2}\lesssim \|\Phi\|_{T^p_2}.\end{align*}
And \eqref{eq4.26.1}--\eqref{eq4.27.1} are proved similarly, invoking \eqref{eq4.26}, \eqref{eq4.27}.
\ep

Essentially the same argument applies to show 
\begin{lemma}\label{l6.29.1}
Let $L=L_1$ be an elliptic operator satisfying assumptions of Theorem~\ref{t1.13}. Then
\begin{equation}
 \nabla L^{-1} \textstyle{\frac 1t}: \widetilde T^p_1\to \widetilde{T}^p_\infty, \quad 1<p< 2+\eps.
\label{eq6.29.2}\end{equation}
\end{lemma}

\bp Once again, by density, it is sufficient to estimate $\nabla L_1^{-1} \textstyle{\frac 1t} \Phi$ for  $\Phi\in C_0^\infty(\reu)$, and due to Lemma~\ref{l6.0.1} the quantity $\|\nabla L_1^{-1} \textstyle{\frac 1t}  \Phi\|_{\widetilde{T}^p_\infty}$ is finite for such $\Phi$. Just as in the proof of Lemma~\ref{l6.26}, 
\begin{align} \|\nabla L_1^{-1} \textstyle{\frac 1t} \Phi\|_{\widetilde{T}^p_\infty} &\leq \|\N(\nabla (L_1^{-1}-L_0^{-1}) \textstyle{\frac 1t}\Phi)\|_{L^p} +\|\N(\nabla L_0^{-1} \textstyle{\frac 1t} \Phi)\|_{L^p}\label{eq6.29.3}\\
&\leq \varepsilon_0 \|\nabla L_0^{-1} \nabla\|_{T^p_2 \to \widetilde{T}^p_\infty} \| \nabla L_1^{-1} \textstyle{\frac 1t}\Phi\|_{\widetilde{T}^p_\infty}+\|\nabla L_0^{-1} \textstyle{\frac 1t}\|_{\tT^p_1 \to \widetilde{T}^p_\infty} \|\Phi\|_{\tT^p_1}.\notag
\end{align}
Assuming that $\varepsilon_0$ is sufficiently small compared to $1/\|\nabla L_0^{-1} \nabla\|_{T^p_2 \to \widetilde{T}^p_\infty}$, this furnishes the desired result.
\ep

We now turn to the proof of Theorem~\ref{t1.13}.

\noindent {\it Proof of Theorem~\ref{t1.13}}. 

\vskip 0.08 in \noindent {\bf Step I. Proof of \eqref{eq1.14}. } For any elliptic operator $L$ (not necessarily $t$-independent) we can define the single layer potential acting on $C_0^\infty$ functions as follows. For every 
$f\in C_0^\infty(\rn)$, $\Psi\in C_0^\infty(\reu)$ let
\begin{equation}\label{eq6.29.4}\langle \nabla \SL^{L}_s f, \Psi \rangle:=
\langle f, \tr\circ (L^*)^{-1}\dv \Psi\rangle,
\end{equation} 
\noindent where $\tr$ denotes the trace operator. For any elliptic operator $L$,
\begin{equation}\label{eq6.29.5}
\nabla (L^*)^{-1}\dv :L^2(\ree)\mapsto L^2(\ree),
\end{equation}
i.e., $(L^*)^{-1}\dv :L^2(\ree)\mapsto \dot{W}^{1,2}(\ree)$, so that by the trace theorem,
$$\tr \circ (L^*)^{-1} \dv: L^2(\ree)\mapsto \dot{H}^{1/2}(\rn)\,.$$ This justifies \eqref{eq6.29.4}. 

It follows from our definition \eqref{eq6.29.4} and \eqref{eq6.27.1} that 
for every $f\in C_0^\infty(\rn)$, $\Psi\in C_0^\infty(\reu)$ the identity 
\begin{multline}\label{eq6.29.6} \langle \nabla \SL^{L_1}_s f - \nabla \SL^{L_0}_s f, \Psi \rangle=\langle f, \tr\circ ((L_1^*)^{-1}-(L_0^*)^{-1})\dv \Psi\rangle\\[4pt]
=\langle f, \tr\circ ((L_0^*)^{-1}\dv (A_1^*-A_0^*)\nabla (L_1^*)^{-1})\dv \Psi\rangle=\langle \nabla \SL^{L_0}_s f,  (A_1^*-A_0^*)\nabla (L_1^*)^{-1}\dv \Psi\rangle\\[4pt]
=\langle\nabla L_1^{-1}\dv (A_1-A_0) \nabla \SL^{L_0} f, \Psi\rangle,
\end{multline} 
\noindent is valid. It follows that for every $f\in C_0^\infty(\rn)$
\begin{equation}\label{eq6.29.7}
\nabla \SL^{L_1}_s f - \nabla \SL^{L_0}_s f=\nabla L_1^{-1}\dv (A_1-A_0) \nabla \SL^{L_0} f \quad {\mbox{a.e. in }} \reu.
\end{equation}

Having \eqref{eq6.29.7} at hand, we simply invoke \eqref{eq1.14} for the $t$-independent operator $L_0$ (established in \cite{HMiMo}), Lemma~\ref{l2.20}, and Lemma~\ref{l6.26}, to obtain 
\begin{equation}\label{eq6.29.8}\|\nabla \SL^{L_1}_s f - \nabla \SL^{L_0}_s f\|_{\tT^p_\infty} =\|\nabla L_1^{-1}\dv (A_1-A_0) \nabla \SL^{L_0} f\|_{\tT^p_\infty}\lesssim \eps_0\|f\|_{H^p}, \quad p_0<p<2+\eps, 
\end{equation}
for every $f\in C_0^\infty(\rn)$ and then for all $f\in H^p$ by density. This entails \eqref{eq1.14} for $L=L_1$.

\vskip 0.08 in \noindent {\bf Step II. Proof of \eqref{eq1.16}--\eqref{eq1.19}. } The estimates \eqref{eq1.16}--\eqref{eq1.19} are essentially the dual versions of the bounds for $\nabla_A L^{-1}\nabla$ (recall the notation from Section~\ref{s2.1}), restricted to the boundary. There is no need to distinguish $L_0$ and $L_1$. Indeed, for $1<p<2+\eps$,
\begin{multline*}\|t\nabla\SL_t^{L^*}\nabla f\|_{T^{p'}_2}=\sup_{\Phi\in T^p_2, \,\|\Phi\|=1}\iint_{\reu}\int_{\rn}\nabla_{x,t}\nabla_{y,s}\Gamma_{L^*}(x,t,y,0)f(y)\,dy\,\overline{\Phi(x,t)}\,dxdt\\[4pt]
=\sup_{\Phi\in T^p_2, \,\|\Phi\|=1}\int_{\rn} (\nabla L^{-1}\nabla\cdot\Phi) (x,0)f(x)\,dx \lesssim 
\|\Phi\|_{T^p_2} \|f\|_{L^{p'}},
\end{multline*}

\noindent using \eqref{eq4.25}, Remark~\ref{r4.34}, \eqref{eq4.25.1}. 
The same duality argument gives \eqref{eq1.19} and \eqref{eq1.16} as a consequence of \eqref{eq4.26}, \eqref{eq4.26.1} and \eqref{eq4.27}, \eqref{eq4.27.1}, respectively.

\vskip 0.08 in \noindent {\bf Step III. Proof of \eqref{eq1.15} and \eqref{eq1.15.1}}. Let us start with \eqref{eq1.15} and, respectively, $1<p<2+\eps$. In the $t$-independent case $L=L_0$ the bound \eqref{eq1.15} was proved in \cite{HMiMo}. We note that for any operator, $t$-independent or not, the double layer potential is the dual of $\partial_{\nu_{A_i}} L_i^{-1}\textstyle{\frac 1t}$, properly interpreted on the boundary. Indeed, for $\Phi\in C_0^\infty(\reu)$
\begin{equation}\label{eq6.21.2}
\iint_{\reu} \mathcal{D}^{L^*_i}_tf(x)\, \overline{\Phi(x,t)} \,\frac{dxdt}{t}=\int_{\rn} f(y)\, \overline{\vec{N} \cdot A^{i}(y,0)\left( \nabla L_i^{-1} \textstyle{\frac 1t} \Phi\right)(y,0)} \,dy.
\end{equation}

This duality (or considerations similar to Corollary~\ref{c4.22}) can be used to define $\partial_{\nu_{A_0}} L_0^{-1}\textstyle{\frac 1t}\Phi$ on the boundary and to justify that
$$\|\partial_{\nu_{A_0}} L_0^{-1}\textstyle{\frac 1t}\Phi\|_{L^p}\lesssim \|\Phi\|_{\tT^p_1}. $$
In fact, more generally, by the duality with the \cite{HMiMo} estimates for $\SL^{L_0}\nabla$, we get 
$$\|\nabla L_0^{-1}\textstyle{\frac 1t}\Phi(\cdot, 0)\|_{L^p}\lesssim \|\Phi\|_{\tT^p_1}. $$

Now we can write 
\begin{equation}\label{eq6.21.4}
\nabla (L_1^{-1}-L_0^{-1})\textstyle{\frac 1t}\Phi=\nabla L_0^{-1}\dv (A_1-A_0)\nabla L_1^{-1}\textstyle{\frac 1t}\Phi,
\end{equation}

\noindent where formula itself makes sense for all $\Phi\in C_0^\infty(\reu)$ and then we can take restriction to the boundary on the right-hand side and obtain 
$$\|\nabla (L_1^{-1}-L_0^{-1})\textstyle{\frac 1t}\Phi(\cdot, 0)\|_{\tT^p_\infty}\lesssim \|\Phi\|_{\tT^p_1}, $$
 using Lemma~\ref{l6.29.1}, Lemma~\ref{l2.20}, Corollary~\ref{c4.22}, and Remark~\ref{r4.34}. This entails the desired estimate on $\nabla L_1^{-1}\textstyle{\frac 1t}\Phi(\cdot, 0)$ and hence, by duality, on $\mathcal{D}^{L^*_1}_tf$ in \eqref{eq1.15}.

To be precise, the duality considerations above yield equivalence of the estimates for $\partial_{\nu_{A_i}} L_i^{-1} \textstyle{\frac 1t}$ on the boundary and the estimate
 $\mathcal{D}^{L^*_i}_t: L^{p'}\to \tT^{p'}_\infty$. We, instead, depart from  $\mathcal{D}^{L^*_0}_t: L^{p'}\to T^{p'}_\infty$ and aim at $\mathcal{D}^{L^*_1}_t: L^{p'}\to T^{p'}_\infty$. 
However, recalling that the double layer is a solution and solutions satisfy Moser estimates, we can replace $\N$ by $N_*$ to remove or implement the extra averaging encoding the difference between $\tT^{p'}_\infty$ and $T^{p'}_\infty$ in the present context for free.  

Let us turn to \eqref{eq1.15.1} . Applying Meyers' characterization \eqref{eq2.12} to $\ree$ and restricting to $\reu$, we see that it is sufficient to prove that for every $(n+1)$-dimensional cube $I=Q\times [t_0, t_0+l(Q)]$, $t_0\geq 0$, there exists a constant $c_I$ such that 
\begin{equation}\label{eq6.21.3}\frac{1}{l(I)^\beta}\left( \fiint_I |\mathcal{D}_s f(y)-c_I|^2 dyds\right)^{1/2} \leq C \|f\|_{\Lambda_\beta(\rn)}\,.\end{equation}

Recall the following result.
\begin{lemma}\label{l6.34}
Let $D$ be a bounded domain in $\mathbb{R}^d$, whose boundary is locally the graph of a Lipschitz function. Then every $u \in L^1_{loc}(D)$ such that $\nabla u \dist(\cdot, \partial D) \in W^{1,p}(D)$
satisfies the following type of Poincar\'{e} inequality. There exists a constant $u_D$ such that 
\begin{equation}
\|u-u_D\|_{L^p(D)} \leq C\| \nabla u \dist(\cdot, \partial D)\|_{L^p(D)},
\label{eq6.35}\end{equation}
where $\dist(x, \partial D)$ is the distance from $x \in D$ to the boundary of $D$ and $C$ is a constant depending on the character of $D$ and not depending on its size. 
\end{lemma}
The proof of the lemma can be found in \cite{BS}, \cite{HuS}. The authors prove the result in considerably more general domains and for more general powers of distance to the boundary. The dependence of the constant $C$ on particular features of the domain is not formally stated but is evident from the proof in \cite{HuS}. For a collection of cubes  $C$ would be a universal constant. 

Then, returning to \eqref{eq6.21.3}, we see that 
when $\ell(I) \gtrsim \dist(I, \rn\times\{0\})$, we may replace $I$ by a Carleson box
$R_Q:= Q\times (0,\ell(Q))$ of comparable side length, and then apply Lemma~\ref{l6.34} to obtain a bound
in terms of \eqref{eq1.16}.
If, on the other hand, if $\ell(I) \ll \dist(I, \rn\times\{0\})$, we may then use
 \eqref{eq1.6} to reduce to the previous case. It is useful to note that $|I|=|Q|^\frac{n+1}{n}$ when calculating the normalization.

\vskip 0.08 in \noindent {\bf Step IV. Proof of \eqref{eq1.20}--\eqref{eq1.22.1}. } 
The proof is postponed until Section~\ref{s7}. See Proposition~\ref{p7.15}.

This finishes the proof of Theorem~\ref{t1.13}, modulo Lemma~\ref{l7.1}, Proposition~\ref{p7.15}, and the proof of Lemma~\ref{l6.0.1}.
\ep

\noindent {\it Proof of Lemma~\ref{l6.0.1}.}
We proceed in several steps. 
First, let us fix a small $\epsilon>0$ and define $Q_{1/\epsilon}$ to be a cube in $\rn$ centered at $0$ with the side-length $1/\epsilon$. Let $A_{\epsilon}$ be the elliptic matrix such that
\begin{equation*}A_{\epsilon}(x,t):=
\begin{cases}A_1(x,t),\,\,\,\text{if}\,\,\,\ (x,t)\in R_{1/\epsilon}=Q_{1/\epsilon} \times \bigl((-1/\epsilon, -\epsilon)\cup(\epsilon,1/\epsilon)\bigr),\\
A_0(x,t),\,\,\,\text{otherwise}.
\end{cases}
\end{equation*}

For convenience, we shall split the complement of $R_{1/\epsilon}$ as 
 $$\rec=\recs\cup \recb, \quad \recs=Q_{1/\epsilon}\times  [-\epsilon, \epsilon], \quad \recb=\ree\setminus\Bigl(Q_{1/\epsilon}\times(-1/\epsilon, 1/\epsilon)\Bigr).$$

The first order of business is to establish the desired estimate for $L_\epsilon$ in place of $L_1$. Note that $\|\N(\textbf{1}_{R_{1/\epsilon}}\nabla L_\epsilon^{-1} \Phi )\|_{L^p}$ is automatically finite by definition of $R_{1/\epsilon}$ and the general fact that $$\nabla L_\epsilon^{-1}:L^{\frac{2(n+1)}{n+3}}(\ree)\to L^2(\ree)$$ for all elliptic operators. The underlying estimates, of course, depend on $\epsilon$. However, knowing the fact of finiteness we can now run a familiar perturbation procedure to obtain  uniform in $\epsilon$ bounds. Specifically, by \eqref{eq6.27.1} applied to $L_\epsilon$ in place of $L_1$, the fact that $A_\epsilon-A_0$ is supported in $R_{1/\epsilon}$, and Proposition~\ref{p4.6}
$$\|\nabla L_\epsilon^{-1}\Phi -\nabla L_0^{-1}\Phi\|_{\tT^p_\infty}\leq \eps_0\|\nabla L_0^{-1}\nabla\|_{T^p_2\to\tT^p_\infty}\|\textbf{1}_{R_{1/\epsilon}}\nabla L_\epsilon^{-1}\Phi\|_{\tT^p_\infty}\lesssim \eps_0\|\textbf{1}_{R_{1/\epsilon}}\nabla L_\epsilon^{-1}\Phi\|_{\tT^p_\infty}.$$
Hence, $\|\nabla L_\epsilon^{-1}\Phi\|_{\tT^p_\infty}$ is finite and
$$\|\nabla L_\epsilon^{-1}\Phi -\nabla L_0^{-1}\Phi\|_{\tT^p_\infty}\lesssim \eps_0\|\nabla L_\epsilon^{-1}\Phi\|_{\tT^p_\infty}.$$
Hiding the small term, we have 
\begin{equation}\label{eq6.9}
\|\nabla L_\epsilon^{-1}\Phi\|_{\tT^p_\infty}\lesssim \|\nabla L_0^{-1}\Phi\|_{\tT^p_\infty}\leq C_\Phi,\quad n/(n+\alpha_0)<p<2+\eps,
\end{equation}
with constant $C_\Phi$ independent of $\epsilon$. 

Now let us estimate the difference between $L_\epsilon$ and $L_1$. For $\widetilde{y}, \widetilde{s}$, $\eta$, $R$ such that  $\epsilon \ll \eta<\widetilde{s}<R/2\ll 1/\epsilon$ and $|\widetilde{y}|<R\ll 1/\epsilon$, we have 
\begin{multline}
\fiint \limits_{W(\tilde{y},\tilde{s})} |\nabla L_1^{-1}\Phi(y,s)-\nabla\Le^{-1}\Phi(y,s)|^2 dyds
\\[4pt]
=\fiint_{W(\tilde{y},\tilde{s})} \left|\nabla_{y,s} \iint_{\ree} \nabla_{w,r} \Gamma_1 (y,s,w,r) (A^1-A^\epsilon)(w,r) \nabla_{w,r}\Le^{-1}\Phi(w,r) dwdr\right|^2 dyds\\[4pt]
=\fiint_{ W(\tilde{y},\tilde{s})} \left|\nabla_{y,s} \iint_{\ree}  \nabla_{w,r} \Gamma_1 (y,s,w,r) \textbf{1}_{\recs}(w,r)(A^1-A^\epsilon)(w,r) \nabla_{w,r}\Le^{-1}\Phi(w,r) dwdr\right|^2 dyds\\[4pt]
+ \fiint_{W(\tilde{y},\tilde{s})} \Bigl|\nabla_{y,s} \iint_{\ree}  \nabla_{w,r}( \Gamma_1 (y,s,w,r) -\Gamma_1 (\tilde y,\tilde s,w,r) )\times \\[4pt]
\times \textbf{1}_{\recb}(w,r)(A^1-A^\epsilon)(w,r) \nabla_{w,r} \Le^{-1}\Phi(w,r) dwdr\Bigr|^2 dyds\\[4pt]
\lesssim \frac{1}{\eta^{\,2}} \fiint_{\widetilde W(\tilde{y},\tilde{s})} \left| \iint_{\ree}  \nabla_{w,r} \Gamma_1 (y,s,w,r) \textbf{1}_{\recs}(w,r)(A^1-A^\epsilon)(w,r) \nabla_{w,r}\Le^{-1}\Phi(w,r) dwdr\right|^2 dyds\\[4pt]
+  \frac{1}{\eta^{\,2}} \fiint_{\widetilde W(\tilde{y},\tilde{s})} \Bigl| \iint_{\ree}  \nabla_{w,r}( \Gamma_1 (y,s,w,r) -\Gamma_1 (\tilde y,\tilde s,w,r) )\times \\[4pt]
\times \textbf{1}_{\recb}(w,r)(A^1-A^\epsilon)(w,r) \nabla_{w,r} \Le^{-1}\Phi(w,r) dwdr\Bigr|^2 dyds =: I+II,
\label{eq6.10}\end{multline}

\noindent where, as usually, $\widetilde W(\tilde{y},\tilde{s})$ is a slightly enlarged version of $W(\tilde{y},\tilde{s})$ (in the present context, one can take, for instance, a $c \eta$ - neighborhood of $W(\tilde{y},\tilde{s})$ for some small $c$). 

Let us start by estimating $II$. 

To this end, we show that $r  \nabla(\Gamma_1(y,s,w,r)-\Gamma_1(\tilde y,\tilde s,w,r))\textbf{1}_{\recb}(w,r) $ is in $ T^{p'}_{2}$ as a function of $w,r$  for any fixed $(y,s) \in \widetilde W(\tilde{y},\tilde{s})$ and $p'<2$ sufficiently close to 2. Slightly abusing the notation, we keep writing $r$ inside the norm. Then\begin{multline*}
\left\|r  \nabla(\Gamma_1(y,s,\cdot,\cdot)-\Gamma_1(\tilde y,\tilde s,\cdot,\cdot))\textbf{1}_{\recb}\right\|_{T^{p'}_2}\\[4pt]
\lesssim \sum_{j=1}^\infty \left\|r  \nabla(\Gamma_1(y,s,\cdot,\cdot)-\Gamma_1(\tilde y,\tilde s,\cdot,\cdot))\textbf{1}_{S_j(Q_{1/\epsilon}\times(-1/\epsilon,1/\epsilon))}\right\|_{T^{p'}_2},
\end{multline*}

\noindent where, as usually, $S_j(Q_{1/\epsilon}\times(-1/\epsilon,1/\epsilon))$, $j\geq 1$, are dyadic annuli around the set $Q_{1/\epsilon}\times(-1/\epsilon,1/\epsilon)$. Using the definition of the tent spaces via the square function we observe that the last expression above is equal to 
\begin{multline*}
 \sum_{j=1}^\infty \left\|\A_2(r  \nabla(\Gamma_1(y,s,\cdot,\cdot)-\Gamma_1(\tilde y,\tilde s,\cdot,\cdot))\textbf{1}_{S_j(Q_{1/\epsilon}\times(-1/\epsilon,1/\epsilon))})\right\|_{L^{p'}(C2^jQ_{1/\epsilon})}\\[4pt]
\lesssim  \sum_{j=1}^\infty (2^j/\epsilon)^{\frac{n}{p'}-\frac n2} \left\|\A_2(r  \nabla(\Gamma_1(y,s,\cdot,\cdot)-\Gamma_1(\tilde y,\tilde s,\cdot,\cdot))\textbf{1}_{S_j(Q_{1/\epsilon}\times(0,1/\epsilon))})\right\|_{L^{2}(C2^jQ_{1/\epsilon})}\\[4pt]
\lesssim  \sum_{j=1}^\infty (2^j/\epsilon)^{\frac{n}{p'}-\frac n2} \left(\iint_{ S_j(Q_{1/\epsilon}\times(0,1/\epsilon))}  |r\nabla(\Gamma_1(y,s,w,r)-\Gamma_1(\tilde y,\tilde s,w,r))|^2 \,\frac{dwdr}{r}\right)^{1/2}
\\[4pt]
\lesssim  \sum_{j=1}^\infty (2^j/\epsilon)^{\frac{n}{p'}-\frac n2-\frac{1}{2}} \left(\iint_{\widetilde S_j(Q_{1/\epsilon}\times(0,1/\epsilon))}  |\Gamma_1(y,s,w,r)-\Gamma_1(\tilde y,\tilde s,w,r)|^2 \,dwdr\right)^{1/2}.
\end{multline*}

\noindent Here, as usually, $\widetilde S_j$ is a slightly enlarged version of $S_j$.  Note that by choosing suitable constants we can make sure that  $\widetilde S_j$ has roughly the same separation from $\widetilde W(\tilde{y},\tilde{s})$ as $S_j$. It follows from the calculation above and \eqref{eq6.9} that 
\begin{multline*}
II^{1/2}\lesssim  \frac{C_\Phi}{\eta} \sum_{j=1}^\infty (2^j/\epsilon)^{\frac{n}{p'}-\frac n2-\frac 12} \left(\iint_{\widetilde S_j(Q_{1/\epsilon}\times(0,1/\epsilon))} \fiint \limits_{\widetilde W(\tilde{y},\tilde{s})}  |\Gamma_1(y,s,w,r)-\Gamma_1(\tilde y,\tilde s,w,r)|^2 \,dyds\,dwdr\right)^{1/2} 
\\[4pt]
\lesssim  \frac{C_\Phi}{\eta} \sum_{j=1}^\infty (2^j/\epsilon)^{\frac{n}{p'}-\frac n2-\frac 12} \left(\frac{R}{2^j/\epsilon}\right)^{\alpha_0}\left(\iint_{\widetilde S_j(Q_{1/\epsilon}\times(0,1/\epsilon))} \fiint \limits_{c2^j\widetilde W(\tilde{y},\tilde{s})}  |\Gamma_1(y,s,w,r)|^2 \,dyds\,dwdr\right)^{1/2},
\end{multline*}

\noindent where we used H\"older continuity of $\Gamma_1(y,s,w,r)$ as a solution in $(y,s)$ and $c2^j\widetilde W(\tilde{y},\tilde{s})$ is a concentric dialate of $\widetilde W(\tilde{y},\tilde{s})$ with $c$ chosen so that  $c2^j\widetilde W(\tilde{y},\tilde{s})$ and $\widetilde S_j(Q_{1/\epsilon}\times(-1/\epsilon,1/\epsilon))$ are separated by a strip of width $2^j/\epsilon$. Then the expression above is bounded by 
\begin{equation*}
\,\frac{C_\Phi}{\eta} \sum_{j=1}^\infty (2^j/\epsilon)^{\frac{n}{p'}-\frac n2-\frac 12} \left(\frac{R}{2^j/\epsilon}\right)^{\alpha_0}(2^j/\epsilon)^{-(n-1)+\frac{n+1}{2}} \lesssim C_\Phi\,\frac{R^{\alpha_0}}{\eta}\sum_{j=1}^\infty (2^j/\epsilon)^{\frac{n}{p'}-n+1-\alpha_0} \lesssim C_\Phi\,\frac{R^{\alpha_0}}{\eta}\epsilon^{-\frac{n}{p'}+n-1+\alpha_0}.
\end{equation*}

Turning to  $I$, we note that for $p'<2$
\begin{multline*}
\left\|r  \nabla(\Gamma_1(y,s,\cdot,\cdot))\textbf{1}_{\recs}\right\|_{T^{p'}_2}
\lesssim |Q_{1/\epsilon}|^{\frac 1{p'}-\frac 12}\left\|\A_2(r  \nabla(\Gamma_1(y,s,\cdot,\cdot))\textbf{1}_{\recs})\right\|_{L^2(2Q_{1/\epsilon})}\\[4pt]
\lesssim \epsilon^{\frac n2-\frac{n}{p'}}\left(\int_{Q_{1/\epsilon}}\int_0^{\epsilon}|r \nabla\Gamma_1(y,s,w,r)|^2\,\frac{drdw}{r}\right)^{1/2}\lesssim \epsilon^{\frac 12 +\frac n2-\frac{n}{p'}}\left(\int_{Q_{1/\epsilon}}\int_0^{\epsilon}| \nabla\Gamma_1(y,s,w,r)|^2\,drdw\right)^{1/2}\\[4pt]
\lesssim \epsilon^{\frac 12 +\frac n2-\frac{n}{p'}}\left(\iint_{\ree\setminus 2\widetilde W(\tilde y,\tilde s)}\frac{|\Gamma_1(y,s,w,r)|^2}{|(y,s)-(w,r)|^2}\,{drdw}\right)^{1/2}  \lesssim  \epsilon^{\frac 12 +\frac n2-\frac{n}{p'}} (\eta)^{-\frac n2+\frac 12}.
\end{multline*}

Combining the estimates above with \eqref{eq6.9}, we conclude that for every $p>2$ close to $2$ and $x\in Q_{R/2}$
\begin{equation}\label{eq6.11}
\N^{\eta,R/2} (\nabla L_1^{-1}\Phi-\nabla\Le^{-1}\Phi)(x) \lesssim C_\Phi \left(\frac{R^{\alpha_0}}{\eta}\epsilon^{-\frac{n}{p'}+n-1+\alpha_0}+\epsilon^{\frac 12 +\frac n2-\frac{n}{p'}} (\eta)^{-\frac n2-\frac 12}\right),\end{equation}
which in turn shows that
\begin{equation}\label{eq6.12}
\N^{\eta,R/2} (\nabla L_1^{-1}\Phi)(x) \lesssim C_\Phi \left(\frac{R^{\alpha_0}}{\eta}\epsilon^{-\frac{n}{p'}+n-1+\alpha_0}+\epsilon^{\frac 12 +\frac n2-\frac{n}{p'}} (\eta)^{-\frac n2-\frac 12}\right)+ \N^{\eta,R/2} (\nabla\Le^{-1}\Phi)(x),
\end{equation}
where 
$$\N^{\eta,R/2}(u)(x):=\left(\sup_{\substack{(\tilde{y}, \tilde{s}) \in \Gamma(x) \cap B(x,R/2)\\ \eta< \tilde{s}}}\fiint\limits_{W(\tilde{y},\tilde{s})}|u(y,s)|^2 dyds \right)^{1/2}, $$
\noindent with a suitably chosen aperture of the underlying cones and $\epsilon \ll \eta<R \ll 1/\epsilon$.

Taking the norm of both sides of \eqref{eq6.12} in $L^p(Q_{R/2})$, with $p>2$ close to 2, one has
\begin{multline}\label{eq6.13}
\left(\int_{Q_{R/2}}\left|\N^{\eta,R/2} (\nabla L_1^{-1}\Phi)(x)\right|^p\,dx\right)^{1/p} \\[4pt] \lesssim C_\Phi \left(\frac{R^{\alpha_0+n/p}}{\eta}\epsilon^{-\frac{n}{p'}+n-1+\alpha_0}+ R^{n/p}\epsilon^{\frac 12 +\frac n2-\frac{n}{p'}} (\eta)^{-\frac n2-\frac 12}\right)  + \left(\int_{Q_{R/2}}\left|\N^{\eta,R/2} (\nabla {L_\epsilon}^{-1}\Phi)(x)\right|^p\,dx\right)^{1/p}\\[4pt] \lesssim C_\Phi \left(\frac{R^{\alpha_0+n/p}}{\eta}\epsilon^{-\frac{n}{p'}+n-1+\alpha_0}+ R^{n/p}\epsilon^{\frac 12 +\frac n2-\frac{n}{p'}} (\eta)^{-\frac n2-\frac 12}+1\right),
\end{multline}

\noindent by \eqref{eq6.9}. Taking the limit first as $\epsilon\to 0$ and then as $\eta\to 0$, $R\to\infty$, we deduce that 
\begin{equation}\label{eq6.14}
\|\N(\nabla L_1^{-1}\Phi )\|_{L^p} \lesssim C_\Phi,
\end{equation}

\noindent for $p>2$ close to 2, as desired.

Let us now treat the case of $p\leq 1$. We shall prove that $r \nabla\Gamma_1(y,s,\cdot,\cdot)\textbf{1}_{\reu \setminus {R_{1/\epsilon}}} \in T^\infty_{2,\alpha}$ whenever $(y,s) \in W(\tilde{y},\tilde{s})$. 
First of all, 
\begin{multline*}
\sup\limits_{Q:\,l(Q)\lesssim 1/\epsilon} \frac{1}{|Q|^{1+\frac{2\alpha}{n}}}\iint\limits_{R_Q\cap\recb} |\nabla \Gamma_1(y,s,w,r)|^2 rdwdr\\[4pt] 
\leq 
\sup\limits_{Q:\,l(Q)\lesssim 1/\epsilon} \frac{l(Q)}{|Q|^{1+\frac{2\alpha}{n}}}\iint\limits_{R_Q\cap\recb}|\nabla (\Gamma_1(y,s,w,r)-\Gamma_1(y,s,c_Q,0))|^2 dwdr,
\end{multline*}
\noindent where $c_Q$ denotes the center of the cube $Q$. Note that we consider at the moment $Q:\,l(Q)\lesssim 1/\epsilon$ (in fact, $Q:\,l(Q)\leq c/\epsilon$ for suitably small $c$). Then for every $(y,s)\in W(\tilde{y},\tilde{s})$ subject to the conditions $\epsilon \ll \eta<\widetilde{s}<R/2\ll 1/\epsilon$ and $|\widetilde{y}|<R\ll 1/\epsilon$, the expression $\Gamma_1(y,s,w,r)-\Gamma_1(y,s,c_Q,0)$ as a function of $w,r$ is a solution in a $c/\epsilon$-neighborhood of $2R_Q$, denoted by $U_{c/\epsilon}(2R_Q)$. Here, by $2R_Q$ we denote a subregion of $\ree$ concentric with $R(Q)$ and twice as large.  
Thus, first applying Caccioppoli inequality in $2R_Q$ and then exploring H\"older continuity of solutions in $U_{c/\epsilon}(2R_Q)$, we deduce that the last expression above is bounded by 
\begin{multline*}
\sup\limits_{Q:\,l(Q)\lesssim 1/\epsilon} \frac{l(Q)^{-1}}{|Q|^{1+\frac{2\alpha}{n}}}\iint\limits_{2R_Q}|\Gamma_1(y,s,w,r)-\Gamma_1(y,s,c_Q,0)|^2 dwdr\\[4pt]
\lesssim \sup\limits_{Q:\,l(Q)\lesssim 1/\epsilon} \frac{1}{l(Q)^{2\alpha}}\,\left(\frac{l(Q)}{\dist \{(y,s), 2R_Q\}}\right)^{2\alpha_0}\dist \{(y,s), 2R_Q\}^{-n-1}\iint\limits_{U_{c\dist \{(y,s), 2R_Q\}}(2R_Q)}| \Gamma_1(y,s,w,r)|^2 dwdr.
\end{multline*}
\noindent Thus, using pointwise estimates on solutions, the expression above can further be estimated by 
\begin{multline*}
 \sup\limits_{Q:\,l(Q)\lesssim 1/\epsilon} \frac{1}{l(Q)^{2\alpha}}\,\left(\frac{l(Q)}{\dist \{(y,s), 2R_Q\}}\right)^{2\alpha_0}\dist \{(y,s), 2R_Q\}^{-2(n-1)}\\[4pt] \lesssim (1/\epsilon)^{2\alpha_0-2\alpha} \dist \{(y,s), 2R_Q\}^{-2(n-1)-2\alpha_0}
 \lesssim \epsilon^{2\alpha+2(n-1)}.
\end{multline*}

Now consider $Q$ such that $l(Q)\geq c/\epsilon$. Similarly to above, $\Gamma_1(y,s,w,r)$, as a function of $(w,r)$, is a solution in $U_{c/\epsilon}':=U_{c/\epsilon}(R_Q\cap\recb)$. Thus, we can use Caccioppoli inequality in a decomposition of $R_Q\cap\recb$ into cubes of size $c/\epsilon$. Then 
\begin{multline*}
\sup\limits_{Q:\,l(Q)\gtrsim 1/\epsilon} \frac{1}{|Q|^{1+\frac{2\alpha}{n}}}\iint\limits_{R_Q\cap\recb} |\nabla \Gamma_1(y,s,w,r)|^2 rdwdr\\[4pt] 
\lesssim
\sup\limits_{Q:\,l(Q)\gtrsim 1/\epsilon} \frac{l(Q)}{|Q|^{1+\frac{2\alpha}{n}}}\,\epsilon^2\iint\limits_{U_{c/\epsilon}'} | \Gamma_1(y,s,w,r)|^2 dwdr \\[4pt]
\lesssim
\sup\limits_{Q:\,l(Q)\gtrsim 1/\epsilon} \frac{l(Q)}{|Q|^{1+\frac{2\alpha}{n}}}\,\epsilon^2\iint\limits_{1/\epsilon\lesssim |(y,s)-(w,r)|\lesssim l(Q)} |(y,s)-(w,r)|^{-2(n-1)} dwdr \\[4pt]
\lesssim 
\sup\limits_{Q:\,l(Q)\gtrsim 1/\epsilon} \frac{l(Q)}{|Q|^{1+\frac{2\alpha}{n}}}\,\epsilon^2 (1/\epsilon)^{-2(n-1)-\beta+n} l(Q)^{\beta+1},
\end{multline*}

\noindent for any $\beta>0$, since $-2(n-1)-\beta+n<0$. Taking now $\beta>0$ sufficiently small, we observe that the latter expression is bounded by 
\begin{equation*}
\sup\limits_{Q:\,l(Q)\gtrsim 1/\epsilon} \frac{l(Q)^{2+\beta-n-2\alpha}}{(1/\epsilon)^{n+\beta}} \lesssim \epsilon^{2n-2+2\alpha}. 
\end{equation*}

It remains to estimate 
\begin{equation*} \sup\limits_{Q} \frac{1}{|Q|^{1+\frac{2\alpha}{n}}}\iint\limits_{R_Q\cap\recs} |\nabla \Gamma_1(y,s,w,r)|^2 rdwdr
\lesssim \sup\limits_{Q} \frac{\min\{\epsilon, l(Q)\}}{|Q|^{1+\frac{2\alpha}{n}}}\iint\limits_{R_Q\cap\recs} |\nabla \Gamma_1(y,s,w,r)|^2 dwdr.
\end{equation*}

\noindent Consider first the case when $l(Q)<c\eta$ for some small $c>0$. Then the corresponding part of the supremum above is bounded by
\begin{multline*} \sup\limits_{Q:\,l(Q)<c\eta} \frac{\min\{\epsilon, l(Q)\}}{|Q|^{1+\frac{2\alpha}{n}}}\iint\limits_{R_Q} |\nabla (\Gamma_1(y,s,w,r)-\Gamma_1(y,s,c_Q,0))|^2 dwdr
\\[4pt] \lesssim
 \sup\limits_{Q:\,l(Q)<c\eta} \frac{\min\{\epsilon, l(Q)\}}{|Q|^{1+\frac{2\alpha}{n}}}\frac {1}{l(Q)^2} \iint\limits_{\widetilde R_Q} |\Gamma_1(y,s,w,r)-\Gamma_1(y,s,c_Q,0)|^2 dwdr
 \\[4pt] 
\lesssim \sup\limits_{Q:\,l(Q)<c\eta} \frac{\min\{\epsilon, l(Q)\}}{|Q|^{1+\frac{2\alpha}{n}}}\frac {|Q|^{1+\frac 1n}}{l(Q)^2}\,\left(\frac{l(Q)}{\eta}\right)^{2\alpha_0}\fiint\limits_{U_{c\eta}(R_Q)} |\Gamma_1(y,s,w,r)|^2 dwdr \\[4pt]
\lesssim \sup\limits_{Q:\,l(Q)<c\eta} \frac{\min\{\epsilon, l(Q)\}}{l(Q)^{2\alpha+1}}\,\left(\frac{l(Q)}{\eta}\right)^{2\alpha_0} (\eta)^{-2(n-1)}\lesssim \epsilon^{2\alpha_0-2\alpha} \,(\eta)^{-2(n-1)-2\alpha_0}.
\end{multline*}

It remains to analyze the case when $l(Q)>c\eta$. We have 
\begin{multline*}  \sup\limits_{Q:\,l(Q)>c\eta} \frac{\min\{\epsilon, l(Q)\}}{|Q|^{1+\frac{2\alpha}{n}}}\iint\limits_{R_Q\cap\recs} |\nabla \Gamma_1(y,s,w,r)|^2 dwdr\\[4pt]
\lesssim  \sup\limits_{Q:\,l(Q)>c\eta} \frac{\min\{\epsilon, l(Q)\}}{|Q|^{1+\frac{2\alpha}{n}}}\,\frac{1}{\eta^{\,2}}\,\iint\limits_{U_{c\eta}(R_Q\cap\recs)} | \Gamma_1(y,s,w,r)|^2 dwdr
\\[4pt]
\lesssim  \sup\limits_{Q:\,l(Q)>c\eta} \frac{\min\{\epsilon, l(Q)\}}{|Q|^{1+\frac{2\alpha}{n}}}\,\frac{1}{\eta^{\,2}}\, \eta^{\,-2(n-1)}\, \eta\, l(Q)^n \lesssim  \sup\limits_{Q:\,l(Q)>c\eta} \frac{\min\{\epsilon, l(Q)\}}{l(Q)^{2\alpha}}\,\frac{1}{\eta^{\,2n-1}}\lesssim \epsilon\, \eta^{\,-2\alpha-2n+1}.
\end{multline*}

Therefore, by \eqref{eq6.10}, for every $n/(n+\alpha_0)<p\leq 1$ and $x\in Q_{R/2}$
 
\begin{equation}\label{eq6.15}
\N^{\eta.R} (\nabla L_1^{-1}\Phi-\nabla \Le^{-1}\Phi)(x) \lesssim C_\Phi \left(\epsilon^{n-1+\alpha}+\epsilon^{\alpha_0-\alpha} \,(\eta)^{-n+1-\alpha_0}+ \epsilon^{1/2} \eta^{\,-\alpha-n+1/2}\right),\end{equation}
where $\eta$ and $R$ such that  $\epsilon \ll \eta<R/2\ll 1/\epsilon$.

Having this at hand, the argument analogous to \eqref{eq6.11}-- \eqref{eq6.14} shows that 
\begin{equation}\label{eq6.16}
\| \N(\nabla L_1^{-1}\Phi )\|_{L^p} \lesssim C_\Phi,
\end{equation}

\noindent for $n/(n+\alpha_0)<p\leq 1$. By interpolation with \eqref{eq6.14}, we deduce \eqref{eq6.16} for all $n/(n+\alpha_0)<p\leq 2+\eps$. This finishes the proof of Lemma~\ref{l6.0.1}.\ep

\section{Existence and Invertibility of Layer Potentials on the boundary}\label{s7}

\subsection{Normal and tangential traces on the boundary}\label{s2.4}

Given a divergence form operator $L=-\dv (A \nabla)$ defined in
$\mathbb{R}^{n+1}_+$, we shall say that a solution $u$ of the equation $Lu=0$ has a variational co-normal
in the sense of tempered distributions if there is an $f\in {\bf S}'(\rn)$
such that for every $\Phi \in {\bf S}(\mathbb{R}^{n+1})$, we have
\begin{equation}\label{eq7.0}\iint_{\mathbb{R}^{n+1}_+}A\nabla u \cdot \overline{\nabla \Phi} =\langle f,\varphi \rangle\,,
\end{equation}
where $\varphi:=\Phi(\cdot,0)$ (observe that $\varphi\in {\bf S}(\rn)$.)  We then set $\partial_{\nu_A}u:= f$. 

To formulate the following convergence results, recall that we use the usual conventions $H^p= L^p$ and $H^{1,p}=\dot L_1^p$ when $p>1$.

\begin{lemma}\label{l7.1} Let $L=-\dv(A\nabla)$ be an elliptic operator (not necessarily $t$-independent), and 
suppose that $u\in W^{1,2}_{loc}(\mathbb{R}^{n+1}_+)$ is a weak solution of $Lu=0$, which satisfies
$\N(\nabla u) \in L^p(\rn)$ with $n/(n+1)<p <2+\eps$.  Then
\begin{itemize}
\item[(i)]  there exists $g \in H^{1,p}(\rn)$ such that $u \to g$ n.t. a.e., with 
\begin{equation}\label{eq7.1.1}
|u(y,t)-g(x)| \lesssim t \N(\nabla u)(x), \quad {\mbox{for every $(y,t) \in \Gamma(x)$, $x\in \rn$,}}
\end{equation} 
and 
\begin{equation}\label{eq7.1.2}
\|g\|_{H^{1,p}}\lesssim \|\N(\nabla u)\|_{L^p}, \quad n/(n+1)<p <2+\eps;
\end{equation}
\item[(ii)] for the limiting function $g$ from {\rm{(i)}}, one has 
\begin{equation}\label{eq7.1.3}\sup_{t>0} \left\|\fint_{t/2}^{2t}\nabla_{\|}u(\cdot,\tau) \,d\tau\right\|_{L^p(\rn)}\lesssim \|\N(\nabla u)\|_{L^p(\rn)},
\end{equation}
and 
\begin{equation}\label{eq7.1.4}\fint_{t/2}^{2t}\nabla_{\|}u(\cdot,\tau) \,d\tau \to \nabla_{\|}g\quad \mbox{as $t \to 0$,}
\end{equation}
weakly in $L^p$ when $p>1$. 
\end{itemize}
Furthermore, 
\begin{itemize}
\item[(iii)]  there exists $f \in H^p(\rn)$ such that $f=\partial_{\nu_A}u$ in the variational sense, i.e., $$\iint\limits_{\reu} A(X) \nabla u(X) \overline{\nabla \Phi(X)}\,dX= \int\limits_{\rn} f(x) \overline{\varphi(x)}\,dx,$$ for $\Phi \in {\bf S}(\mathbb{R}^{n+1})$ and $\varphi:= \Phi\!\mid_{t=0}$, and 
\begin{equation}\label{eq7.2}
\|\partial_{\nu_A} u\|_{H^p(\rn)}\leq C\, \|\N(\nabla u)\|_{L^p(\rn)}, \quad n/(n+1)<p <2+\eps;
\end{equation}
\item[(iv)] for the limiting function $f$ from {\rm{(iii)}}, one has $\partial_{\nu_A} u(\cdot,t) \to f$,  as $t \to 0$, in the sense of tempered distributions. \end{itemize}
\end{lemma}

We remark that the proof of  (iii), (iv) below is presented primarily to treat the case $p\leq 1$.  For the case $p> 1$ one can directly adopt a construction in \cite{AAAHK} (or make obvious modifications in the argument below). 

A statement analogous to Lemma~\ref{l7.1} holds in the lower half space. 

\begin{proof}

\noindent {\bf Step I}. The proof of (i). 

The existence of a limiting function $g$ satisfying $|u(y,t)-g(x)| \lesssim t \N(\nabla u)(x)$, for every $(y,t) \in \Gamma(x)$, $x\in \rn$, follows verbatim an analogous argument in  \cite[pp. 461--462]{KP}. Moreover, in \cite{KP} the authors show that for any $x\in\rn$,
$$|g(x)-g(y)|\leq C t \N(\nabla u)(x), \quad\mbox{for every}\quad y\in \Delta(x,t).$$
Just as in \cite{KP}, this yields $g\in \dot L^p_1(\rn)$ whenever $\N(\nabla u)\in L^p$, $p>1$, with the desired estimate on $\|g\|_{\dot L^p_1}$.
Moreover, the inequality above immediately implies that for any $x,y\in\rn$,
$$|g(x)-g(y)|\leq C |x-y|\,(\N(\nabla u)(x)+\N(\nabla u)(y)).$$
Hence, by Lemma~\ref{l2.8}, $g \in H^{1,p}(\rn)$, when $\N(\nabla u)\in L^p$, $\left(\frac{n}{n+1},1\right]$, with the desired estimate on $\|g\|_{H^{1,p}}$.

\vskip 0.08 in

\noindent \noindent {\bf Step II}. The proof of (ii).  Let $p>1$.

The fact that \eqref{eq7.1.3} is valid follows directly from the definition of the non-tangential maximal function. Turning to \eqref{eq7.1.4}, let us denote 
$$u_{{\rm{Ave}}}(x,t):=\fint_{t/2}^{2t} u(\cdot,\tau) \,d\tau,$$
and take $\vec{\psi} \in \SL(\rn, \mathbb{C}^n)$. Then (with the integrals interpreted as ${\bf S}', {\bf S}$ pairings)
\begin{multline}\label{eq7.3}\left| \int_{\rn} \overline{(\nabla_{\|}u_{{\rm{Ave}}}(x,t)-\nabla_{\|}g(x))} \vec{\psi}(x)\,dx\right|=\left| \int_{\rn}  \overline{(u_{{\rm{Ave}}}(x,t)-g(x))} \dv_{\|}\vec{\psi}(x)\,dx\right|\\[4pt]
\lesssim t \left( \int_{\rn} |\N(\nabla u)(x)|^{p} dx\right)^{1/p}\left( \int_{\rn} |\dv_{\|}\vec{\psi}(x)|^{p'}dx\right)^{1/p'}.\end{multline}
This directly implies 
\begin{equation}\label{eq7.4}\int_{\rn}  \overline{(\nabla_{\|}u_{{\rm{Ave}}}(x,t)-\nabla_{\|}g(x))}\vec{\psi} (x)\,dx \xrightarrow{t \to 0} 0,\quad \mbox {for any}\quad \vec{\psi} \in \SL(\rn, \mathbb{C}^n),
\end{equation}
which justifies convergence weakly in $L^p.$

\vskip 0.08 in

\noindent \noindent {\bf Step III}. The proof of (iii).   

By hypothesis and Lemma~\ref{l2.31}, $\nabla u \in L^r(\mathbb{R}^{n+1}_+)$, with $r:= p(n+1)/n >1.$
We may then define a linear functional $\Lambda=\Lambda_u \in {\bf S}'(\mathbb{R}^{n+1})$ by
\begin{equation}\label{eq7.5}\langle\Lambda,\Phi\rangle:= \iint_{\!\mathbb{R}^{n+1}_+} A\nabla u\overline{\nabla\Phi}\,,\qquad \Phi \in {\bf S}(\mathbb{R}^{n+1})\,.\end{equation}
Given $\varphi\in{\bf S}(\rn)$, we say that $\Phi\in {\bf S}(\mathbb{R}^{n+1})$ is an extension of $\varphi$ if $\Phi(\cdot,0)=\varphi$.
We now define a linear functional $f \in {\bf S}'(\rn)$ by
$$\langle f,\varphi\rangle:= \langle\Lambda,\Phi\rangle\,,$$
where $\Phi$ is any extension of $\varphi$.  Since this extension need not be unique, we must verify that
$f$ is well defined.  To this end, fix $\varphi\in{\bf S}(\rn)$,
and let $\Phi_1,\Phi_2\in{\bf S}(\mathbb{R}^{n+1})$ be any two extensions
of $\varphi$.  Then $\Psi:= \Phi_1-\Phi_2\in {\bf S}(\mathbb{R}^{n+1}),$ with $\Psi(\cdot,0)\equiv 0$. Then
$ \langle\Lambda,\Psi\rangle = 0$, by definition of a weak solution.  Thus, the linear functional
$f$ is well defined, and therefore $u$ has a variational co-normal in ${\bf S}'$.

Setting $f=:\partial_{\nu_A} u$, we now proceed to prove \eqref{eq7.2}.  We fix $\varphi \in \C_0^\infty(\rn)$, with $0\leq\varphi$, $\int\varphi = 1$, and ${\rm supp}(\varphi)\subset \Delta(0,1):=\{x\in\rn:|x|<1\}$.
Denote
\begin{equation}M_\varphi f:=\sup\limits_{t>0} |(f\ast \varphi_t)|,\label{eq7.6}\end{equation}
where $\varphi_t(x):=t^{-n}\varphi(x/t)$. Then
$$\|f\|_{H^p(\rn)}\leq C\, \|M_\varphi f\|_{L^p(\rn)},$$
as usually, with $H^p\equiv L^p$ when $p>1$.
Hence, it suffices to show that 
\begin{equation}\label{eq7.7}
\|M_\varphi(\partial_{\nu_A} u)\|_{L^p(\rn)}  \leq C \|\N(\nabla u)\|_{L^p(\rn)} \,,
\end{equation}
\noindent for $\partial_{\nu_A} u$ defined as above. 
We claim that
\begin{equation}\label{eq7.8}
M_\varphi(\partial_{\nu_A} u)(x)\leq C \left(M\left(\N(\nabla u)\right)^{n/(n+1)}\right)^{(n+1)/n}(x)\,,
\end{equation}
for every $x\in\rn$, where $M$ denotes the usual Hardy-Littlewood maximal operator.
Taking the claim for granted momentarily, we see that
\begin{equation}\label{eq7.8.1}
\int_{\rn}M_\varphi(\partial_{\nu_A} u)^p\,dx \lesssim\int_{\rn}\left(M\left(\N(\nabla u)\right)^{n/(n+1)}\right)^{p(n+1)/n}\,dx
\lesssim \int_{\rn}\N(\nabla u)^{p}\,dx\,,
\end{equation}
as desired, since $p(n+1)/n>1$.

It therefore remains only to establish \eqref{eq7.8}.  To this end, we fix $x\in \rn$ and $t>0$, 
set $B:= B(x,t):=\{Y\in \mathbb{R}^{n+1}: |Y-x|<t\}$, and define a smooth cut-off
$\eta_{B}\in  C_0^\infty(2B)$, with
$\eta_{B}\equiv 1$ on $B$, $0\leq\eta_{B}\leq 1,$
and $|\nabla \eta_{B}|\lesssim 1/t$.   Then
$$\Phi_{x,t}(y,s):= \eta_B(y,s)\, \varphi_t(x-y)$$
is an extension of $\varphi_t(x-\cdot)$, with $\Phi_{x,t}\in C_0^\infty(2B),$ which satisfies
$$0\leq \Phi_{x,t} \lesssim t^{-n}\,,\qquad|\nabla_Y \Phi_{x,t}(Y)|\lesssim t^{-n-1}\,.$$
We then have
\begin{multline}\label{eq7.9}
|\left(\varphi_t*\partial_{\nu_A} u\right)(x)|=|\langle \partial_{\nu_A} u,\varphi_t(x-\cdot)\rangle| =\left|
 \iint_{\!\mathbb{R}^{n+1}_+} A\nabla u\overline{\nabla\Phi_{x,t}}\,dY\right|\\[4pt]
 \lesssim \,t^{-n-1}\iint_{\!\mathbb{R}^{n+1}_+\cap 2B} |\nabla u|\,dY\,
 \lesssim \,t^{-n-1}\left(\int_{\Delta(x,Ct)}\left(\N(|\nabla u| 1_{2B})(y)\right)^{n/(n+1)}dy\right)^{(n+1)/n}\,,
\end{multline}
where in the last step we have used \eqref{eq2.32} with $p = n/(n+1)$ and $C$ is chosen sufficiently large, so that we have
\begin{equation}\label{eq7.9.1}\N(|\nabla u| 1_{2B}) \leq 1_{\Delta(x,Ct)}  \N(|\nabla u| 1_{2B})\,,
\end{equation}
with $\Delta(x,Ct):=\{y\in\rn:|x-y|<Ct\}.$  Hence, 
\begin{equation}\label{eq7.9.2}|\left(\varphi_t*\partial_{\nu_A} u\right)(x)|\lesssim 
\left(t^{-n}\int_{|x-y|<Ct}\left(\N(\nabla u)(y)\right)^{n/(n+1)}dy\right)^{(n+1)/n}\,,
\end{equation}
and taking the supremum over $t>0$, we obtain \eqref{eq7.8}.

\vskip 0.08 in

\noindent \noindent {\bf Step IV}. The proof of (iv).   

By the same method as above (see formula \eqref{eq7.5} and related discussion), the variational normal derivative of $u$, $\partial_{\nu_A} u(\cdot, t)$, is well-defined for any $t\geq 0$ in the sense of 
\begin{equation}\label{eq7.5.1}\langle\partial_{\nu_A} u(\cdot, t),\phi\rangle:= \int_{\rn}\int_t^\infty  A\nabla u\overline{\nabla\Phi}\,,\qquad \Phi \in {\bf S}(\mathbb{R}^{n+1})\,,\end{equation}
where $\Phi\in {\bf S}(\mathbb{R}^{n+1})$ is any extension of $\varphi=\Phi(\cdot,t)\in {\bf S}(\rn)$.
The result does not depend on a particular choice of extension.
Thus, it is enough to prove that for any $\varphi \in {\bf S}(\rn)$ and $\Phi \in {\bf S}(\mathbb{R}^{n+1})$ such that $\varphi=\Phi(\cdot,0)$
$$\iint_{\reu} A(x,t+s) \nabla u(x,t+s) \overline{\nabla \Phi(x,s)} dx  \xrightarrow{t \to 0} \iint_{\reu} A(x,s) \nabla u(x,s) \overline{\nabla \Phi(x,s)} dxds. $$
However,
\begin{multline}
\iint_{\reu} A(x,t+s) \nabla u(x,t+s) \overline{\nabla \Phi(x,s)}\, dxds\\[4pt]
=\int_{\rn}\int_t^\infty A(x,s) \nabla u(x,s) \overline{\nabla \Phi(x,s)}\, dxds\\[4pt]
+
\int_{\rn}\int_t^\infty A(x,s) \nabla u(x,s) \overline{\nabla (\Phi(x,s-t)-\Phi(x,s))} \,dxds=: I_t+II_t.
\end{multline}
By Lemma~\ref{l2.31} we have $\nabla u \in L^{p\frac{(n+1)}{n}}(\reu)$ and hence, by dominated convergence, $II_t$ converges to 0 and $I_t$ converges to 
$$\iint_{\reu} A(x,s) \nabla u(x,s) \overline{\nabla \Phi(x,s)} dxds,$$
as desired.
\end{proof}

\begin{remark}\label{r7.15.1} A careful look at the proof reveals that the property $Lu=0$ has not been used in Steps I, II of the argument. Hence, the statements (i)--(ii) apply to any $u\in W^{1,2}_{loc}(\reu)$ with $\N (\nabla u)\in L^p$, $n/(n+1)<p<2+\eps.$
\end{remark}

\subsection{Layer potentials at the boundary}\label{s7.2}

\begin{proposition}\label{p7.15} Retain the assumptions of Theorem~\ref{t1.13}. Let $p_0<p<2+\eps$ and $0\leq \alpha<\alpha_0$. There exist bounded operators $\widetilde{\K}^{L_j}:H^p \to H^p$, $\K^{L_j}:\la \to \la$ (resp., $\K^{L_j}:L^{p'} \to L^{p'}$ when $p>1$),  and $\SL^{L_j}_t\!\mid_{t=0}:H^p \to H^{1,p}$,  $j=0,1$, such that for every $f \in H^p$
\begin{itemize}
\item[(i)] $\partial^\pm_{\nu_{A_j}} \SL^{L_j} f = (\pm\frac{1}{2} I+\widetilde{\K}^{L_j})f$ for every $f\in H^p$ and 
\begin{equation}\label{eq5.15.0.0}\partial^\pm_{\nu_{A_j,t}}\SL^{L_j}f \xrightarrow{t \to 0^\pm}(\pm \frac{1}{2} I+\widetilde{\K}^{L_j})f,\end{equation}
where $\partial^\pm_{\nu^j}$ and $\partial^\pm_{\nu_{A_j,t}}$  denote the co-normal derivatives on $\RR^n\times \{s=0\}$ and $\RR^n \times \{s=t\}$, respectively, both interpreted in the variational sense.
The convergence is in the sense of tempered distributions.

\item[(ii)] \begin{equation}\label{eq5.15.0.1}\SL^{L_j}_t f \xrightarrow{t \to 0^\pm} \SL^{L_j}_t\!\mid_{t=0} f \quad\mbox{n.t. a.e. }\end{equation} In addition,  
\begin{equation}\label{eq5.15.0.2}\frac 1t\int_{t/2}^{2t}\nabla_{\|} \SL^{L_j}_\tau f\,d\tau \xrightarrow{t \to 0^{\pm}} \nabla_{\|}\SL^{L_j}_t\!\mid_{t=0} f\end{equation}
in the sense of tempered distributions. 
When $p>1$, we have 
\begin{equation}\label{eq5.15.0.4}\sup_{t\neq 0}\left\|\frac 1t\int_{t/2}^{2t}\nabla_{\|} \SL^{L_j}_\tau f\,d\tau\right\|_{L^p(\rn)}\lesssim \|f\|_{L^p(\rn)},\end{equation}
and the convergence in \eqref{eq5.15.0.2} holds weakly in $L^p$. 
\item[(iii)] \begin{equation}\label{eq5.15.0.3}\D^{L^*_j}_t g \xrightarrow{t \to 0^\pm} (\mp\frac{1}{2}I+ \K^{L^*_j})g\end{equation} for every $g\in C_0^\infty(\rn)$, in the sense of tempered distributions. When $p>1$ we have
\begin{equation}\label{eq5.15.0.5}\sup_{t\neq 0}\|\D^{L^*_j}_t g\|_{L^{p'}(\rn)}\lesssim \|f\|_{L^{p'}(\rn)},\end{equation}
and the convergence in \eqref{eq5.15.0.3} is in $L^{p'}$ on compacta in $\rn$. When $0\leq \alpha<\alpha_0$, 
\begin{equation}\label{eq5.15.0.6}\sup_{t\neq 0}\|\D^{L^*_j}_t g\|_{\la(\rn)}\lesssim \|f\|_{\la(\rn)},\end{equation}
and the convergence in \eqref{eq5.15.0.3} holds in the weak* topology of  $\la$.
\end{itemize}
Moreover, if $\eps_0$ is sufficiently small, then invertibility of  $\mp\frac 12 I + \K^{L_j^*}$, $\pm\frac 12 I+\widetilde{\K}^{L_j}$, or $\SL^{L_j}_t\!\mid_{t=0}$ in a given function space for $j=0$ implies their invertibility, in the same function space, for $j=1$. The precise statement is as follows. 

\begin{itemize}
\item[(iv)]  If there exists $p_*\in (p_0, 2+\eps)$ such that $\SL^{L_0}_t\!\mid_{t=0}: H^{p_*}(\rn)\to H^{1,p_*}(\rn)$ is invertible and $\eps_0$ is sufficiently small, depending on the standard constants and on the norm of the inverse of $\SL^{L_0}_t\!\mid_{t=0}$, then $\SL^{L_1}_t\!\mid_{t=0}: H^{p_*}(\rn)\to H^{1,p_*}(\rn)$ is invertible as well. 

\item[(v)] If there exists $p_*\in (1, 2+\eps)$ such that $\pm\frac 12 I+\widetilde{\K}^{L_0}: L^{p_*}\to L^{p_*}$ is invertible (equivalently, $\mp\frac 12 I + \K^{L_0^*}:L^{p'_*}\to L^{p'_*}$, is invertible) and $\eps_0$ is sufficiently small, depending on the standard constants and on the norms of the inverses, then $\pm\frac 12 I+\widetilde{\K}^{L_1}:L^{p_*}\to L^{p_*}$ and $\mp\frac 12 I + \K^{L_1^*}: L^{p'_*}\to L^{p'_*}$ are invertible as well. 

\item[(vi)] If there exists $p_*\in (p_0, 1]$ such that $\pm\frac 12 I+\widetilde{\K}^{L_0}:H^{p_*}\to H^{p_*}$ is invertible (and hence, $\mp\frac 12 I + \K^{L_0^*}:\Lambda_{\alpha_*}\to\Lambda_{\alpha_*}$, $\alpha^*=n(1/p_*-1)$, is invertible) 
and $\eps_0$ is sufficiently small, depending on the standard constants and on the norms of the inverses, then $\pm\frac 12 I+\widetilde{\K}^{L_1}:H^{p_*}\to H^{p_*}$ and $\mp\frac 12 I + \K^{L_1^*}:\Lambda_{\alpha_*}\to\Lambda_{\alpha_*}$ are invertible as well. 
\end{itemize}
\end{proposition}

A couple of comments are in order here. First, the boundary trace of the double layer potential is an adjoint operator to the boundary trace of the normal derivative of the single layer. To be precise, it follows directly from the definitions that for a $t$-independent operator ${\mathcal D}_t^{L_0^*}=adj \,(\partial_{\nu_{A_0}} S^{L_0})(\cdot, -t)$, $adj$ denoting the Hermitian adjoint on $\rn$, and $\widetilde{\K}^{L_0}=adj\,(\K^{L_0^*})$. This explains the nature of the statements in (v), (vi). For non-$t$ independent operators the situation is more complicated. We will discuss it in the course of the proof.

Secondly, we mention that, by extrapolation, given our boundedness results, invertibility of $\pm\frac 12 I+\widetilde{\K}^{L_i}:H^{p_*}\to H^{p_*}$ for a given $p_*\in (p_0, 2+\eps)$, implies invertibility in $H^p$ for $p$ in a small neighborhood of $p_*$, and analogous statements hold for  $\mp\frac 12 I + \K^{L_j^*}$ and $\SL^{L_j}_t\!\mid_{t=0}$.

\begin{proof} 

\noindent {\bf Step I:} the normal derivative of the single layer potential at the boundary. 

The existence of the limit in (i) in the appropriate sense and mapping properties of the emerging boundary operators follow directly from Lemma~\ref{l7.1}. Respectively, we can define $\widetilde{\K}^{L_1}:H^p\to H^p$ such that $\partial^+_{\nu_{A_1}} \SL^{L_1} f = (\frac{1}{2} I+\widetilde{\K}^{L_1})f$ and the desired limiting properties will hold. The same can be done in the lower-half space. It only remains to justify that that the difference between the operators emerging in the upper and the lower-half space is indeed the identity. To this end, it is enough to show that 
\begin{equation}\label{eq7.15.0}
\partial^+_{\nu_{A_1}} \SL^{L_1} f-\partial^-_{\nu_{A_1}} \SL^{L_1} f=f, \quad \mbox{for every $f\in C_0^\infty(\rn)$}. 
\end{equation}
It is convenient to postpone the argument until the end of the proof of the Proposition (see Step V).

Now we turn to the question of preservation of the invertibility by the corresponding operators, that is, to the corresponding parts of (v) and (vi).
Following the arguments as in Lemma \ref{l7.1}, ((iii) and (iv), {\it loc. cit.}), we can prove that 
\begin{equation}\|\partial_{\nu_{A_1}}\SL^{L_1} f-  \partial_{\nu_{A_0}}\SL^{L_0} f\|_{H^p} \lesssim  \varepsilon_0 \|f\|_{H^p}.\label{eq7.16}\end{equation}
Let us track the details. Recall the argument and notation of Step III in the proof of Lemma~\ref{l7.1}, in particular, the definitions of $\varphi_t$ and $\Phi_{x,t}$. With the same definitions, we have
\begin{multline}\label{eq7.16.1}\left|\left(\varphi_t*(\partial_{\nu_{A_1}}\SL^{L_1} f-  \partial_{\nu_{A_0}}\SL^{L_0} f)\right)(x)\right|=|\langle \partial_{\nu_{A_1}}\SL^{L_1} f-  \partial_{\nu_{A_0}}\SL^{L_0} f,\varphi_t(x-\cdot)\rangle|\\
 =\left|\iint_{\!\mathbb{R}^{n+1}_+}(A_1(y,s) \nabla \SL^{L_1}_s f(y) - A_0(y) \nabla \SL^{L_0}_s f(y))\overline{\nabla\Phi_{x,t}(y,s)}\,dyds\right|\\
 \leq \left|\iint_{\!\mathbb{R}^{n+1}_+}A_1(y,s) (\nabla \SL^{L_1}_s f (y)- \nabla \SL^{L_0}_s f(y))\overline{\nabla\Phi_{x,t}(y,s)}\,dyds\right|\\[4pt]
 \qquad +\left|\iint_{\!\mathbb{R}^{n+1}_+}(A_1(y,s) - A_0(y) )\nabla \SL^{L_0}_s f(y)\overline{\nabla\Phi_{x,t}(y,s)}\,dyds\right|.
 \end{multline}
 
By \eqref{eq6.29.7} (with the roles of $L_0$ and $L_1$ interchanged) the last expression in \eqref{eq7.16.1} is equal to 
 \begin{multline*}  
\left|\iint_{\!\mathbb{R}^{n+1}_+}A_1(y,s) (\nabla L_0^{-1}\dv (A_1-A_0) \nabla \SL^{L_1} f)(y,s)\overline{\nabla\Phi_{x,t}(y,s)}\,dyds\right|\\[4pt]
 \qquad +\left|\iint_{\!\mathbb{R}^{n+1}_+}(A_1(y,s) - A_0(y) )\nabla \SL^{L_0}_s f(y)\overline{\nabla\Phi_{x,t}(y,s)}\,dyds\right|\\
 \lesssim \,t^{-n-1}\iint_{\!\mathbb{R}^{n+1}_+\cap 2B} |\nabla L_0^{-1}\dv (A_1-A_0) \nabla \SL^{L_1} f(y,s)|\,dyds\\[4pt]
 +  t^{-n-1}\iint_{\!\mathbb{R}^{n+1}_+\cap 2B} |(A_1(y,s)-A_0(y))\nabla \SL^{L_0}_s f(y)|\,dyds\\
 \lesssim \,t^{-n-1}\left(\int_{\rn}\left(\N(|\nabla L_0^{-1}\dv (A_1-A_0) \nabla \SL^{L_1} f| 1_{2B})(y)\right)^{n/(n+1)}dy\right)^{(n+1)/n}\\
 \qquad + t^{-n-1}\left(\int_{\rn}\left(\N(|(A_1-A_0)\nabla \SL^{L_0}_t f| 1_{2B})(y)\right)^{n/(n+1)}dy\right)^{(n+1)/n}\\
 \lesssim \,t^{-n-1}\left(\int_{|x-y|<Ct}\left(\N(|\nabla L_0^{-1}\dv (A_1-A_0) \nabla \SL^{L_1} f| )(y)\right)^{n/(n+1)}dy\right)^{(n+1)/n}\\
 \qquad + \eps_0 t^{-n-1}\left(\int_{|x-y|<Ct}\left(\N(|\nabla \SL^{L_0}_t f| )(y)\right)^{n/(n+1)}dy\right)^{(n+1)/n}.
 \end{multline*}
Combining this with the considerations in \eqref{eq7.6}--\eqref{eq7.8.1} and invoking \eqref{eq1.14} and \eqref{eq4.7}, we conclude  \eqref{eq7.16}.

Having \eqref{eq7.16} at hand, the method of continuity shows that if there exists $p_*\in (n/(n+\alpha_0), 2+\eps)$ such that $\pm\frac 12 I+\widetilde{\K}^{L_0}:H^{p_*}\to H^{p_*}$ is invertible, then $\pm\frac 12 I+\widetilde{\K}^{L_1}:H^{p_*}\to H^{p_*}$ is invertible as well, provided that $\eps_0$ is sufficiently small depending on the standard constants {\it and on the norm of the inverse of $\pm\frac 12 I+\widetilde{\K}^{L_0}$}. 

\vskip 0.08 in

\noindent {\bf Step II:} the tangential derivative of the single layer potential at the boundary. 

As above, (ii) follows directly from Lemma~\ref{l7.1}, and hence we only have to establish the preservation of invertibility, that is, (iv). We have to show that 
\begin{equation}
\|\nabla_\|\SL_t^{L_1}\!\mid_{t=0} f -\nabla_\|\SL_t^{L_0}\!\mid_{t=0} f \|_{H^p} \lesssim \varepsilon_0 \| f\|_{H^p},
\label{eq7.18}\end{equation}
for any $f\in H^p$, $p_0<p<2+\eps$ and then apply the method of continuity. However, by \eqref{eq7.1.2}
$$
\|\nabla_\|\SL_t^{L_1}\!\mid_{t=0} f -\nabla_\|\SL_t^{L_0}\!\mid_{t=0} f \|_{H^p} \lesssim \|\nabla_\|\SL_t^{L_1} f -\nabla_\|\SL_t^{L_0}f \|_{\tT^p_\infty},$$
which gives \eqref{eq7.18} by virtue of \eqref{eq6.29.8}.

\vskip 0.08 in

\noindent {\bf Step III:} the double layer potential at the boundary. 

\noindent {\bf Step III(a):} uniform estimates on slices. Let $L$ be as in Theorem~\ref{t1.13}.


First of all, by \eqref{eq1.15}, 
\begin{equation}\label{eq7.18.1.1} \sup_{t>0}\|\D_t^{L^*}f\|_{L^{p'}(\rn)}\lesssim \|f\|_{L^{p'}(\rn)},\quad 1<p<2+\eps,
\end{equation}
for every $f\in L^{p'}$. An analogous estimate holds in $\la(\rn)$:
\begin{equation}\label{eq7.18.u1}
\sup_{t>0} \|\D_t^{L^*} f\|_{\la(\rn)} \lesssim  \|f\|_{\la(\rn)}, \quad 0\leq \alpha<\alpha_0.
\end{equation}
For $\alpha>0$ this is simply a consequence of  \eqref{eq1.15.1}, using the definition of $\la$ spaces. It remains to treat the case of $BMO$ ($\alpha=0$). For future reference though we address $f\in \la$, $0\leq\alpha<\alpha_0$, in the argument below. 

To start, we observe that for any $f\in BMO(\rn)$ the double layer potential $\D_t^{L^*} f$ is well-defined in $\reu$ and satisfies \eqref{eq1.16}. In particular, $\D_t^{L^*} f\in W^{1,2}_{loc}(\reu)$ and hence, by De Giorgi-Nash-Moser estimates  \eqref{eq1.6}--\eqref{eq1.7}, for every fixed $t>0$ the function $\D_t^{L^*} f \in L^p_{loc}(\rn)$, $0<p\leq \infty$. We also note that $\D_t^{L^*} 1$ is well-defined and equal to zero for any $t>0$, in the sense of $BMO(\rn)$. Indeed, for any $\tau>0$, and $f, \varphi\in C_0^\infty(\rn)$
\begin{equation}\label{eq7.18.1.3}\langle \D_\tau^{L^*}f, \varphi\rangle=\langle f, \partial_{\nu_{A}}^+ \SL^{L, \tau} \varphi\rangle=-\langle f, \partial_{\nu_{A}}^- \SL^{L, \tau} \varphi\rangle,
\end{equation}
where
\begin{equation}\label{eq7.18.1.4}\SL^{L, \tau}_sf(y)  
\equiv\int\limits_{\mathbb{R}^{n}}\Gamma_{L}
(y,s,x,\tau)\,\varphi(x)\,dx.
\end{equation}
The conormal derivative of $\SL^{L_i, \tau}f$ is, as usual, defined in the weak sense, via \eqref{eq7.0},  taken in the upper half space when we write $\partial_{\nu_{A}}^+$ and in the lower half-space when we write $\partial_{\nu_{A}}^-$. The fact that  $\partial_{\nu_{A}}^+ \SL^{L, \tau} \varphi=-\partial_{\nu_{A}}^- \SL^{L, \tau} \varphi$ is justified by the observation that $\SL^{L, \tau}_s\varphi$ is a solution for $s\neq \tau$, $\tau\neq 0$ and thus, for any $F\in C_0^\infty(\rn\times (-|\tau|/2,|\tau|/2)$ such that $F\Big|_{\rn}=f$ we have
\begin{multline*} \langle f,\partial_{\nu_{A}}^+ \SL^{L, \tau} \varphi\rangle+\langle f,\partial_{\nu_{A}}^- \SL^{L, \tau} \varphi\rangle\\[4pt]=\iint_{\reu} \overline{\nabla_{y,s}F(y,s)} \,A(y,s) \SL^{L, \tau}_s \varphi(y)\,dyds+\iint_{\rel} \overline{\nabla_{y,s}F(y,s)} \,A(y,s) \SL^{L, \tau}_s \varphi(y)\,dyds
\\[4pt]=\iint_{\ree} \overline{\nabla_{y,s}F(y,s)} \,A(y,s) \SL^{L, \tau}_s \varphi(y)\,dyds=0.
\end{multline*}

 For future reference we record that, due to \eqref{eq7.18.1.1}, \eqref{eq7.18.1.3}, we have 
\begin{equation}\label{eq7.18.1.5} \sup_{\tau>0}\|\partial_{\nu_{A}} \SL^{L, \tau} \varphi\|_{L^{p}(\rn)}\lesssim \|\varphi\|_{L^{p}(\rn)}, \quad 1<p<2+\eps,
\end{equation}
and the duality \eqref{eq7.18.1.3}  extends to  $f\in L^{p'}$ and $\varphi\in L^p$.  Using the weak definition of conormal derivative, one can easily see that $\int_{\rn} \partial_{\nu_{A}} \SL^{L, \tau} a\, dx=0$, for any $\tau>0$ and for any $H^p$-atom $a$. This implies that $\D_t^{L^*} 1=0$ for any $t>0$, in the sense of $\la(\rn)$, as desired.

Now let us fix $f\in \la(\rn)$ and some $H^p$-atom $a$, $\supp a\subset Q$, $\alpha=n(1/p-1)$. Denote $f_{4Q}=\fint_{4Q}f$, $ f-f_{4Q}=f_0+\sum_{k=2}^\infty f_k$, where $f_0=(f-f_{4Q}) {\bf 1}_{4Q}$ and $f_k=(f-f_{4Q}) {\bf 1}_{2^{k+1}Q\setminus 2^k Q}$. Then 
$$\langle \D_t^{L^*}f_0, a\rangle \lesssim \|f_0\|_{L^2(4Q)}\|a\|_{L^2(4Q)}\lesssim l(Q)^{-\alpha}\left(\fint_{4Q}|f-f_{4Q}|^2\,dx\right)^{1/2} \lesssim \|f\|_{\la}.$$
On the other hand, denoting by $x_Q$ the center of $Q$ and using the vanishing moment condition of $a$, we have
\begin{multline}\label{eqD1}
\langle \D_t^{L^*}f_k, a\rangle=\int_{Q} \left(\D_t^{L^*}f_k(x)-\D_t^{L^*}f_k(x_Q)\right)\,a(x)\,dx
\\[4pt] \lesssim 
2^{-k\alpha_0} l(Q)^{-\alpha} \left(\fiint_{2^{k-1}Q\times (t-2^{k-1}l(Q), t+2^{k-1}l(Q))} |\D_t^{L^*}f_k(x)|^2\,dx\right)^{1/2},
\end{multline}
where we used the fact that $f_k$ is supported away from $2^{k-1}Q$, and hence, $\D_t^{L_1^*}f_k$ is a solution in $2^{k-1}Q\times (t-2^{k-1}l(Q), t+2^{k-1}l(Q))$. Now we can use uniform in $t$ estimates for the double layer potential in $L^2$ to show that the expression above is controlled by
\begin{equation}\label{eqD2}C\, l(Q)^{-\alpha}2^{-k\alpha_0} (2^kl(Q))^{-\frac{n+1}{2}}\|f_k\|_{L^2(\rn)} \lesssim 2^{-k(\alpha_0-\alpha)} (2^kl(Q))^{-\alpha}\left(\fiint_{2^{k+1}Q\setminus 2^kQ} |f-f_{4Q}|^2\,dx\right)^{1/2}.\end{equation} 
Now, using the usual telescoping argument, we conclude that
\begin{equation}\label{eqD3}\left|\langle \D_t^{L^*}f, a\rangle\right|=\left|\langle \D_t^{L^*}(f-f_{4Q}), a\rangle\right| \lesssim \|f\|_{\la},\end{equation} 
for any $H^p$ atom $a$, which yields \eqref{eq7.18.u1} for $0\leq \alpha<\alpha_0$.

\vskip 0.08 in
\noindent {\bf Step III(b):} convergence.

The next order of business is to show that for any $f\in C_0^\infty(\rn)$ the sequence $\D_t^{L^*}f $ converges as $t\to 0$ in the sense of distributions. Given \eqref{eq7.18.1.1}, this will entail that there exists a boundary operator, to be  denoted $-\frac{1}{2}I+ \K^{L^*}$, which is bounded in $L^{p'}$ and such that for every $f\in L^{p'}$ we have 
\begin{equation}\label{eq7.18.1.2}\D_t^{L^*}f \to(-\textstyle{\frac{1}{2}}I+ \K^{L^*})f \quad\mbox{weakly in $L^{p'}$.}
\end{equation}

Since $L=L_0$ (the $t$-independent case) has been treated in \cite{HMiMo}, here we write $L_1$ in place of $L$ although the same argument applies to $L_0$.

To this end, take any $f, \varphi \in C_0^\infty(\rn)$ and let us show that $\langle \D_\tau^{L_1^*}f, \varphi\rangle$ is Cauchy.  In fact, we shall prove a slightly stronger statement, the bound on $\langle \D_\tau^{L_1^*}f, \varphi\rangle-\langle \D_{\tau'}^{L_1^*}f, \varphi\rangle$ in terms of $\|\varphi\|_{L^1(\rn)}$, thus establishing convergence of $ \D_\tau^{L_1^*}f$ on compacta of $\rn$.

Take  any $0<\tau'< \tau$, and write for $f, \varphi \in C_0^\infty(\rn)$ and $F\in C_0^\infty(\ree)$ such that $F\Big|_{\rn}=f$,
\begin{multline}\label{eq7.18.c3}
\langle \D_\tau^{L_1^*}f-\D_{\tau'}^{L_1^*}f, \varphi\rangle=\iint_{\rel} \overline{A_1(y,s)(\nabla \SL^{L_1, \tau}_s-\nabla \SL^{L_1, \tau'}_s)\, \varphi (y)}\,\nabla_{y,s} F(y,s)\,dyds\\[4pt]
=\iint_{\rel} \overline{A_1(y,s)\int_{\rn}\nabla_{y,s}(\Gamma_{L_1}(y,s,x,\tau)-\Gamma_{L_1}(y,s,x,\tau')) \varphi(x)\,dx}\,\nabla_{y,s} F(y,s)\,dyds\\[4pt]
\leq C \|\varphi\|_{L^1(\rn)} \sup_{x\in \supp \varphi} \iint_{\rel} \left|\nabla_{y,s}(\Gamma_{L_1}(y,s,x,\tau)-\Gamma_{L_1}(y,s,x,\tau')) \right|\,\left|\nabla_{y,s} F(y,s)\right|\,dyds
\\[4pt]
\leq C_{F} \|\varphi\|_{L^1(\rn)} \sup_{x\in \supp \varphi} \iint_{\rel\cap \supp F} \left|\nabla_{y,s}(\Gamma_{L_1}(y,s,x,\tau)-\Gamma_{L_1}(y,s,x,\tau')) \right|\,dyds.
\end{multline}

\noindent Now let $B_{x,\tau}:= B((x,\tau), c\tau)$, 
and let $S_k(B_{x,\tau})$ be the dyadic annuli around $B_{x,\tau}$. The constant $c$ is such that $(x,\tau')\in \frac 12 B_{x,\tau}$. For a given set $E$ we denote by $U_{r}(E)$ an $c'r$-neighborhood of $E$, where the constant $c'$ is chosen (and fixed throughout the argument) to preserve suitable separation. For instance, for all $k\geq 2$ we have $\dist\{U_{2^k\tau} (S_k(B_{x,\tau})), 2^{k-2}B_{x,\tau}\}\approx 2^k\tau$. Then for every fixed $x\in \supp \varphi$
\begin{multline*}
 \iint_{\rel\cap \supp F} \left|\nabla_{y,s}(\Gamma_{L_1}(y,s,x,\tau)-\Gamma_{L_1}(y,s,x,\tau')) \right|\,dyds \\[4pt]
 = \iint_{2B_{x,\tau}\cap \supp F} \left|\nabla_{y,s}(\Gamma_{L_1}(y,s,x,\tau)-\Gamma_{L_1}(y,s,x,\tau')) \right|\,dyds
 \\[4pt]
 +\sum_{k\geq 2: \,S_k(B_{x,\tau})\cap \supp F\neq \emptyset} \iint_{S_k(B_{x,\tau})} \left|\nabla_{y,s}(\Gamma_{L_1}(y,s,x,\tau)-\Gamma_{L_1}(y,s,x,\tau')) \right|\,dyds=:I+II.
\end{multline*}
Now,  
\begin{multline*}
 II\lesssim \sum_{k\geq 2: \,S_k(B_{x,\tau})\cap \supp F\neq \emptyset} (2^k\tau)^{\frac{n+1}{2}} \left(\iint_{S_k(B_{x,\tau})} \left|\nabla_{y,s}(\Gamma_{L_1}(y,s,x,\tau)-\Gamma_{L_1}(y,s,x,\tau')) \right|^2\,dyds\right)^{1/2}
 \\[4pt]
 \lesssim \sum_{k\geq 2: \, S_k(B_{x,\tau})\cap \supp F\neq \emptyset} (2^k\tau)^{\frac{n+1}{2}-1} \left(\iint_{U_{2^k\tau}(S_k(B_{x,\tau}))} \left|\Gamma_{L_1}(y,s,x,\tau)-\Gamma_{L_1}(y,s,x,\tau') \right|^2\,dyds\right)^{1/2}
 \\[4pt]
 \lesssim \sum_{k\geq 2: \, S_k(B_{x,\tau})\cap \supp F\neq \emptyset} (2^k\tau)^{\frac{n+1}{2}-1-\alpha_0} |\tau-\tau'|^{\alpha_0} \left(\iint_{U_{2^k\tau}(S_k(B_{x,\tau}))}\fiint_{2^{k-2}  B_{x,\tau}} \left|\Gamma_{L_1}(y,s,z,r) \right|^2\, dzdr \,dyds\right)^{1/2}
  \\[4pt]
 \lesssim \sum_{k\geq 2: \,  S_k(B_{x,\tau})\cap \supp F\neq \emptyset} (2^k\tau)^{\frac{n+1}{2}-1-\alpha_0} |\tau-\tau'|^{\alpha_0}  (2^k\tau)^{\frac{n+1}{2}-(n-1)}   \\[4pt]\lesssim \sum_{k\geq 2: \, S_k(B_{x,\tau})\cap \supp F\neq \emptyset}  |\tau-\tau'|^{\alpha_0}  (2^k\tau)^{1-\alpha_0}\lesssim C_F |\tau-\tau'|^{\alpha_0}, 
\end{multline*}
where we used Caccioppoli inequality in $y,s$ in the second inequality, H\"older continuity of solutions in $x,\tau$ in the third one, and pointwise bounds on the fundamental solution for the fourth inequality. 
As for $I$, for any $\tau'\leq \tau$ we have
\begin{multline*}
\iint_{2B_{x,\tau}\cap \supp F} \left|\nabla_{y,s}\Gamma_{L_1}(y,s,x,\tau')\right|\,dyds\\[4pt]
\lesssim \iint_{2B_{x,\tau'}\cap 2B_{x,\tau}\cap \supp F} \left|\nabla_{y,s}\Gamma_{L_1}(y,s,x,\tau')\right|\,dyds\\[4pt]
+\sum_{k\geq 2: \,S_k(B_{x, \tau'})\cap 2B_{x,\tau}\cap \supp F\neq\emptyset}\iint_{S_k(B_{x, \tau'})\cap 2B_{x,\tau}\cap \supp F} \left|\nabla_{y,s}\Gamma_{L_1}(y,s,x,\tau)\right|\,dyds.
\end{multline*}
Here the implicit constant $c_0$ in the definition of $B_{x,\tau'}=B((x,\tau'), c_0\tau')$ is chosen so that $3B_{x,\tau'}$ is separated from $F$. In particular, it is normally smaller than the constant $c$ in the definition of $B_{x,\tau}$ above. Hence, the first integral is necessarily zero, even if $\tau'=\tau$. Then the expression above is further controlled by
\begin{multline*}
\sum_{k\geq 2: \,S_k(B_{x, \tau'})\cap 2B_{x,\tau}\cap \supp F\neq\emptyset}(2^k\tau')^{\frac{n+1}{2}} \left(\iint_{S_k(B_{x, \tau'})\cap 2B_{x,\tau}\cap \supp F} \left|\nabla_{y,s}\Gamma_{L_1}(y,s,x,\tau)\right|^2\,dyds\right)^{1/2} \\[4pt] \lesssim \sum_{k\geq 2: \,S_k(B_{x, \tau'})\cap 2B_{x,\tau}\cap \supp F\neq\emptyset}(2^k\tau')^{\frac{n+1}{2}-1} \left(\iint_{U_{2^k\tau'}(S_k(B_{x, \tau'})\cap 2B_{x,\tau}\cap \supp F)} \left|\Gamma_{L_1}(y,s,x,\tau)\right|^2\,dyds\right)^{1/2}\\[4pt]
\lesssim \sum_{k\geq 2: \,S_k(B_{x, \tau'})\cap 2B_{x,\tau}\cap \supp F\neq\emptyset}2^k\tau' \lesssim C_f\,\tau.\end{multline*}
\noindent Note that this computation automatically covers both integrals arising from $I$, taking $\tau'=\tau$ to handle the first one.

All in all, this line of reasoning shows that 
\begin{equation}\left|\langle \D_\tau^{L_1^*}f-\D_{\tau'}^{L_1^*}f, \varphi\rangle\right|\leq C_{f} \|\varphi\|_{L^1(\rn)}\left( |\tau-\tau'|^{\alpha_0}+\tau\right),\label{eq7.18.c4}\end{equation}
and thus, the convergence in the sense of distributions holds and \eqref{eq7.18.1.2} is valid, for some operator $\K^{L_1^*}$, which is bounded in $L^{p'}$. Moreover, \eqref{eq7.18.c4} shows that $\D_\tau^{L_1^*} f$ converges to its boundary data sternly in $L^{p'}$ on compacta of $\rn$.

Let us now establish weak* convergence in $\la$, $0\leq\alpha<\alpha_0$. To this end, take $f\in \la$, $0\leq\alpha<\alpha_0$ and an $H^p$-atom $a$, $p=n/(\alpha+n)$. Fix $\eps>0$ and let $k_0\in \NN$ be a large number to be chosen depending on $\eps$. Let $f-f_{2^{k_0}Q}=f_{k_0}+\sum_{k=k_0}^\infty f_k$, where $f_{k_0}=(f-f_{2^{k_0}Q}) {\bf 1}_{2^{k_0}Q}$ and $f_k=(f-f_{2^{k_0}Q}) {\bf 1}_{2^{k}Q\setminus 2^{k-1} Q}$, $k\geq k_0+1$. By weak-$L^2$ convergence proved above there exists $\delta=\delta_{\eps, f,a,k_0}$, such that 
\begin{equation}\label{eq7.18.c1}\left|\langle \D_\tau^{L_1^*}f_{k_0}-\D_{\tau'}^{L_1^*}f_{k_0}, a\rangle\right|<\eps \quad\mbox{for all}\quad \tau,\tau'<\delta. 
\end{equation}
On the other hand, for $0<\tau, \tau'<l(Q)$
\begin{multline}\label{eq7.18.c2}\sum_{k=k_0+1}^\infty \left|\langle \D_\tau^{L_1^*}f_k-\D_{\tau'}^{L_1^*}f_k, a\rangle\right|\lesssim l(Q)^{-\alpha} \sum_{k=k_0+1}^\infty  \left(\fint_Q |\D_\tau^{L_1^*}f_k-\D_{\tau'}^{L_1^*}f_k|^2\, dx\right)^{1/2}\\[4pt]
\lesssim  l(Q)^{-\alpha}\sum_{k=k_0+1}^\infty 2^{-k\alpha_0} \left(\fiint_{2^{k-2}Q\times (-2^{k-1}l(Q), 2^{k-1}l(Q))} |\D_\tau^{L_1^*}f_k|^2\, dx\right)^{1/2}\\[4pt] \lesssim  l(Q)^{-\alpha}\sum_{k=k_0+1}^\infty 2^{-k\alpha_0} (2^kl(Q))^{-n/2}\|f_k\|_{L^2}\lesssim  \|f\|_{\la} \sum_{k=k_0+1}^\infty k\,2^{k(\alpha-\alpha_0)}. \end{multline}
Now we choose $k_0=k_0(\eps, f)$ so that the expression on the right-hand side above is controlled by $\eps$ and this finishes the proof of the weak-* convergence in $\la$. Moreover, by the same reasoning as above one can conclude that the boundary operator 
$\K^{L_1^*}$, defined a priori in $L^{p'}$, is well-defined and bounded in $\la$ for $0\leq \alpha<\alpha_0$, with $ \D_\tau^{L_1^*}f$ converging to $\K^{L_1^*}f$ for every $f\in\la$ in the weak* sense.

In fact, one can say more than that. 

\vskip 0.08 in \noindent {\bf Step IV}: Let us demonstrate that $\K^{L_1^*}=adj (\widetilde \K^{L_1})$ or rather (before the justification of the jump formulas) that $ \D_s^{L_1^*}\Big|_{s=0^{\pm}}=adj \,(\partial_{\nu_{A_1}}^{\mp}\SL^{L_1})$. 

Following \eqref{eq7.18.c3}--\eqref{eq7.18.c4} one can demonstrate that for every $f, \varphi\in C_0^\infty(\rn)$ (and $F\in C_0^\infty(\ree)$ - an extension of $f$) and $r<0$, $\tau>0$ such that $\tau\ll |r| \ll {\rm diam} \,(\supp F)$ there holds
\begin{equation}\label{eq7.18.1.10}
\left|\langle\partial_{\nu_{A_1}}^-\SL^{L_1, \tau}\varphi(\cdot, r)-\partial_{\nu_{A_1}}^-\SL^{L_1}\varphi(\cdot, r), f\rangle \right|\lesssim C_{\varphi, f} \,\tau^{\alpha_0}. 
\end{equation}
\noindent  Here $\partial_{\nu_{A_1}}^-\SL^{L_1, \tau}\varphi(\cdot, r)$ and $\partial_{\nu_{A_1}}^-\SL^{L_1}\varphi(\cdot, r)$, $r<0$, are the conormal derivatives at $\RR^n\times\{t=r\}$, in the sense of \eqref{eq7.0}. In particular,  
\begin{equation}\label{eq7.18.1.12}
\lim_{r\to 0} \left|\langle\partial_{\nu_{A_1}}^-\SL^{L_1}\varphi(\cdot, r)-\partial_{\nu_{A_1}}^-\SL^{L_1}\varphi, f\rangle \right|=0,  
\end{equation}
as per (iv) of Lemma~\ref{l7.1}, and 
\begin{equation}\label{eq7.18.1.11}
\lim_{r\to 0} \sup_{\tau>0} \left|\langle\partial_{\nu_{A_1}}^-\SL^{L_1, \tau}\varphi(\cdot, r)-\partial_{\nu_{A_1}}^-\SL^{L_1, \tau}\varphi, f\rangle \right|=0.   
\end{equation}

The latter can be shown following the lines of the proof of \eqref{eq7.18.1.12}, as soon as we confirm that $\nabla\SL^{L_1, \tau}\varphi \in L^2(\rel)$, uniformly in $\tau>0$. However, by ellipticity, $L_1^{-1}: (\dot W^{1,2}(\ree))^*\to \dot W^{1,2}(\ree) $ and hence, by trace theorems, $\tr_\tau \circ L_1^{-1}: (\dot W^{1,2}(\ree))^*\to \dot H^{1/2}(\rn)$, where $\tr_\tau$ is the trace on the hyperplane $\rn\times\{t=\tau\}$, and the implicit estimates are uniform in $\tau$. By duality, analogously to \eqref{eq6.29.4}, for every $\varphi\in C_0^\infty(\rn)$ and $\Psi\in C_0^\infty(\rel)$ we have
$$ \langle \nabla \SL^{L_1,\tau}_s \varphi, \Psi \rangle=\iint_{\rel}  \nabla_{y,s} \SL^{L_1,\tau}_s \varphi(y) \,\Psi (y,s)\,dyds=
\langle \varphi, \tr_\tau \circ (L^*)^{-1}\dv \Psi\rangle.$$
Due to the considerations above this extends to all $\Psi\in L^2(\rel)$ with the uniform in $\tau$ estimates. Hence, $\nabla\SL^{L_1, \tau}\varphi \in L^2(\rel)$ for every $\varphi \in C_0^\infty(\rn)$ with the estimates uniform in $\tau$. At this point we can indeed invoke the argument akin to that of \eqref{eq7.18.1.12}. Using the definition of weak derivative, we have to estimate, for any $f,\varphi\in C_0^\infty$ and any $F, \widetilde F\in C_0^\infty(\ree)$ with $F\Big|_{\rn}=\widetilde F\Big|_{\rn}=f$ the difference
$$\iint_{\rel} A(x,r+s) \nabla \SL^{L_1,\tau}\varphi(x,r+s) \overline{\nabla \widetilde F(x,s)} dx  ds- \iint_{\rel} A(x,s) \nabla \SL^{L_1,\tau}\varphi(x,s) \overline{\nabla  F(x,s)} dx  ds.
$$ 
Since all extensions give equal integrals, we take $\widetilde F$ such that $\widetilde F(x,s)=F(x,r+s)$. Then the absolute value of the difference above is bounded by 
$$ \left|\iint_{(x,s)\in\rel:\, r<s<0} A(x,s) \nabla \SL^{L_1,\tau}\varphi(x,s) \overline{\nabla  F(x,s)} dx  ds\right|\lesssim C_F |r|^{1/2} \|\nabla \SL^{L_1,\tau}\varphi\|_{L^2(\rel)},  $$
which, in view of the uniform bounds on $\|\nabla \SL^{L_1,\tau}\varphi\|_{L^2(\rel)}$,
justifies \eqref{eq7.18.1.11}.


Now, combining \eqref{eq7.18.1.10}--\eqref{eq7.18.1.11}, and taking the limit in $\tau$ and then in $r$, we conclude that in fact, for every $\varphi\in C_0^\infty(\rn)$ the quantity $\partial_{\nu_{A_1}}^-\SL^{L_1, \tau}\varphi$ converges to $\partial_{\nu_{A_1}}^-\SL^{L_1}\varphi$ in the sense of distributions (and weakly in $L^p$, $1<p<2+\eps$). Comparing this to \eqref{eq7.18.1.3} and invoking the results of Step I, we see that the boundary operator identified in \eqref{eq7.18.1.2} is a Hermitian adjoint on $\rn$ of the normal derivative of the single layer on the boundary. That is, for all $f, \varphi\in C_0^\infty (\rn)$ we have 
\begin{equation}\label{eqD4} \left\langle \D_s^{L_1^*}\Big|_{s=0^{\pm}}f, \varphi\right\rangle 
=\left\langle f, \partial_{\nu_A}^{\mp} \SL^{L_1}\,\varphi\right\rangle,
\end{equation}
 and hence,  $ \D_s^{L_1^*}\Big|_{s=0^{\pm}}=adj \,(\partial_{\nu_{A_1}}^{\mp}\SL^{L_1})$ as an operator in $L^{p'}$. 

Now, one has to show that for any $f\in \la$ and $a$-an $H^p$ atom we have 
$ \left\langle \D_s^{L_1^*}\Big|_{s=0^{\pm}}f, a\right\rangle 
=\left\langle f, \partial_{\nu_A}^{\mp} \SL^{L^1}\,a\right\rangle$. This follows essentially the arguments in \eqref{eq7.18.c1}--\eqref{eq7.18.c2} and \eqref{eqD1}--\eqref{eqD3}. With the same notation for $f_{k_0}$ and $f_k$, $k>k_0$, we have 
$$ \left\langle \D_s^{L_1^*}\Big|_{s=0^{\pm}}f_{k_0}, a\right\rangle 
=\left\langle f_{k_0}, \partial_{\nu_A}^{\mp} \SL^{L_1}\,a\right\rangle,$$
simply because both both $f_{k_0}$ and $a$ belong to $L^2$. Furthermore, we claim that for any $\eps>0$ there exists $k_0=k_0(a,f,\eps)$ such that 
$$ \sum_{k=k_0+1}^\infty \left\langle \D_s^{L_1^*}\Big|_{s=0^{\pm}}f_{k}, a\right\rangle<\eps.$$
Note that the supports of $a$ and $f_k$, $k>k_0$, are disjoint, and hence, $\D_s^{L_1^*}\Big|_{s=0^{\pm}}f_{k}$ is a solution across the corresponding part of the boundary. Thus, arguing as in \eqref{eqD1}--\eqref{eqD3}, we can establish the same estimates as in \eqref{eq7.18.c2} for 
$$ \langle \D_s^{L_1^*}\Big|_{s=0^{\pm}}f_k, a\rangle=\int_{Q} \left(\D_s^{L_1^*}\Big|_{s=0^{\pm}}f_k(x)-\D_s^{L_1^*}\Big|_{s=0^{\pm}}f_k(x_Q)\right)\,a(x)\,dx,$$
in terms of $\eps$ provided that $k_0$ is sufficiently large and finish the proof of \eqref{eqD4}. 

This justifies the desired duality relations for the trace of the double layer and normal derivative of the single layer potential on the boundary and thus verifies jump formula for the double layer potential in $L^{p'}$ and in $\la$ (assuming  \eqref{eq7.15.0} which we establish below in Step V).  

We have concluded the proof of the entire statement (iii). 
Moreover one establishes the validity of the invertibility claimed in (v), (vi) for the operator $\K^{L_1^*}$, also by duality. 

It only remains to establish the jump formulas. 

\vskip 0.08 in \noindent {\bf Step V}: jump relations. 

We only have to verify \eqref{eq7.15.0}. We adopt the general line of reasoning from \cite{AAAHK}, although the fact that $A_1$ depends on $t$ brings up some changes. 

We aim to show that for all $\Psi\in C_0^\infty(\ree)$ we have
\begin{equation}\label{eq7.15.1}
\iint_{\reu} A\nabla u^+\,\nabla \Psi\,dxdt+ \iint_{\rel} A\nabla u^-\,\nabla \Psi\,dxdt = \int_{\rn} f\,\Psi\,dx,
\end{equation}
where $u^{\pm}=\SL^{L_1}_tf$, $t\in \RR_{\pm}$. Let $u_\eta^\pm(\cdot, t)=(L_1^{-1} (f_\eta))(\cdot, t)$, $t\in \RR_{\pm}$, $f_\eta(y,s):=f(y)\varphi_\eta(s)$, $(y,s)\in \ree$, where $\varphi_\eta$ is, as usual, the kernel of a smooth approximate identity, in particular, $\int \varphi_\eta =1$ and $\varphi_\eta\in C_0^\infty(-\eta, \eta)$. Finally, let $U_\eta=u_\eta^+ {\bf 1}_{\reu}+u_\eta^- {\bf 1}_{\rel}$. Then 
\begin{multline*} \iint_{\reu} A\nabla u^+_\eta\,\nabla \Psi\,dxdt+ \iint_{\rel} A\nabla u^-_\eta\,\nabla \Psi\,dxdt \\[4pt]
= \iint_{\ree} A\nabla U_\eta\,\nabla \Psi\,dxdt=\iint_{\ree} f_\eta \Psi \,dxdt\to \int_{\rn} f\Psi \,dx,
\end{multline*}
as $\eta \to 0$. So far this is the same set up as in \cite{AAAHK}. Now we have to prove that the left-hand side of the expression above converges to the left-hand side of \eqref{eq7.15.1} as $\eta\to 0$. To this end, fix some $\delta\gg \eta$ and split 
\begin{equation}\label{eq7.15.2} \iint_{\reu} A\nabla (u^+_\eta-u^+)\,\nabla \Psi\,dxdt=\int_\delta^\infty\!\!\!\!\int_{\rn}+\int_0^\delta \!\!\!\int_{\rn}=: I_{\delta, \eta}+II_{\delta, \eta}.
\end{equation}
Let us start with $I_{\delta, \eta}$. Given that $\delta\gg\eta$, one can estimate the $L^2$ norm of
$$(\nabla u^+ -\nabla u_\eta^+)(x,t)=\iint_{\ree} \nabla_{x,t}(\Gamma(x,t; y,0)-\Gamma(x,t; y,s)) f(y)\varphi_\eta (s)\,dyds$$
using Caccioppoli inequality, H\"older continuity of solutions, and pointwise estimates on the fundamental solution, as well as the fact that $f$ is compactly supported. All in all, we will have 
$I_{\delta, \eta} \leq C_{f,\Psi, \delta}\, \eta^{\alpha_0}$. Hence, for a fixed $\delta>0$, the integral $I_{\delta, \eta}$ converges to zero as $\eta\to 0.$

Turning to $II_{\delta, \eta}$, we observe that by \eqref{eq1.14} 
\begin{equation}\label{eqjf1}\sup_{0<t<\infty} \|(\nabla u^+)_W(\cdot, t)\|_{L^2(\rn)}\lesssim \sup_{0<t<\infty} \|(\nabla \SL^{L_1} f)_W(\cdot, t)\|_{L^2(\rn)}\lesssim \|f\|_{L^2(\rn)},
\end{equation} where $(\cdot)_W$, as before, denotes $L^2$-averaging over Whitney balls \eqref{eq2.12.1}. Since $\Psi\in C_0^\infty(\reu)$, \eqref{eqjf1} and Dominated Convergence Theorem entail  $$\int_{\rn}\!\int_0^\delta A\nabla u^+\,\nabla \Psi\,dxdt\to 0, \quad \mbox{as } \delta\to 0.$$
On the other hand, 
\begin{equation}\label{eqjf2}\|\nabla u^+_\eta\|_{L^2(\ree)}=\|\nabla L_1^{-1} (f\varphi_\eta)\|_{L^2(\ree)} \lesssim \sup_{\tau} \|\nabla S^{L_1, \tau} f\|_{L^2(\ree)}\leq C_f,
\end{equation}
with the constants independent of $\eta$ (see the discussion following \eqref{eq7.18.1.11} for the last estimate). Then, using once again the Dominated Convergence Theorem, 
 $$\sup_{\eta>0}\int_{\rn}\!\int_0^\delta A\nabla u_\eta^+\,\nabla \Psi\,dxdt\to 0, \quad \mbox{as } \delta\to 0.$$ Thus, $\sup_{\eta>0} II_{\delta, \eta}\to 0$ as $\delta\to 0$.
This finishes the argument of \eqref{eq7.15.0}. By density, we get the required jump relations for the normal derivative of the single layer in $H^p$, $p_0<p<2+\eps$, and, by duality, for the double layer in $L^{p'}$ when $1<p<2+\eps$ and in $\la$, $0\leq \alpha<\alpha_0$. 
\end{proof}
 
At this point Proposition~\ref{p7.15} established \eqref{eq1.20}--\eqref{eq1.22.1} and thus we finished the proof of Theorem~\ref{t1.13}. Moreover, combining  Proposition~\ref{p7.15} with  Theorem~\ref{t1.13} and employing the representation formulas \eqref{eq1.28.1}--\eqref{eq1.28.3}, we have proved Theorem~\ref{t1.32}, with the exception of the uniqueness statement. We turn to the latter.

\section{Uniqueness}\label{s8}

Le us start with some additional properties of fundamental solutions which will be useful in the proof of uniqueness. 

\begin{lemma}\label{l8.01} Let $L$ be an elliptic operator falling under the scope of Theorem~\ref{t1.13}. Then for any $1<p<2+\eps$ the fundamental solution of $L^*$ satisfies
\begin{equation}\label{eq8.02}
\left(\int_{\rn} |\nabla_{z}\Gamma^* (x,t,z,0)|^p\,dz\right)^{1/p} \leq Ct^{-n(1-1/p)}, \quad \mbox{for all } (x,t)\in \reu.
\end{equation} 
\end{lemma}

\begin{proof} Fix $(x,t)\in \reu$ and a cone $\Gamma_{x,t}$ with a vertex at $(x,t/2)$ of aperture 1. Let us further denote by $A^*_{01}$ the matrix equal to $A^*_1$ on $\Gamma_{x,t}$ and $A^*_0$ in $\reu\setminus \Gamma_{x,t}$, and by $L^*_{01}$ the corresponding operator. 

Since $A^*_{01}$ is $t$-independent near the boundary, \eqref{eq8.02} follows for the fundamental solution of $L^*_{01}$ from the ``flat" version of the Caccioppoli inequality (see Proposition~2.1 in \cite{AAAHK}):
\begin{multline}\label{eq8.04}
\left(\int_{\rn} |\nabla_{z}\Gamma^*_{01} (x,t,z,0)|^p\,dz\right)^{1/p}  \lesssim \sum_{k=0}^\infty (2^kt)^{n/p}\left(\fint_{S_k(\Delta(x,t))} |\nabla_{z}\Gamma^*_{01} (x,t,z,0)|^p\,dz\right)^{1/p} \\[4pt]
\lesssim \sum_{k=0}^\infty (2^kt)^{-1+n/p} \left(\fiint_{S_k(\Delta(x,t))\times (-2^kt, 2^kt)} |\nabla_{z}\Gamma^*_{01} (x,t,z,s)|^2\,dzds\right)^{1/2} \\[4pt]
\lesssim \sum_{k=1}^\infty (2^kt)^{-n(1-1/p)} \lesssim t^{-n(1-1/p)},
\end{multline}
where $S_k(\Delta(x,t))$ as usual, denote dyadic annuli around the ball $\Delta(x,t)\subset \rn$.

Now take some $f\in L^{p'}$ and observe that $\int_{\rn} \nabla_{z}\Gamma^* (x,t,z,0) f(z)\,dz = (\SL^{L^*_1}\nabla_\|)f(x)$ and furthermore, following our usual perturbation considerations (cf., e.g., \eqref{eq6.29.6}), we have 
$$ (\SL^{L^*_1}_t\nabla_\|)f=(\SL^{L^*_{01}}_t\nabla_\|)f+ (L^*_{01})^{-1} 
\dv (A^*_1-A^*_{01}) \nabla (\SL^{L^*_1}\nabla_\|)f (\cdot, t),$$
and in view of the computation in \eqref{eq8.04}, it remains to analyze the second term above. However, 
\begin{multline*}(L^*_{01})^{-1} \dv (A^*_1-A^*_{01}) \nabla (\SL^{L^*_1}_t\nabla_\|)f(x,t)
\\[4pt]=\iint_{\ree}\nabla_{y,s}  \Gamma^*_{01} (x,t,y,s) (A^*_1-A^*_{01})(y,s) \nabla_{y,s} (\SL^{L^*_1}_s\nabla_\|)f(y)\,dyds\end{multline*}
for $(x,t)\in\reu$. Hence, given the support of $A^*_1-A^*_{01}$, the fact that $|A^*_1-A^*_{01}|\leq |A^*_1-A^*_0|$, and \eqref{eq1.17}, we only have to prove that 
\begin{equation}\label{eq8.05}
\chi_{\ree\setminus \Gamma_{x,t}}\nabla  \Gamma^*_{01} (x,t,\cdot,\cdot) \in \widetilde T^p_\infty,\quad 1<p<2+\eps,
\end{equation}
uniformly in $(x,t)\in\reu$ with the norm $Ct^{-n(1-1/p)}$. 

Throughout this argument it will be convenient to use cylinders rather than balls in the definition of the non-tangential maximal function, and as usual, we assume that the aperture is small enough (but fixed) to ensure proper separation. In particular, we can always assume that for any $(y,s)\in \ree\setminus \Gamma_{x,t}$ the corresponding cylinder $C_{y,s}:=\Delta(y,cs)\times(s-cs,s+cs)$ and its slightly fattened version $\widetilde C_{y,s}:=\Delta(y,c_1s)\times(s-c_1s,s+c_1s)$, $c_1>c$,  have a distance from $(x,t)$ equivalent to $|(y,s)-(x,t)|$. Let us now separate several cases. 

If $(y,s)\in \reu\setminus \Gamma_{x,t}$ is such that $\widetilde C_{y,s}\cap \Gamma_{x,t}\neq\emptyset$ (so that, in particular, for all $(z,\tau)\in \widetilde C_{y,s}$ we have $|(x,t)-(z,\tau)|\approx |x-z|\approx \tau\approx s \gtrsim t$) we have by Caccioppoli inequality
\begin{multline}\label{eq8.06}
\left(\fiint_{C_{y,s}}| \nabla  \Gamma^*_{01} (x,t,z,\tau)|^2\,dzd\tau\right)^{1/2}\\[4pt]
 \lesssim \frac 1s \left(\fiint_{\widetilde C_{y,s}}|   \Gamma^*_{01} (x,t,z,\tau)|^2\,dzd\tau\right)^{1/2} \approx |(x,t)-(y,s)|^{-n}\approx (|x-y|+t)^{-n}.
\end{multline}

If, on the other hand, $(y,s)\in \reu\setminus \Gamma_{x,t}$ is such that $\widetilde C_{y,s}\cap \Gamma_{x,t}=\emptyset$ we adopt the argument from \cite{KP}, p. 494, to write 
\begin{multline}\label{eq8.07}
\left(\fiint_{C_{y,s}}| \nabla  \Gamma^*_{01} (x,t,z,\tau)|^2\,dzd\tau\right)^{1/2}\\[4pt]
 \lesssim \frac 1s \fiint_{\widetilde C_{y,s}}\Big|   \Gamma^*_{01} (x,t,z,\tau)-\fint_{\Delta(y,c_1s)}  \Gamma^*_{01} (x,t,w,0)\,dw \Big|\,dzd\tau
 \\[4pt]
 \lesssim \frac 1s \fiint_{\widetilde C_{y,s}}\Big|   \Gamma^*_{01} (x,t,z,\tau)- \Gamma^*_{01} (x,t,z,0) \Big|\,dzd\tau\\[4pt]
+ \frac 1s \fiint_{\widetilde C_{y,s}}\Big|   \Gamma^*_{01} (x,t,z,0)-\fint_{\Delta(y,c_1s)}  \Gamma^*_{01} (x,t,w,0)\,dw \Big|\,dzd\tau
 \\[4pt]
 \lesssim \sup_{z\in \Delta(y,c_1s), \,\tau<(1+c_1)s}\Big|  \partial_\tau \Gamma^*_{01} (x,t,z,\tau)\Big|
+  \fint_{\Delta(y,c_1s)}\Big| \nabla_z  \Gamma^*_{01} (x,t,z,0) \Big|\,dz=: I+II,
 \end{multline}
using Caccioppoli inequality and local Moser estimates for solutions for the first bound above and Poincar\'e inequality for the third one. The second term, $II$, is bounded by the Hardy-Littleweeod maximal function and hence we can invoke \eqref{eq8.04} to deduce the desired bound. 

Turning to $I$, we observe that for all $(z,\tau) \in \reu\setminus \Gamma_{x,t}$ 
the integrand satisfies $L_0\,{\overline{\Gamma^*_{01} (x,t,z,\tau)}}=0$, with $L_0$ acting in $z,\tau$ variables. In particular, this is valid for $z\in \Delta(y,c_1s), \,\tau<(1+c_1)s$ (with $(y,s)$ chosen so that $\widetilde C_{y,s}\cap \Gamma_{x,t}=\emptyset$) and their neighborhood of radius $c|(x,t)-(z,\tau)|\gtrsim |x-z|+t$ (modulo the points near the boundary of $\Gamma_{x,t}$ which would fall under the scope of the same argument as in \eqref{eq8.06}).

Since $L_0$ is $t$-independent, derivatives in $\tau$ are also solutions, and we can apply Moser local bounds. Then for any 
$z\in \Delta(y,c_1s), \,\tau<(1+c_1)s$, we have 
 \begin{multline}\label{eq8.08}
\Big|  \partial_\tau \Gamma^*_{01} (x,t,z,\tau)\Big| \lesssim \left(\fiint_{B((z,\tau), c|(x,t)-(z,\tau)| )} | \partial_\tau \Gamma^*_{01} (x,t,w,r)|^2\,dwdr \right)^{1/2}\\[4pt]
\lesssim |(x,t)-(z,\tau)|^{-1}\left(\fiint_{B((z,\tau), c'|(x,t)-(z,\tau)| )} | \Gamma^*_{01} (x,t,w,r)|^2\,dwdr \right)^{1/2}\\[4pt] \lesssim  |(x,t)-(z,\tau)|^{-n}\lesssim (|x-z|+t)^{-n}.
 \end{multline}
 
The $L^p(\rn)$ norm in $z$ of the  right-hand sides of \eqref{eq8.06} and \eqref{eq8.08} gives a bound by $t^{-n(1-1/p)}$, as desired. Moreover, if $(y,s)\in\rel$, the exact same argument as in \eqref{eq8.07}--\eqref{eq8.08} also gives the same bound. This finishes the proof of \eqref{eq8.05} and thus, of \eqref{eq8.02}.
\end{proof}

\begin{lemma}\label{l8.09} Let $L$ be an elliptic operator falling under the scope of Theorem~\ref{t1.13}. Then for any $1<p<2+\eps$ the fundamental solution of $L^*$ satisfies
\begin{equation}\label{eq8.010}
\chi_{\ree\setminus \Gamma_{x,t}}\nabla  \Gamma^* (x,t,\cdot,\cdot) \in \widetilde T^p_\infty(\reu),\quad 1<p<2+\eps,
\end{equation}
uniformly in $x,t\in\reu$ with the norm $Ct^{-n(1-1/p)}$. Here $\Gamma_{x,t}$ is a cone with vertex at $(x,t/2)$ of aperture 1.
\end{lemma}

\bp Since by \eqref{eq8.05} the desired estimate \eqref{eq8.010} is valid for $\Gamma_{01}$ defined in the preceding Lemma, it is sufficient to estimate the difference. To underline dependence of $A_{01}$ on $x,t$ we shall write explicitly $x,t$ as super-indices for the corresponding matrices and operators. 

Pairing  $\chi_{\ree\setminus \Gamma_{x,t}}\nabla  (\Gamma^*_{1}-\Gamma^{*x,t}_{01}) (x,t,\cdot,\cdot)$ with $\Phi\in C_0^\infty(\reu)$ such that $\|\Phi\|_{\widetilde T^{p'}_1(\reu)}=1$ we observe that it is sufficient to obtain the estimate on the absolute value of 
\begin{multline}\label{eq8.012}
\left(\iint_{\reu} \chi_{\ree\setminus \Gamma_{x,t}}(z,s)\nabla_{z,s} (\Gamma^*_{1}-\Gamma^{*x,t}_{01}) (y,\tau,z,s) \,\Phi(z,s)\,\frac{dzds}{s}\right)\,\Big|_{(y,\tau)=(x,t)}\\[4pt]
=\left(\left( (L_{01}^{*x,t})^{-1}- (L_{1}^{*})^{-1} \right) \dv\, \Bigl(\chi_{\ree\setminus \Gamma_{x,t}} \frac{\Phi}{s}\Bigr)(y,\tau)\right) \,\Big|_{(y,\tau)=(x,t)}
\end{multline}
by $C t^{-n(1-1/p)}\|\Phi\|_{\widetilde T^{p'}_1(\reu)}$ (here the division by $s$ signifies the division of $\Phi=\Phi(y,s)$ by its vertical variable $s$). To this end, we write 
\begin{equation}\label{eq8.011}
(L^*_1)^{-1} \dv-(L^{*x,t}_{01})^{-1}\dv  =(L^{*x,t}_{01})^{-1} \dv (A^*_1-A^{*x,t}_{01}) \nabla (L^*_1)^{-1}\dv.
\end{equation}
This makes sense acting on any $L^2$ function, in particular, in our case.  Hence, the right-hand side of \eqref{eq8.012} is equal to 
\begin{equation}\label{eq8.013}
-\left(\left( (L^{*x,t}_{01})^{-1} \dv (A^*_1-A^{*x,t}_{01}) \nabla (L^*_1)^{-1}\dv\right) \, \Bigl(\chi_{\ree\setminus \Gamma_{x,t}} \frac{\Phi}{s}\Bigr)(y,\tau)\right) \,\Big|_{(y,\tau)=(x,t)}.
\end{equation}
Now observe that by construction $(A^*_1-A^{*x,t}_{01})= \chi_{\ree\setminus \Gamma_{x,t}}(A^*_1-A^{*x,t}_{01})$ and hence, by \eqref{eq8.05}, the absolute value of the expression in \eqref{eq8.013} is bounded by 
\begin{equation}\label{eq8.014}
C t^{-n(1-1/p)}\left\|\tau\,(A^*_1-A^{*x,t}_{01}) \nabla (L^*_1)^{-1}\dv  \, \Bigl(\chi_{\ree\setminus \Gamma_{x,t}} \frac{\Phi}{s}\Bigr)\right\|_{\widetilde T^{p'}_1(\reu)}.
\end{equation}
We are slightly abusing the notation here writing $\tau$ under the norm: it is understood that the norm corresponds to the integration in $y,\tau$. Recall now that by Lemma~\ref{l2.20} we have ${\bf T}^\infty_2 \cdot {T}^{p'}_2 \hookrightarrow \widetilde T^{p'}_1.$ Hence, the expression in \eqref{eq8.014} is, in turn, bounded by 
\begin{equation}\label{eq8.015}
C t^{-n(1-1/p)}\left\|\tau\,\nabla (L^*_1)^{-1}\dv  \, \Bigl(\chi_{\ree\setminus \Gamma_{x,t}} \frac{\Phi}{s}\Bigr)\right\|_{T^{p'}_2(\reu)}.
\end{equation}
Finally, by duality with \eqref{eq6.27}, the latter does not exceed $C t^{-n(1-1/p)} \|\Phi\|_{\widetilde T^{p'}_1(\reu)}$, as desired. \ep 

To guarantee uniqueness in Hardy spaces for the regularity problem, we shall 
impose some additional regularity at the boundary  for the solutions of an elliptic operator at hand. To this end, recall that an elliptic operator satisfies the H\"{o}lder continuity condition at the boundary (or the De Giorgi-Nash estimates at the boundary) if 
there exists a constant $\alpha_{0}>0$ such that for every solution to $L u=0$ in $R_{2Q}=2Q \times (0, 2 \ell(Q))$, with $u(x,0)=0$ in $2Q$ in the weak sense and $(x,t) \in R_Q$,  we have
\begin{equation}
|u(x,t)| \leq C (t/R)^{\alpha_{0}} \left( \frac{1}{|R_{2Q}|}\iint\limits_{R_{2Q}} |u|^2\right)^{1/2}.
\label{eq1.8}
\end{equation}
 This condition will allow us to use suitable pointwise estimates for the Green function associated with $L_0$ in order to obtain uniqueness. 
 
For operators with
complex coefficients, it is not clear whether interior H\"older regularity \eqref{eq1.6}
implies boundary H\"older regularity \eqref{eq1.8}.  However, it is known that
operators with complex coefficients that are
small $L^\infty$ perturbations of real coefficients, do enjoy both
\eqref{eq1.6} and \eqref{eq1.8} (see \cite{A}).

Let us state the following auxiliary lemma, whose proof can be found in \cite{KK}.
\begin{lemma}[\cite{KK}]\label{l8.1} Let $L$ be an elliptic operator such that the solutions to $L$ and $L^*$ satisfy both the interior De Giorgi-Nash-Moser bounds \eqref{eq1.6}--\eqref{eq1.7} and
the condition of H\"older continuity at the boundary \eqref{eq1.8}. Let $G(X,Y)$ be the Green function of $L$ in $\reu$ (see \cite{HK} for the detailed definition). Then for all $X,Y \in \reu$, with $X\neq Y$, we have
\begin{align}
|G(X,Y)| &\lesssim|X-Y|^{1-n},\label{eq8.2}\\
|G(X,Y)| &\lesssim \min\{\delta(X),|X-Y|\}^{\alpha_0} \min\{\delta(Y),|X-Y|\}^{\alpha_0} |X-Y|^{1-n-2\alpha_0},
\label{eq8.3}\end{align}
where $\delta(X)$ is the distance of $X$ to the boundary and $\alpha_0$ is the minimum of H\"older exponents in \eqref{eq1.6} and \eqref{eq1.8}.
\end{lemma}

\begin{proposition}\label{p8.3.1} Let $L$ be an elliptic operator falling under the scope of Theorems~\ref{t1.13}, \ref{t1.32} and
retain the significance of $p_0$, $\eps$, $\alpha_0$ from Theorem~\ref{t1.13}.
Then the following is true: 
\begin{itemize}
\item[(i)] Assume that $(R)_p$ is solvable for some $1<p<2+\eps$. If $u\in W^{1,2}_{loc}(\reu)$ is a solution of the Dirichlet problem $(D)_{p'}$, with zero boundary data in the sense that $L^*u=0$, $N_* (u)\in L^{p'}(\rn)$, and $u(\cdot, t)\to 0$ as $t\to 0$ strongly in $L^{p'}$ on compact subsets of $\rn$, then $u\equiv 0$. 

\item[(ii)] Assume that $(D)_{p'}$ is solvable for some $1<p<2+\eps$.
If $u\in W^{1,2}_{loc}(\reu)$ is a solution of the Regularity problem $(R)_{p}$, $1<p<2+\eps$ and $p<n$, with zero boundary data, that is, $Lu=0$, $\N(\nabla u)\in L^p(\rn)$, and $u(\cdot, t)\to 0$ as $t\to 0$ n.t., then $u\equiv 0$ (modulo constants).    

\item[(ii-a)] Assume that \eqref{eq1.8} is valid for $L^*$. Then the following is true. If $u\in W^{1,2}_{loc}(\reu)$ is a solution of the Regularity problem $(R)_{p}$, $p_0<p<2+\eps$, with zero boundary data, that is, $Lu=0$, $\N(\nabla u)\in L^p(\rn)$, and $u(\cdot, t)\to 0$ as $t\to 0$ n.t., then $u\equiv 0$ (modulo constants).  Here $p_0$ and $\eps$ depend on the corresponding parameter in the statement of Theorem~\ref{t1.13} and also on the boundary regularity exponent from \eqref{eq1.8}.



\item[(iii)] Assume  that $S_t^{L}|_{t=0}:H^p\to H^{1,p}$ and $\frac 12 I+\widetilde{\K}^{L}:H^{p}\to H^{p}$ are both invertible for some $p_0<p<2+\eps$ and that the solutions to the regularity problem $(R)_p$ are unique (e.g., the conditions of {\rm{(ii)}} or {\rm{(ii-a)}} are satisfied). Then the solution to the Neumann problem is unique. Specifically, if $u\in W^{1,2}_{loc}(\reu)$ is a solution of the Neumann problem $(N)_{p}$, $p_0<p<2+\eps$, with zero boundary data, that is, $Lu=0$, $\N(\nabla u)\in L^p(\rn)$, and $\partial_{\nu} u=0$ (with the normal derivative interpreted via Lemma~\ref{l7.1}), then $u\equiv 0$ (modulo constants).  
\end{itemize}
\end{proposition}


Let us point out that for $t$-indepependent operators uniqueness of solutions to $(D)_{\Lambda^0}$ was demonstrated in \cite{HMiMo}. However, at the moment we do not see how to extend the underlying argument to the current setting of operators satisfying the small Carleson measure condition. Uniqueness of $(D)_{\Lambda^\alpha}$, $\alpha>0$, again for $t$-independent operators and under the assumption of invertibility of layer potentials, was shown in \cite{BM}. 

\vskip 0.08 in

\noindent {\it Proof of Proposition~\ref{p8.3.1}}.  {\bf Step I}: Proof of (i), the set-up.

Let us start with the definition of the Green function. First note that by Lemma~\ref{l8.01} for every fixed $(x,t)\in \reu$ we have  $\nabla_\| \Gamma^*(x,t, \cdot, 0)\in L^{p}(\rn),$ with the norm $Ct^{-n/p'}$. By our assumptions the regularity problem is solvable, and hence, there exists a solution to the following $(R)_p$ problem:
\begin{equation}
\begin{cases} 
Lw_{x,t}=0,\;\text{ in }\mathbb{R}_{+}^{n+1},\\ 
w_{x,t}(\cdot, s) \to {\overline{\Gamma^*(x,t, \cdot, 0)}} \text{ n.t. as $s\to 0$},\\
\|\N(\nabla w_{x,t})\|_{L^p}\lesssim t^{-n/p'}.
\end{cases}
\label{eq8.3.2}\end{equation}
Now define 
$$G(x,t, y,s):=\overline{\Gamma^*(x,t,y,s)}-w_{x,t}(y,s), \quad (x,t),(y,s)\in \reu.$$
Then, in particular, for $(y,s)\in \reu$ such that $s\leq t/4+|x-y|/2$ (that is, all $(y,s)$ which stay away from $\Gamma_{x,t}$ together with their Whitney cubes) we have
\begin{equation}\label{eq8.3.3-bis2}
|G(x,t, y,s)| \lesssim C s\, \left(\N(\chi_{\reu\setminus \Gamma_{x,t}} \nabla \Gamma^*(x,t,\cdot,\cdot))(y) +\N(\chi_{\reu\setminus \Gamma_{x,t}}\nabla w_{x,t})(y)\right), \quad y\in\rn.\end{equation}
This inequality follows (much as Lemma~\ref{l7.1}, \eqref{eq7.1.1}) from carefully tracking the argument in \cite{KP}, pp. 461--462. Furthermore, then
\begin{multline}\label{eq8.3.3}
\|G(x,t, \cdot,s)\chi_{y\in\rn:\,|x-y| \geq 2s-t/2}\|_{L^p(\rn)} \\[4pt]
\leq C s\,\| \N(\chi_{\reu\setminus \Gamma_{x,t}} \nabla \Gamma^*(x,t,\cdot,\cdot)) +\N(\chi_{\reu\setminus \Gamma_{x,t}} \nabla w_{x,t})\|_{L^p(\rn)}
\leq C s t^{-n/p'}.
\end{multline}
due to Lemma~\ref{l8.09} and \eqref{eq8.3.2}. 


Assume that $u\in W^{1,2}_{loc}(\reu)$ is a solution of $L^*u=0$ in $\reu$. Fix a point $X=(x,t) \in \reu$. Then for every $\phi\in C_0^\infty(\reu)$ such that $\phi\equiv 1$ in a neighborhood of $X$ we have 
\begin{align}
&u(x,t) =u(x,t)\phi(x,t)=\iint_{\reu} \overline{A \nabla G(X,Y)}\, \nabla \left(u\phi\right)(Y)dY=\notag\\
&=\iint_{\reu} \overline{A \nabla G(X,Y)} \,\nabla u(Y) \phi(Y)dY +\iint_{\reu} \overline{A \nabla G(X,Y)}\,u(Y) \nabla \phi (Y)dY\notag\\
&=-\iint_{\reu} \overline{G(X,Y)} \,A^* \nabla u(Y) \nabla\phi (Y)dY+\iint_{\reu} \overline{A \nabla G(X,Y)}\,u(Y) \nabla \phi (Y)dY\notag\\
&=:I+II,\label{eq8.4}
\end{align}
where  $G$ is the Green function defined as above.
Now choose $\eta\in C_0^\infty(-2,2)$ such that $\eta=1$ in $(-1,1)$ with $0\leq\eta\leq 1$, and set 
$$\phi(y,s):=(1-\eta (s/\epsilon))\,\eta (s/(100 \rho)) \,\eta(|x-y|/\rho), $$
where $\epsilon<t/8$ and $\rho>8t$. Then the domain of integration of $I$ and $II$ in \eqref{eq8.4} is contained in the union of 
\begin{itemize}
\item[(1)] $\Omega_1\subset \Delta_{2\rho}(x)\times \{\epsilon<s<2\epsilon\}$, with $\|\nabla \phi\|_{L^\infty(\Omega_1)}\lesssim \epsilon^{-1}$.\\
\item[(2)] $\Omega_{2,1}\subset \Delta_{2\rho}(x)\times \{100\rho<s<200\rho\}$, with $\|\nabla \phi\|_{L^\infty(\Omega_{2,1})}\lesssim \rho^{-1}$.\\
\item[(3)] $\Omega_{2,2}\subset (\Delta_{2\rho}(x)\setminus \Delta_{\rho}(x))\times \{0<s<200\rho\}$, with $\|\nabla \phi\|_{L^\infty(\Omega_{2,1})}\lesssim \rho^{-1}$.
\end{itemize}
We shall denote $\Omega_2:=\Omega_{2,1}\cap \Omega_{2,2}$, and split $I=I_{\Omega_1}+I_{\Omega_2}$ and $II=II_{\Omega_1}+II_{\Omega_2}$ according to the domains of integration.

\vskip 0.08 in \noindent {\bf Step II}: the proof of (i). 

Applying Caccioppoli inequality, we have 
$$ I_{\Omega_1}+II_{\Omega_1} \lesssim  \frac{1}{\epsilon}\iint_{ \Omega_1} (\overline{G(x,t, \cdot, \cdot)})_W(y,s)\, (u)_W(y,s)\, \frac{dyds}{s}.$$
As before,  $(\cdot)_W$ stands for the averaging over Whitney cubes \eqref{eq2.12.1}. Note that both $G$ and $u$ are solutions in the $c\epsilon$-neighborhood of the domain of integration and thus, $L^2$ averages denoted by the subscript $W$ can be substituted by any $L^p$ average as convenient using Moser estimates. 
Using this observation and \eqref{eq8.3.3}, one can see that the latter expression is bounded by 
\begin{multline*}
C\,\sup_{\epsilon/8<s<8\epsilon} \int_{\Delta_{C\rho}}|u(y,s)| \, \Bigl(\N(\chi_{\reu\setminus \Gamma_{x,t}} \nabla \Gamma^*(x,t,\cdot,\cdot))(y) +\N(\chi_{\reu\setminus \Gamma_{x,t}} \nabla w_{x,t})(y)\Bigr) \,dy\\[4pt]
\lesssim \sup_{\epsilon/8<s<8\epsilon}\left( \int_{\Delta_{C\rho}}|u(y,s)|^{p'}\,dy\right)^{1/p'} \, \left( \int_{\Delta_{C\rho}} \Bigl(\N(\chi_{\reu\setminus \Gamma_{x,t}} \nabla \Gamma^*(x,t,\cdot,\cdot))(y) +\N(\chi_{\reu\setminus \Gamma_{x,t}} \nabla w_{x,t})(y)\Bigr)^p \,dy\right)^{1/p}\\[4pt]
\leq C_t \sup_{\epsilon/8<s<8\epsilon}\left( \int_{\Delta_{C\rho}}|u(y,s)|^{p'}\,dy\right)^{1/p'}.\end{multline*}
The latter vanishes as $\epsilon \to 0$ since by assumptions $u$ converges to zero strongly in $L^{p'}$ on compact subsets of $\rn$. Going further,

\begin{multline*}
I_{\Omega_2}+II_{\Omega_2} \lesssim \frac{1}{\rho}\iint_{\Omega_2} (\overline{G(x,t, \cdot, \cdot)})_W(y,s)\, (u)_W(y,s)\, \frac{dyds}{s}\lesssim \frac{1}{\rho}\iint_{\widetilde\Omega_2} |G(x,t, y,s)|\, |u(y,s)|\, \frac{dyds}{s}\\[4pt]
\lesssim \frac{1}{\rho}\iint_{(y,s)\in\widetilde\Omega_2:\, s\leq t/4+|x-y|/2} |G(x,t, y,s)|\, |u(y,s)|\, \frac{dyds}{s}\\[4pt]
+\frac{1}{\rho}\iint_{(y,s)\in\widetilde\Omega_2:\, s> t/4+|x-y|/2} |G(x,t, y,s)|\, |u(y,s)|\, \frac{dyds}{s}=:J_1+J_2.
\end{multline*}
Here $\widetilde\Omega_2$ denotes a slight enlargement of $\Omega_2$ obtained by including in the set the Whitney cubes of all points of $\Omega_2$ and then  fattening those Whitney cubes by a small fixed factor of their size.  Properly adjusting the constants, we make sure that  the distance from Whitney cubes centered at points of $\widetilde\Omega_2$ to $(x,t)$ is still comparable to $\rho$. 

The region of integration in $J_1$ is chosen to allow one to employ \eqref{eq8.3.3}. By Lemma~\ref{l2.31}, $u\in L^{\frac{p'(n+1)}{n}}(\reu)$ and hence, $\|u\|_{L^{\frac{p'(n+1)}{n}}(\widetilde\Omega_2)}\to 0$ as $\rho \to \infty.$ Hence, $J_1\to 0$ as $\rho\to \infty$ if we show that 
$$\frac{1}{\rho}\left(\iint_{(y,s)\in\widetilde\Omega_2:\, s\leq t/4+|x-y|/2} |G(x,t, y,s)|^q\, \frac{dyds}{s^q}\right)^{1/q}\leq C_t,$$
with $q$ being a dual exponent of $\frac{p'(n+1)}{n}$, that is, $q=\frac{p(n+1)}{n+p}.$ Note that $1<q<p$. Using H\"older inequality to pass from the $L^q$ to $L^p$ norm  in $y$ and then \eqref{eq8.3.3}, we indeed have 
$$ \frac{1}{\rho}\left(\iint_{(y,s)\in\widetilde\Omega_2:\, s\leq t/4+|x-y|/2} |G(x,t, y,s)|^q\, \frac{dyds}{s^q}\right)^{1/q} \lesssim \frac{1}{\rho}\left(\int_0^{C\rho} \left(st^{-n/p'}\right)^{q} \rho^{n(1-q/p)}\,\frac{ds}{s^q}\right)^{1/q} \leq C t^{-n/p'},
$$
as desired.

The argument for $J_2$ is essentially the same once we observe that the region of integration in $J_2$ is included in $\Delta_{C\rho}\times({C_1\rho, C_2\rho})$ staying at a distance proportional to $\rho$ from $(x,t)$. Moreover, every point $(y.s) \in \Delta_{C\rho}\times({C_1\rho, C_2\rho})$ belongs to every cone in $\reu$ with a vertex at $(z,0)$ such that $z\in \Delta_{C_3\rho}\setminus \Delta_{C_4\rho}$ truncated at a height $C_5\rho$ (with a proper choice of constants and the aperture), and we can make sure that every such truncated cone stays at a distance proportional to $\rho$ from $(x,t)$. Then, once again invoking the argument in \cite{KP}, pp. 461--462, we have for all $(y.s) \in \Delta_{C\rho}\times({C_1\rho, C_2\rho})$ 
$$|G(x,t,y,s)|\lesssim s \left(\N(\nabla \Gamma^*(x,t,\cdot,\cdot))(z)+\N(\nabla w_{x,t})(z)\right)\lesssim \rho^{1-n}+s\N(\nabla w_{x,t})(z), $$
for any $z\in \Delta_{C_3\rho}\setminus \Delta_{C_4\rho}$. Here, the second estimate uses Caccioppoli inequality and pointwise estimates on the fundamental solution. Hence, we can show that by \eqref{eq8.3.2}
$$\sup_{C_1\rho<s<C_2\rho}\|G(x,t,\cdot,s)\|_{L^p(\Delta_{C\rho})}\lesssim  \rho^{1-n/p'}+\rho t^{-n/p'},$$ and with this at hand the same argument as for $J_2$ applies. 

All in all, taking first the limit as $\epsilon\to 0$ and then $\rho \to\infty$, we conclude that $u\equiv 0$ in $\reu$.

\vskip 0.08 in \noindent {\bf Step III}: the proof of (ii).

We shall use the set-up of Step I, with the Green function defined using the solvability of the Dirichlet problem. Indeed, $\|\Gamma(x,t,\cdot, 0)\|_{L^{p'}} \lesssim t^{1-n/p}$ and thus, by assumptions, there  there exists a solution to the following $(D)_{p'}$ problem:
\begin{equation}
\begin{cases} 
L^*w_{x,t}=0,\;\text{ in }\mathbb{R}_{+}^{n+1},\\ 
w_{x,t}(\cdot, s) \to \Gamma (x,t, \cdot, 0)\text{ as  $s\to 0$ strongly in $L^{p'}$ on compacta in $\rn$,}\\
\|N_*(w_{x,t})\|_{L^{p'}}\lesssim t^{1-n/p}.
\end{cases}
\label{eq8.3.2-reg}\end{equation}
Now define 
\begin{equation}\label{regun0}G(x,t, y,s):={\overline{\Gamma(x,t,y,s)}}-w_{x,t}(y,s), \quad (x,t),(y,s)\in \reu.\end{equation}

Then, exactly as above, 
\begin{multline}\label{regun1}I_{\Omega_1}+II_{\Omega_1} \lesssim   \frac{1}{\epsilon}\iint_{\widetilde\Omega_1} |G(x,t, y,s)|\, |u(y,s)|\, \frac{dyds}{s}\\[4pt]
\lesssim   \frac{1}{\epsilon}\iint_{\widetilde\Omega_1} |G(x,t, y,s)|\, s\N(\nabla u)(y)\, \frac{dyds}{s}
\lesssim   \sup_{\epsilon/8<s<8\epsilon}\left(\int_{\Delta_{C\rho}} |G(x,t, y,s)|^{p'}\, dy\right)^{1/p'}, \end{multline}
using the fact that $u$ is a solution to $(R)_p$ with zero boundary data and Lemma~\ref{l7.1}. Since by assumptions $w_{x,t}$ converges strongly in $L^{p'}(\Delta_{C\rho})$ to its boundary data and the same can be directly checked for $\Gamma$ (using, e.g., interior H\"older regularity of solutions), the Green function converges to zero  strongly in $L^{p'}(\Delta_{C\rho})$, and hence, the right-hand side of the expression above vanishes as $\epsilon\to 0$.

As for $\Omega_2$, we recall that by Lemma~\ref{l2.31} $\|w_{x,t}\|_{L^q(\reu)}\leq Ct^{-n/p+1}$, $q=p'(n+1)/n$ and therefore,  for a fixed $(x,t)$ we have $\|w_{x,t}\|_{L^q(\widetilde\Omega_2)}\to 0$, as $\rho\to\infty$. Moreover, by pointwise estimates on the fundamental solution $\|\Gamma(x,t,\cdot,\cdot)\|_{L^{p'(n+1)/n}(\widetilde \Omega_2)}\leq C\rho^{-n/p+1}$. Hence, for a fixed $(x,t)$
$$\|G(x,t,\cdot,\cdot)\|_{L^q(\widetilde\Omega_2)}\to 0, \quad\mbox{as }\rho\to\infty, \quad q=p'(n+1)/n,$$
provided that $p<n$. 
Note that the dual exponent $q'=\frac{p'(n+1)}{p'(n+1)-n}$ so that $1<q'<p$. Combing these considerations, we have 
\begin{multline}\label{regun2}I_{\Omega_2}+II_{\Omega_2} \lesssim   \frac{1}{\rho}\iint_{\widetilde\Omega_2} |G(x,t, y,s)|\, |u(y,s)|\, \frac{dyds}{s}\\[4pt]
\lesssim   \frac{1}{\rho}\,\|G(x,t,\cdot,\cdot)\|_{L^q(\widetilde\Omega_2)} \left(
\iint_{\widetilde\Omega_2} |u(y,s)|^{q'}\, \frac{dyds}{s^{q'}}\right)^{1/q'} \\[4pt]
\lesssim \rho^{-1+1/q'}\|G(x,t,\cdot,\cdot)\|_{L^q(\widetilde\Omega_2)} \left(
\int_{\Delta_{C\rho}} |\N(\nabla u)|^{q'}\, dy\right)^{1/q'}\\[4pt]
\lesssim \rho^{-1+(n+1)/q'-n/p}\|G(x,t,\cdot,\cdot)\|_{L^q(\widetilde\Omega_2)}\left(
\int_{\Delta_{C\rho}} |\N(\nabla u)|^{p}\, dy\right)^{1/p}\\[4pt]\lesssim \|G(x,t,\cdot,\cdot)\|_{L^q(\widetilde\Omega_2)}\left(
\int_{\Delta_{C\rho}} |\N(\nabla u)|^{p}\, dy\right)^{1/p},
\end{multline}
which by the aforementioned considerations vanishes as $\rho\to\infty$ if $p<n$, as desired.

\vskip 0.08 in \noindent {\bf Step IV}: the proof of (ii-a).

If the solutions satisfy the condition of H\"older continuity at the boundary \eqref{eq1.8}, the argument is considerably simpler and does not require the condition $1<p<n$. To be precise, it is valid for all $p_0< p<2+\eps$ where however $p_0<1$ and $\eps>0$ depend, in addition, on the exponent of boundary H\"older regularity (denoted by $\alpha_0$ in \eqref{eq1.8}). Throughout this argument we denote by $\alpha_0$ minimum of the exponents of the interior and boundary H\"older regularity. 

Under the assumption of De Giorgi-Nash-Moser estimates at the boundary \eqref{eq1.8} we can directly use the Green function constructed in \cite{HK}, \cite{KK} (see Lemma~\ref{l8.1}). It is straightforward to check that for such a Green function \eqref{eq8.4} is justified. Let us start with $p\geq 1$. Then we have once again
$$ I_{\Omega_2}+II_{\Omega_2} \lesssim \frac{1}{\rho}\iint_{\Omega_2} (\overline{G(x,t, \cdot, \cdot)})_W(y,s)\, (u)_W(y,s)\, \frac{dyds}{s}.$$
However, given that now $u$ is a solution to $(R)_p$ with zero boundary data, we can use Lemma~\ref{l8.1} to majorize the integral above by 
$$
C\, \frac{t^{\alpha_0}}{\rho}\iint_{\widetilde\Omega_2} s^{\alpha_0}\, \rho^{1-n-2\alpha_0}\,\N(\nabla u)\, dyds 
\lesssim t^{\alpha_0} \rho^{1-\alpha_0-n/p} \|\N(\nabla u)\|_{L^p} \to 0, \quad\mbox{as }\rho\to\infty.
$$
At the same time, for the same Green function constructed in \cite{HK}, \cite{KK}, 
\begin{multline*}
I_{\Omega_1}+II_{\Omega_1} \lesssim  \frac{1}{\epsilon}\iint_{\Omega_1} (\overline{G(x,t, \cdot, \cdot)})_W(y,s)\, (u)_W(y,s)\, \frac{dyds}{s}\\[4pt]
\lesssim \frac{t^{\alpha_0}}{\epsilon}\iint_{\widetilde\Omega_1} s^{\alpha_0}\, t^{1-n-2\alpha_0}\,\N(\nabla u)\, dyds \lesssim C_{t, \rho}\, \epsilon^{\alpha_0}  \|\N(\nabla u)\|_{L^p} \to 0, \quad\mbox{as }\epsilon\to 0.
\end{multline*}

Now, as above, taking first the limit as $\epsilon\to 0$ and then $\rho \to\infty$, we conclude that $u\equiv 0$ in $\reu$.

As for the case $p<1$, we recall that by Lemma~\ref{l2.31} we have $\|\nabla u\|_{L^{q}}\leq C$ with $q=\frac{p(n+1)}{n}>1$. Moreover, since $u$ has non-tangential trace zero at the boundary, $u(y,s)=\int_0^s \partial_\tau u(y,\tau)\,d\tau$ a.e. in $\reu$. This is, in fact, a part of the proof of \eqref{eq7.1.1}.  
Then  
\begin{multline*} I_{\Omega_2}+II_{\Omega_2} \lesssim \frac{1}{\rho}\iint_{\widetilde \Omega_2} |\overline{G(x,t, y,s)}|\, |u(y,s)|\, \frac{dyds}{s}\\[4pt]
\lesssim t^{\alpha_0}\rho^{-n-2\alpha_0}\iint_{\widetilde \Omega_2} s^{\alpha_0}\, |u(y,s)|\, \frac{dyds}{s} \lesssim t^{\alpha_0}\rho^{-n-\alpha_0}\iint_{\Delta_{C\rho}\times(0,C\rho)} |\partial_\tau u(y,\tau)|\, dyd\tau 
\\[4pt]\lesssim t^{\alpha_0}\rho^{1-n/p-\alpha_0}\left(\iint_{\Delta_{C\rho}\times(0,C\rho)} |\partial_\tau u(y,\tau)|^q\, dyd\tau\right)^{1/q}, \end{multline*}
vanishes as $\rho\to\infty$ provided that $1-n/p-\alpha_0>0$.
On the other hand, 
\begin{multline*}
I_{\Omega_1}+II_{\Omega_1} \lesssim  \frac{1}{\epsilon}\iint_{\widetilde\Omega_1} |\overline{G(x,t, y,s)}|\, |u(y,s)|\, \frac{dyds}{s}\\[4pt]
\lesssim  \epsilon^{\alpha_0-1}t^{\alpha_0} t^{1-n-2\alpha_0} \iint_{\Delta_{C\rho}\times(0,C\epsilon)} |\partial_\tau u(y,\tau)|\, dyd\tau.
\end{multline*}

Now recall that we have 
$$ \left(\iint_{\Delta_{C\rho}\times(0,C\epsilon)} |\partial_\tau u(y,\tau)|^{\frac{p(n+1)}{n}}\, dyd\tau\right)^{\frac{n}{p(n+1)}}\leq C,$$
and simultaneously, 
$$ \left(\iint_{\Delta_{C\rho}\times(0,C\epsilon)} |\partial_\tau u(y,\tau)|^p\, dyd\tau\right)^{\frac{1}{p}}\lesssim \epsilon^{1/p} 
\left(\int_{\Delta_{C\rho}} |\N(\partial_\tau u(y,\tau))|^p\, dyd\tau\right)^{\frac{1}{p}}
\lesssim \epsilon^{1/p}.$$
Then the usual interpolation inequalities (which in this case consist of an application of the H\"older's inequality to $|\partial_\tau u|^p$) yield 
$$ \left(\iint_{\Delta_{C\rho}\times(0,C\epsilon)} |\partial_\tau u(y,\tau)|^{p_\theta}\, dyd\tau\right)^{p_{\theta}}\leq C \epsilon^{\frac{1-\theta}{p}},$$
for any $0<\theta<1$ and $\frac{1}{p_\theta}=\frac{1-\theta}p +\frac{\theta n}{p(n+1)}$. Choosing $p_\theta=1$ (note that $p<1<\frac{p(n+1)}{n}$ so that $p_\theta=1$ is a valid choice) we compute that $\theta=(n+1)(1-p)$ and, respectively, $\frac{1-\theta}{p}=\frac{1-(n+1)(1-p)}{p}$. All in all, we have 
$$ \iint_{\Delta_{C\rho}\times(0,C\epsilon)} |\partial_\tau u(y,\tau)|\, dyd\tau\leq C \epsilon^{\frac{1-(n+1)(1-p)}{p}}.$$

Returning to the estimates on $I_{\Omega_1}+II_{\Omega_1}$, we then have
$$I_{\Omega_1}+II_{\Omega_1} \lesssim    \epsilon^{\alpha_0-1+\frac{1-(n+1)(1-p)}{p}} t^{1-n-\alpha_0}=\epsilon^{\alpha_0-n/p+n} t^{1-n-\alpha_0},$$
which vanishes as $\epsilon\to 0$ provided that $n(1/p-1)<\alpha_0$.

\vskip 0.08 in \noindent {\bf Step V}: the proof of (iii).

Suppose that $\N(\nabla u) \in L^p$, $p_0<p<2+\eps$, and that $\partial_\nu u=0$, interpreted in the weak sense of Lemma~\ref{l7.1}. Note that Lemma~\ref{l7.1} also yields existence of $u_0 \in H^{1,p}$ such that $u(x,t) \xrightarrow{t \to 0} u_0$, n.t.. By the uniqueness of solutions to $(R)_{H^p}$ established in Step IV,
$$u(\cdot ,t)= \SL(\SL^{-1}_t\!\mid_{t=0}u_0).$$
Thus, by Proposition~\ref{p7.15}
$$0=\partial_{\nu} u=\left(\frac{1}{2}I+ \widetilde{\K} \right)(\SL^{-1}_t\!\mid_{t=0} u_0).$$
By hypothesis, $\frac{1}{2}I+ \widetilde{\K} :H^p \to H^p$ and $\SL_t\!\mid_{t=0}:H^p \to H^{1,p}$ are bijective, so that $u_0 \equiv 0$ in the sense of $H^{1,p}$. Employing again the uniqueness of solutions to $(R)_{H^p}$ established in Step IV, we conclude that $u \equiv 0$, modulo constants, and thus $(N)_{H^p}$ is uniquely solvable.

\ep

\end{document}